%% file: main.tex
\documentclass[11pt]{article}

\usepackage{stmaryrd}
\usepackage{appendix}

\usepackage{subfig}
\usepackage[normalem]{ulem}

\usepackage{amsfonts}
\usepackage{amsthm,amssymb,amsmath}
\usepackage{latexsym}
\usepackage{graphics}
\usepackage{epic}
\usepackage{epsfig}
\usepackage{romanbar}
\usepackage{psfrag}
\usepackage{pict2e}  
\usepackage{bbm}
\usepackage{dsfont}
\usepackage{color}
\usepackage{tikz}
\usepackage{authblk}
\newtheorem{thm}{Theorem}
\newtheorem{lem}{Lemma}
\newtheorem{cor}{Corollary}
\newtheorem{dfn}{Definition}
\newtheorem{prop}{Proposition}
\usetikzlibrary{decorations.pathreplacing}

\usepackage{geometry}
\newcommand{\ee}{\varepsilon}

\input DCMCommands

\newcommand{\Perco}{\textrm{Perco}}
\newcommand{\ncc}{\mathcal{K}}
\renewcommand{\deg}{\textrm{deg}}
\newcommand{\Star}{\textrm{Star}}

\begin{document}

\title{Line-of-sight Cox percolation on Poisson-Delaunay triangulation} 
\date{}
\author[1,2]{David Corlin Marchand\footnote{david.corlin.marchand@gmail.com}}
\author[1]{David Coupier\footnote{david.coupier@imt-nord-europe.fr}}
\author[1]{Beno\^{i}t Henry\footnote{benoit.henry@imt-nord-europe.fr}}

\affil[1]{IMT Nord Europe, Institut Mines-T\'el\'ecom, Univ.\ Lille, F-59000 Lille, France; }
\affil[2]{Laboratoire de Math\'ematiques Rapha\"{e}l Salem, Universit\'e de Rouen Normandie, avenue de
l’Universit\'e, BP 12, Technop\^{o}le du Madrillet, 76801 Saint-\'Etienne-du-Rouvray, France.}

\maketitle

\begin{abstract}
In this work, percolation properties of device-to-device (D2D) networks in urban environments are investigated. The street system is modeled by a Poisson-Delaunay triangulation (PDT). Users are of two types: given either by a Cox process supported by the edges of the PDT or by a Bernoulli process on the vertices of the PDT (i.e. on streets and at crossroads). Percolation of the resulting connectivity graph $\bdcalG_{p,\lambda,r}$ is interpreted as long-range connection in the D2D network. According to the parameters $p,\lambda,r$ of the model, we state several percolation regimes in Theorem \ref{thm:mainResult} (see also Fig. \ref{fig:NotreDiag}). This work completes and specifies results of Le Gall et al \cite{le2021continuum}. To do it, we take advantage of a percolation tool, inspired by enhancement techniques, used to our knowledge for the first time in the context of communication networks.
\end{abstract}

\bigbreak
\textbf{AMS 2010 Subject Classification:}\\
\emph{Primary:} 60K35; 60G55; 60D05\\ \emph{Secondary:} 68M10; 90B15
\bigbreak
\textbf{Acknowledgments.}The authors warmly thank Q. Le Gall, B. Blaszczyszyn and E.
Cali for exciting discussions which were the starting point of this work, as well as D. Dereudre for valuable discussions on the topic. This work has been carried in the context of the project Beyond5G, funded by
the French government as part of the economic recovery plan, namely “France
Relance” and the investments for the future program. D. Corlin Marchand was partly funded by the ANR GrHyDy ANR-20-CE40-0002. Also, this work has been partially supported by the RT GeoSto 3477.

\bigbreak

\newpage

\section{Introduction}

\subsection{State of the art}

Due to the recent explosion of the number of connected devices and the emergence of new and high data rate services (as for instance video sharing, online gaming and internet browsing), cellular networks, i.e. telecommunication networks where the link between devices is wireless, have to be reconsidered. For this purpose, one of the main technologies investigated in the literature for about twenty years is \textbf{device-to-device} (D2D): see \cite{AWM} for a survey. D2D communication in cellular networks is defined as a direct and short-range communication between two mobile users without the need for the signal to be rooted through extra network structure (as a base station). Hence, a D2D communication between two devices far away from each other remains possible via a chain of consecutive D2D links where some devices located between the transmitter and the receiver intervene and relay the signal. This is why long-range connection in D2D networks can be naturally interpreted as a percolation problem.

The basic model to study D2D networks is the Poisson Boolean model or Gilbert's model in which random discs are scattered uniformly and independently in the plane $\mathbb{R}^2$. See \cite{Gilbert} for the seminal article of Gilbert and \cite{meester1996continuum} for a modern reference on continuum percolation theory. Later, more realistic approaches for D2D networks have been developed. Let us mention the SINR graph \cite{DFMMT} taking into account the interferences creating by devices located at proximity to the transmitter-receiver pair. Besides, in real-world networks in urban areas, devices-- we will also talk about users --cannot be located anywhere but rather on a street system. Let us notice that the use of random tesselations to model street systems has already been considered and validated \cite{GVS,VGFS}. A doubly stochastic Poisson Boolean model whose centers of discs are supported by a random tesselation is a \textbf{Cox process}. Percolation of such Cox processes have been recently investigated in \cite{hirsch2019continuum,jahnel2022phase} and especially in the context of D2D networks \cite{le2020crowd,le2021continuum}.\\

These last references, corresponding to the Le Gall thesis work, constitutes the starting point of the present study. Our goal is to complete and specify the percolation regimes stated in \cite{le2021continuum}.

\subsection{Our model: the connectivity graph $\bdcalG_{p,\lambda,r}$}
\label{sect:OurModel}

\noindent
\textbf{Urban media.} Let us consider a marked Poisson Point Process (PPP) $\bar{\bfX}$ on $\mathbb{R}^2$ with intensity $\textrm{dx}$ and independent marks in $[0,1]\times\mathbb{R}_+^{\mathbb{N}}$ with distribution $\mathcal{U}([0,1])\!\otimes\!( \mathcal{E}xp(1) )^{\otimes \mathbb{N}}$, where $\textrm{dx}$, $\mathcal{U}([0,1])$ and $\mathcal{E}xp(1)$ resp. denote the Lebesgue measure on $\mathbb{R}^2$, the uniform distribution on $[0,1]$ and the exponential distribution with rate $1$. We refer to the book by Last and Penrose \cite{last2017lectures} for background on Poisson point processes. Let $\bfX$ be the projection of $\bar{\bfX}$ onto its first ordinate $\mathbb{R}^2$. We can thus write
\[
\bar{\bfX} = \big\{ \big( x, V_x, (\mathcal{E}_{x,k})_{k\geq 1} \big) : x \in \bfX \big\}
\]
where $(V_x, (\mathcal{E}_{x,k})_{k\geq 1})_x$ is a family of elements of $[0,1]\times\mathbb{R}_+^{\mathbb{N}}$. The process $\bfX$ is a homogeneous PPP on $\mathbb{R}^2$ with intensity $1$ (w.r.t. the Lebesgue measure $\textrm{dx}$).

Let $\bfT$ be the Poisson Delaunay Triangulation (PDT) built from the process $\bfX$. Precisely, $\bfT$ is the (undirected) graph whose vertex set is given by the points of $\bfX$ and whose edge set is defined as follows: for $x,x' \in \bfX$, $\{x,x'\}$ is an edge of $\bfT$ if there exists a circle going through $x$ and $x'$ and having no points of $\bfX$ in its interior. With an abuse of notation, we also denote by $\bfT$ the subset of $\bbR^2$ defined as the union of segments $[x,x']$ where $\{x,x'\}$ are edges of the PDT.

The Delaunay triangulation $\bfT$ modelizes the \textit{urban media} (or the random environment) on which lies the D2D network. Edges and vertices (i.e. elements of $\bfX$) of $\bfT$ are resp. interpreted as streets and crossroads.\\

\noindent
\textbf{Users.} There are two kinds of users. For $p \in [0,1]$, we set
\[
\bfX_p := \{ x \in \bfX : V_x < p \}
\]
which can be viewed as a Bernoulli site percolation process with parameter $p$ performed on $\bfX$. When $V_x < p$, we say that there is a \textit{user at crossroad} $x$ or that the vertex $x \in \bfX$ is \textit{open} (and \textit{closed} in the alternative case).

Let $\bar{\bfY}$ be a marked Cox process on $\bbR^2$ with intensity measure $\lambda \Lambda$ and independent marks in $\bbR_+$ with distribution $\mathcal{E}xp(1)$, where $\lambda$ is a positive parameter and $\Lambda$ denotes the one-dimensional Hausdorff (random) measure associated to the Delaunay triangulation $\bfT$. Let $\bfY$ be the projection of $\bar{\bfY}$ onto its first ordinate $\mathbb{R}^2$:
\[
\bar{\bfY} = \big\{ (y, \mathcal{E}_y) : y \in \bfY \big\} ~.
\]
This means that conditionally on $\bfX$, the point process $\bfY$ is a marked PPP with intensity $\lambda$ on $\bfT$. I.e. the number of points of $\bfY$ in a given segment of $\bfT$ with length $\ell$ is a Poisson r.v. with mean $\lambda \ell$. The process $\bfY$ represents \textit{users on streets}.

Furthermore we assume that, conditionally on $\bfX$, the process $\bar{\bfY}$ is independent from the collection $(V_x, (\mathcal{E}_{x,k})_{k\geq 1})_{x \in \bfX}$.\\

\noindent
\textbf{Random connection radii.} Let $\|\cdot\|$ be the Euclidean distance. In order to take into account physical obstacles in the urban media (as buildings), we forbid connections between users on different streets, i.e. only \textit{line-of-sight} (LOS) connections are considered here. Hence, two users (on streets) $y,y' \in \bfY$ are \textit{connected} iff they belong to the same street and
\[
\| y - y' \| \leq \frac{r}{2} \big( \mathcal{E}_y + \mathcal{E}_{y'} \big) ~,
\]
where $r$ is a positive parameter called the \textit{connection radius}.

Due to the LOS constraint, users at crossroads (i.e. elements of $\bfX_p$) then appear to be crucial to ensure the signal propagation from street to street. This is why we also think about them as \textit{relays}. Their presence can also be interpreted as scattering and reflection effects occurring in LOS communications. Let $x \in \bfX_p$ be one of them. It deploys independent connection ranges
\[
\frac{r'}{2} \mathcal{E}_{x,1} , \ldots , \frac{r'}{2} \mathcal{E}_{x,\deg(x)}
\]
along its incident edges, where $\deg(x)$ is the degree of $x$ in $\bfT$ and $r' > 0$ is a constant parameter. See the top side of Fig. \ref{fig:connection}. Precisely, we label edges incident to $x$ from $1$ to $\deg(x)$ in the counterclockwise sense and starting from the semi-line $[x,x+(1,0))$. Any user $y \in \bfY$ located on the $i$-th incident edge to $x$, with $i\leq \deg(x)$, is \textit{connected} to $x$ iff
\[
\| x - y \| \leq \frac{r'}{2} \mathcal{E}_{x,i} + \frac{r}{2} \mathcal{E}_{y} ~,
\]
Let $x' \in \bfX$ be the neighbor of $x$ in $\bfT$ linked by the $i$-th incident edge to $x$. Assume that $x' \in \bfX_p$ and say that $\{x,x'\}$ is the $j$-th incident edge to $x'$, with $j \leq \deg(x')$. In this case, $x,x'$ are \textit{connected} iff
\[
\| x - x' \| \leq \frac{r'}{2} \big( \mathcal{E}_{x,i} + \mathcal{E}_{x',j} \big) ~.
\]
In order to simplify the exposition, we fix from now on
\[
\fbox{$r' = 2r$}
\]
but we claim that all our results hold for any parameter $r' \geq 2r$. See discussion in Section \ref{sec:finaldiscussion}.\\

When two users $u,u' \in \bfX_p \cup \bfY$ are connected (whatever their types), we write $u \sim u'$. This necessarily means that they are on the same street (including both crossroads). The \textit{connectivity graph} $\bdcalG_{p,\lambda,r}$ is the undirected graph $(\bfX_p \cup \bfY , E)$ whose edge set $E$ is made up with pairs $\{u,u'\}$ of connected users. As before, we still use the symbol $\bdcalG_{p,\lambda,r}$ to refer to the subset of $\bfT$ defined as the union of segments $[u,u']$ where $u \sim u'$. See the bottom side of Fig. \ref{fig:connection} for a simulation.

\begin{figure}[h!]
\begin{center}
\begin{tabular}{cp{1cm}c}
\includegraphics[scale=0.9]{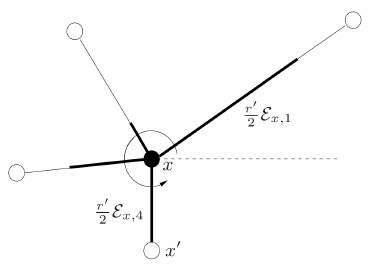} \\ \includegraphics[scale=0.5]{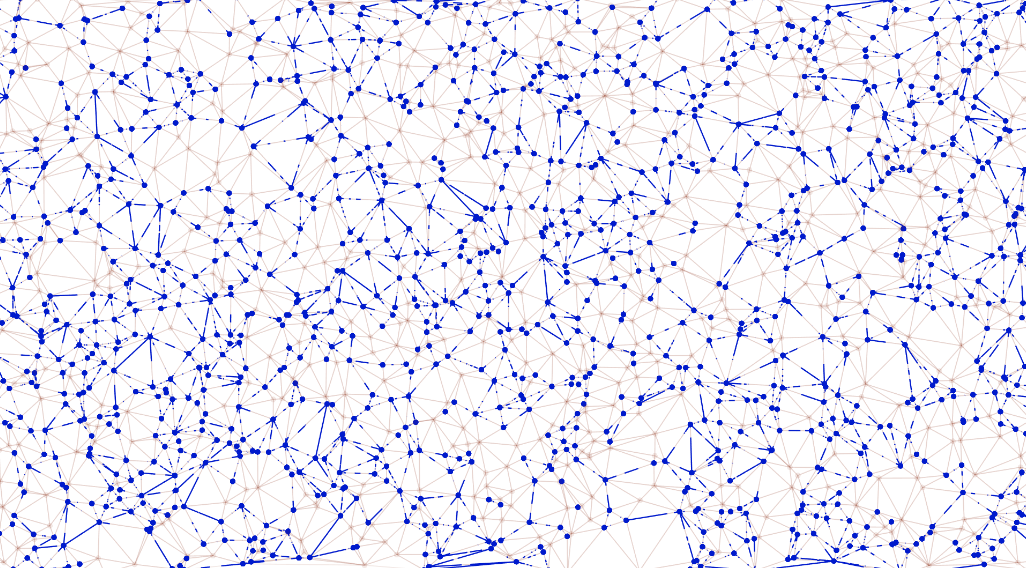}
\end{tabular}
\caption{On top is depicted a user at crossroad $x \in \bfX_p$ (the black point) and its four neighbors in the Delaunay triangulation $\bfT$. The four connection ranges $\frac{r'}{2} \mathcal{E}_{x,1},\ldots,\frac{r'}{2} \mathcal{E}_{x,4}$ are represented with bold lines. On this example, $\frac{r'}{2} \mathcal{E}_{x,4}$ is larger than the distance $\|x - x'\|$, but the corresponding excess does not matter in our model: this is why the bold line is stopped at $x'$. Hence, $x'$-- if it is open --will be automatically connected to $x$, as well as all users on the street $[x,x']$. Below, a simulation of the connectivity graph (in blue) $\bdcalG_{p,\lambda,r}$ in the box $[-60,60]\times [-30,30]$ with $r'=2r=1.8$, $\lambda=1$ and $p=0.7$.}
\label{fig:connection}
\end{center}
\end{figure}

\subsection{Percolation in $\bdcalG_{p,\lambda,r}$}

We say that percolation occurs in $\bdcalG_{p,\lambda,r}$ or that the connectivity graph $\bdcalG_{p,\lambda,r}$ percolates when it contains an unbounded cluster. To describe this phenomenon, we need some notations. For any $x,y \in \bdcalG_{p,\lambda,r}$, we write $x \leftrightarrow y$ in $\bdcalG_{p,\lambda,r}$ (or merely $x \leftrightarrow y$ if there is no ambiguity) if there exists a continuous path $\gamma : [0,1] \to \mathbb{R}^2$ satisfying $\gamma([0,1]) \subset \bdcalG_{p,\lambda,r}$, starting from $\gamma(0) = x$ and going to $\gamma(1) = y$. This notion is naturally extended to subsets $A,B$ of $\bbR^2$: we write $A \leftrightarrow B$ in $\bdcalG_{p,\lambda,r}$ if there exists $(x,y) \in A \times B$ such that $x \leftrightarrow y$ in $\bdcalG_{p,\lambda,r}$. For any real $\alpha > 0$ and $z \in \bbR^2$, let us set $B_\alpha(z) := z + [-\alpha/2 , \alpha/2]^2$ and $S_\alpha(z) := \partial B_\alpha(z)$ its frontier. When $z=0$, let us merely write $B_\alpha := B_\alpha(0)$ and $S_\alpha := S_\alpha(0)$. Hence, $\bdcalG_{p,\lambda,r}$ percolates if the event
\[
\Perco := \bigcup_{u \in \bfX_p \cup \bfY} \{u \leftrightarrow \infty \}
\]
occurs where $\{u \leftrightarrow \infty \} := \cap_{\alpha > 0} \{u \leftrightarrow S_\alpha(u)\}$.

Because our model is ergodic w.r.t. translations of $\mathbb{R}^2$, the event $\Perco$ satisfies a $0$-$1$ law whatever the parameters $p,\lambda,r$: $\bbP \big( \mbox{$\Perco$ in $\bdcalG_{p,\lambda,r}$} \big) \in \{ 0 , 1 \}$. See Chapter 2.1 of \cite{meester1996continuum} for details. Let us now introduce the \textit{percolation threshold} $\lambda_c(p,r)$ defined as
\begin{equation}
\label{critical}
\lambda_c(p,r) := \inf \left\{ \lambda \geq 0 : \, \bbP \big( \mbox{$\Perco$ in $\bdcalG_{p,\lambda,r}$} \big) = 1 \right\} \in [0,\infty]
\end{equation}
indicating from which intensity of the Cox process $\bfY$, the graph $\bdcalG_{p,\lambda,r}$ percolates. Precisely, $\lambda < \lambda_c(p,r)$ implies that $\bdcalG_{p,\lambda,r}$ a.s. admits only finite clusters while $\lambda > \lambda_c(p,r)$ implies the a.s. existence of an unbounded cluster in $\bdcalG_{p,\lambda,r}$.

Let us point out that, by standard coupling arguments, the probability $\bbP \big( \mbox{$\Perco$ in $\bdcalG_{p,\lambda,r}$} \big)$ is a non-decreasing function w.r.t. to each of parameters $p,\lambda,r$. As a consequence, the percolation threshold $\lambda_c(p,r)$ is a non-increasing function w.r.t. $p$ and $r$.\\

Our goal here is to identify whether the connectivity graph $\bdcalG_{p,\lambda,r}$ percolates or not according to parameters $p,\lambda,r$. The current work has been initially motivated by Le Gall et al \cite{le2021continuum} that we describe in the next section.

\subsection{The starting point of our work: Le Gall et al \cite{le2021continuum}}
\label{sect:StartingPoint}

\noindent
\textbf{Their model.} In \cite{le2021continuum}, Le Gall and his co-authors introduced a model for D2D networks which is very similar to our connectivity graph $\bdcalG_{p,\lambda,r}$. Actually there are only two differences between their model and ours.

The first difference consists in the fact that their D2D network is based on \textbf{a Voronoi tesselation instead of a Delaunay triangulation}. Precisely, a Voronoi tesselation, built from a PPP with intensity $1$, provides the urban media in which the streets are the edges delimiting the Voronoi cells and the crossroads are the locations where streets meet.

Users are defined in the same way. Users at crossroads are generated by a Bernoulli site percolation process with parameter $p$, that we still denote by $\bfX_p$. Users on streets are given by a Cox process on $\bbR^2$, with intensity measure $\lambda \Lambda$ where $\Lambda$ is still the one-dimensional Hausdorff measure associated to the Voronoi tesselation. The process of users on streets is again denoted by $\bfY$.

The second difference between the model of \cite{le2021continuum} and ours lies in \textbf{the connection rule which is deterministic for Le Gall et al while it is random for our model} (see Section \ref{sect:OurModel}). In \cite{le2021continuum}, two users $u,u' \in \bfX_p \cup \bfY$ are connected iff they belong to the same street (LOS constraint) and $\| u - u' \| \leq r$.

Finally, we keep the same notations for the connectivity graph $\bdcalG_{p,\lambda,r}$, the percolation event $\Perco$ and the critical intensity $\lambda_c(p,r)$.\\

\noindent
\textbf{Their results.} Results of \cite{le2021continuum} can be described using a phase diagram-- see Fig. \ref{fig:DiagPhase} --in which, for any values of parameters $r$ (in abscissa) and $p$ (in ordinate), the critical intensity $\lambda_c(p,r)$ is determined; null, infinite or non-trivial (i.e. in $\mathbb{R}_+^\ast$).

\begin{figure}[h!]
\begin{center}
\includegraphics[scale=1.5]{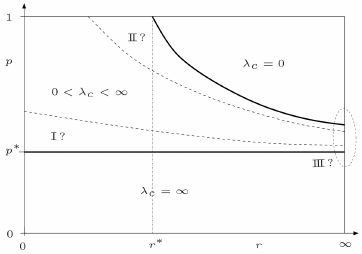}
\caption{Between the two dotted lines is the non-trivial regime corresponding to a non-trivial critical intensity. In this regime, the intensity $\lambda$ of users on streets really matters to determine if percolation occurs or not. This phase diagram presents three uncertain regions marked by the roman symbols \Romanbar{1}, \Romanbar{2} and \Romanbar{3}.}
\label{fig:DiagPhase}
\end{center}
\end{figure}

Le Gall et al exhibit a threshold $p^\ast \in (0,1)$ such that, for any $p < p^\ast$, the graph $\bdcalG_{p,0,\infty}$ a.s. does not percolate. Remark that, in this graph, if both extremities of a street are in $\bfX_p$, they are automatically connected since $r = \infty$. Henceforth, adding users on streets cannot help to percolate: for any $\lambda > 0$, $\bdcalG_{p,\lambda,\infty}$ does not percolate either meaning that $\lambda_c(p,\infty) = \infty$. By monotonicity, $\lambda_c(p,r) = \infty$ for any $p < p^\ast$ and any $r$. Roughly speaking, below the horizontal line $p=p^\ast$, there is never percolation.

In \cite{le2021continuum}, the authors mention an hypothetical region-- denoted by the symbol \Romanbar{1} in Fig. \ref{fig:DiagPhase} --that they conjecture not to exist and corresponding to the overlapping of the \textit{permanantly subcritical regime} $\{(r,p) : \lambda_c(r,p) = \infty \}$ above the line $p = p^\ast$. This suspicious region \Romanbar{1} corresponds to couples $(r,p)$ for which $\bdcalG_{p,0,\infty}$ percolates but not $\bdcalG_{p,\lambda,r}$ whatever the value of $\lambda$ (even large).

Now let us introduce the following percolation threshold: for any $r$,
\begin{equation}
\label{p_c}
p_c(r) := \inf \left\{ p \geq 0 : \, \bbP \big( \mbox{$\Perco$ in $\bdcalG_{p,0,r}$} \big) = 1 \right\} \wedge 1 ~.
\end{equation}
The curve $r \mapsto p_c(r)$ is represented in bold in Fig. \ref{fig:DiagPhase}. It is non-increasing by monotonicity and may admit discontinuity points. It is proved in \cite{le2021continuum} that this curve leaves the horizontal line $p=1$ at some non-trivial $r^\ast$. By definition, given $r$ and $p > p_c(r)$, the graph $\bdcalG_{p,0,r}$ percolates a.s. meaning that $\lambda_c(r,p) = 0$. For such couple $(r,p)$, the graph $\bdcalG_{p,\lambda,r}$ always percolate (for any value of $\lambda$). Again, an hypothetical region appears corresponding to the overlapping of the \textit{permanently supercritical regime} $\{(r,p) : \lambda_c(r,p) = 0 \}$ below the curve $\{ (r,p) : p = p_c(r) \}$ and denoted by the symbol \Romanbar{2} in Fig. \ref{fig:DiagPhase}. This suspicious region \Romanbar{2}, conjectured in \cite{le2021continuum} to not exist, corresponds to couples $(r,p)$ for which $\bdcalG_{p,0,r}$ does not percolate a.s. but, for any (small) $\lambda > 0$, the graph $\bdcalG_{p,\lambda,r}$ percolates.

Finally, let us quote a third uncertain region in the phase diagram depicted in Fig. \ref{fig:DiagPhase} by the symbol \Romanbar{3} and corresponding to the limit of $p_c(r)$ when $r \to \infty$: it is necessarily larger than $p^\ast$ but possibly equal to it.\\ 

In the current paper, about our connectivity graph $\bdcalG_{p,\lambda,r}$ defined in Section \ref{sect:OurModel}, we specify the phase diagram obtained in \cite{le2021continuum} and represented in Fig. \ref{fig:DiagPhase} in the three following directions. We prove that
\begin{itemize}
\item[$(i)$] the region \Romanbar{1} does not exist;
\item[$(ii)$] the region \Romanbar{2} is reduced (at most) to the curve $\{ (r,p) : p = p_c(r) \}$;
\item[$(iii)$] the region \Romanbar{3} is cleaned; when $r\to \infty$, $p_c(r)$ tends to $1/2$ (which will be our $p^\ast$).
\end{itemize}
These three improvements are presented in the next section (Theorem \ref{thm:mainResult}). 

However, the price to pay to get these results is to modify the original model studied in \cite{le2021continuum} by Le Gall and his co-authors. We justify these modifications in Section \ref{sect:Modifs}

\subsection{Our results: a precise phase diagram}

\begin{thm}
\label{thm:mainResult}
Using the notations of Section \ref{sect:OurModel}, the following statements hold.
\begin{enumerate}
\item If $r\in [0,\infty]$ and $p\in [0,1/2]$ then $\lambda_{c}(p,r) = \infty$. If $r\in (0,\infty]$ and $p\in (1/2,1]$ then $\lambda_{c}(p,r) < \infty$. In particular, the permanently subcritical regime is identified:
\[
\{(r,p) : \lambda_c(r,p) = \infty \} = [0,\infty] \times [0,1/2] ~.
\]
\item Let us set $r^{\ast} := \sup \{ r \geq 0 : p_{c}(r) = 1\}$. Then $0 < r^{\ast} < \infty$.
\item The sequence $(p_{c}(r))_{r\geq 0}$ is non-increasing and tends to $1/2$ as $r \to \infty$.
\item The permanently supercritical regime $\{ (r,p) : \lambda_{c}(p,r) = 0 \}$ is equal to
\[
\{ (r,p) : \, r^\ast < r < \infty \; \mbox{ and } \; p > p_c(r) \}
\]
up to possible additional points lying on the curve $\{ (r,p) : p = p_c(r) \}$.
\end{enumerate}
\end{thm}

Let us comment Theorem \ref{thm:mainResult} and explain how the phase diagram represented in Fig. \ref{fig:NotreDiag} is obtained.

Let us first focus on the connectivity graph $\bdcalG_{p,0,\infty}$ with an infinite connection radius $r = \infty$ (in this case, users on streets are irrelevant and we take $\lambda = 0$). This graph exactly corresponds to the Voronoi percolation model on $\bfX$ since two open crossroads which are both extremities of the same edge of $\bfT$, are automatically linked in $\bdcalG_{p,0,\infty}$ regardless of the distance between them. This model has been widely studied \cite{bollobas2006critical,duminil2019exponential}: the critical probability for this model is known to be $1/2$ with no percolation at the critical point. This justifies the first part of Item $1$: there is no percolation when $p\leq 1/2$ because of the lack of users at crossroads (or relays), whatever the parameters $r,\lambda$. In other words $\lambda_{c}(p,r) = \infty$. The second part of Item $1$, that is $\lambda_{c}(p,r) < \infty$ whenever $p > 1/2$ and $r > 0$, means that for such values of $p,r$, the connectivity graph $\bdcalG_{p,\lambda,r}$ percolates for sufficiently large intensity $\lambda$. W.r.t. results of \cite{le2021continuum} recalled in Section \ref{sect:StartingPoint}, we prove the conjecture, namely the hypothetical region \Romanbar{1} does not exist in our context.

The threshold $p_c(r)$ is defined in our context as in \cite{le2021continuum}, see (\ref{p_c}). Thanks to the monotonicity property of $\bbP \big( \mbox{$\Perco$ in $\bdcalG_{p,0,r}$} \big)$ w.r.t. $r$, the application $r \mapsto p_c(r)$ is non-increasing and starts at $p_c(0) = 1$ since $\bbP \big( \mbox{$\Perco$ in $\bdcalG_{1,0,0}$} \big) = 0$. It forks from the horizontal line $p = 1$ at $r^{\ast} = \sup \{ r \geq 0 : p_{c}(r) = 1\} \in (0 , \infty)$ by Item $2$ and tends (from above) to $1/2$ as $r$ tends to infinity by Item $3$. Also, Items $2$ and $3$ allow proving that the permanently supercritical regime $\{ (r,p) : \lambda_{c}(p,r) = 0 \}$ is non-empty. This regime will be specified by Item $4$.

Item $4$ asserts that the permanently supercritical regime does not overflow below the curve $\{ (r,p) : p = p_c(r) \}$. In particular, we prove that the hypothetical region \Romanbar{2} mentioned in \cite{le2021continuum} and recalled in Section \ref{sect:StartingPoint}, made up with couples $(r,p)$ for which the graph $\bdcalG_{p,\lambda,r}$ does not percolate at $\lambda = 0$ but percolates whenever $\lambda > 0$, is necessarily included in the curve $\{ (r,p) : p = p_c(r) \}$ and then has an empty interior.

\begin{figure}[h!]
\begin{center}
\includegraphics[scale=1.4]{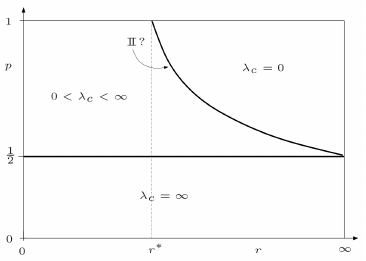}
\caption{Here is the phase diagram corresponding to Theorem \ref{thm:mainResult}. W.r.t. the phase diagram of \cite{le2021continuum} and recalled in Fig. \ref{fig:DiagPhase}, three blurred regions have been removed: the hypothetical region \Romanbar{1} does not exist, the hypothetical region \Romanbar{2}-- if it exists --is reduced to the curve $\{ (r,p) : p = p_c(r) \}$ and $p_c(r)$ tends to $1/2$ as $r$ tends to $\infty$. Remark that both opposite regimes, namely the permanently subcritical and supercritical ones, become very close from each other as $r \to \infty$.}
\label{fig:NotreDiag}
\end{center}
\end{figure}

To complete the previous results, we state crucial information describing the connectivity of the network: when percolation occurs, the unbounded cluster is a.s. unique.

\begin{thm}
\label{thm:uniqueness}
For any set of parameters $(p,\lambda,r)$, there is at most one unbounded cluster in the connectivity graph $\bdcalG_{p,\lambda,r}$ with probability $1$.
\end{thm}

\subsection{Why do we change the model?}
\label{sect:Modifs}

In order to clean the blurred regions present in the phase diagram of \cite{le2021continuum} and recalled in Fig. \ref{fig:DiagPhase}, we have made two modifications. The first one concerns the urban media. In \cite{le2021continuum}, the street system is given by the edges delimiting the Voronoi cells and the set of relays is provided by a Bernoulli site percolation process performed on crossroads of such street system. However, this model behaves badly w.r.t. percolation properties: in a (say large) box, the absence of horizontally crossing open paths does not imply the presence of a vertically crossing closed one. See the left hand side of Fig. \ref{fig:NoFKG}. Mainly for this reason, this percolation model has not been intensively studied and very few percolation results are known about it.

Conversely, the Bernoulli site percolation model performed on the vertices of the Delaunay triangulation $\bfT$ is much more understood (especially in dimension $2$); see \cite{bollobas2006critical,duminil2019exponential}. It is self-dual (the critical probability is $1/2$), it satisfies the FKG inequality and its phase transition is \textit{sharp}. This last (and deep) property is used to prove that the probability for the graph $\bdcalG_{p,0,\infty}$ of containing a cycle in $B_{3n}$ surrounding $B_n$ (event denoted by the event $\mathcal{C}_{n,0,\infty}$ in Section \ref{ssec:highintensity}) tends to $1$ as $n \to \infty$ whenever $p > 1/2$. This important feature allows us to prove that region \Romanbar{1} does not exist in our context, and lead to Item 1. In comparison, the same fact also holds for the site percolation on crossroads of the street system used in \cite{le2021continuum} but only for $p$ close to $1$.

Stating that $p_c(r)$ tends to $1/2$ (Item 3.) and then clarifying what happens when $r \to \infty$, is based on similar arguments.

\begin{figure}[h!]
\begin{center}
\begin{tabular}{cp{1cm}c}
\includegraphics[scale=0.4]{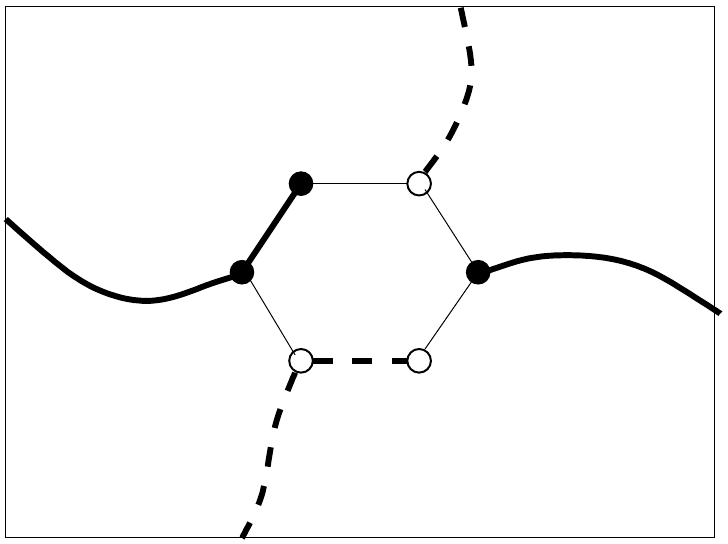} & & \includegraphics[scale=1.2]{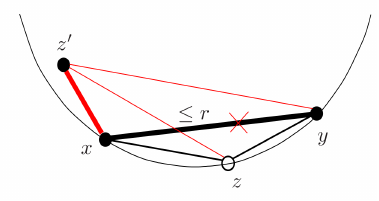}
\end{tabular}
\caption{To the left: for the site percolation model on crossroads of a Voronoi tesselation, it is possible that no horizontally crossing open paths nor vertically crossing closed ones occur. To the right: focus only on $x,y,z \in \bfX$, $z'$ will be added further. Initially, there is no Poisson points in the circumscribed circle $\mathcal{C}$ to $x,y,z$ so the triangle defined by these three vertices is present in the Delaunay triangulation $\bfT$. $x,y$ (black points) are open, i.e. in $\bfX_p$, while $z$ (white point) is closed. In $\bdcalG_{p,0,r}$, $x$ and $y$ are connected since $\| x - y \| \leq r$. Now, let us add a black point (namely $z' \in \bfX_p$) inside $\mathcal{C}$ but at distance from $y$ larger than $r$. The Delaunay triangulation changes (new edges are red): the adding of $z'$ destroyes the edge $\{x,y\}$ and new triangles appear ($xz'z$ and $zz'y$). If $x$ and $y$ were linked in $\bdcalG_{p,0,r}$ only through their common edge, they are no longer linked in $\bdcalG_{p,0,r}$ after adding $z' \in \bfX_p$.}
\label{fig:NoFKG}
\end{center}
\end{figure}

Dealing with the hypothetical region \Romanbar{2} has required much more investigations. The starting point is to assume by absurd that the region \Romanbar{2} is fat. Hence, one could find some $r$ and $p < p_c(r)$ such that $\bdcalG_{p,0,r}$ does not percolate while $\bdcalG_{p,\lambda,r}$ percolates for any (small) $\lambda > 0$. If we were able to prove that $\bdcalG_{\cdot,0,r}$ admits a sharp transition at $p_c(r)$ then we could prove that $\bbP(S_\alpha \leftrightarrow S_{10\alpha} \; \mbox{in $\bdcalG_{p,\lambda,r}$}) \leq \varepsilon$ for $\lambda > 0$ small enough. Thus a multi-scale argument due to Gou\'er\'e \cite{gouere2008subcritical} then allows to show that $\bdcalG_{p,\lambda,r}$ does not percolate which is in contradiction with our initial (absurd) assumption. Our first attempt to get sharp transition has consisted in applying the powerful method of Duminil-Copin et al. \cite{duminil2019exponential,duminil2020subcritical} based on the OSSS inequality. However, the event $\{ S_1 \leftrightarrow S_{n} \; \mbox{in $\bdcalG_{p,0,r}$})$ being not FKG-increasing (i.e. w.r.t. the partial order corresponding to the FKG inequality, see the right hand side of Fig. \ref{fig:NoFKG} for an explanation of this annoying fact), we could not carry out this strategy. To be complete, let us point out here that a recent approach \cite{HJM} for Cox percolation asserts that the strategy of \cite{duminil2019exponential,duminil2020subcritical} may apply without using the FKG inequality, but under Conditions (a)-(b) in Section 2.2 of \cite{HJM} which are both false in our context. A second attempt was to use an alternative approach for sharp transition \cite{duminil2016sharpness}, thus applied by Ziesche \cite{Ziesche} in a continuum context, but the spatial dependencies generated by the Delaunay triangulation makes this strategy inapplicable in our context.

Let us modify our absurd assumption as follows: assume that there exists $(r,p)$ and $\varepsilon > 0$ (small) such that $\bdcalG_{p,0,r+\varepsilon}$ does not percolate while $\bdcalG_{p,\lambda,r}$ percolates for any (small) $\lambda > 0$. Our strategy consists now in proving that, in terms of percolation, a small increase of the user intensity (from $0$ to $\lambda$) can be compensated by a small increase of the connection radius (from $r$ to $r+\varepsilon$), then leading to a contradiction. Such ideas have been developed in the context of \textit{enhancement} (see for instance \cite{GS,Sarkar,DP}). To do it, let us set $\theta_n(r,\lambda) := \bbP(S_1 \leftrightarrow S_{n} \; \mbox{in $\bdcalG_{p,0,r}$})$. Introducing notion of \textit{pivotal} edges, it is possible to give a rigorous sense to partial derivatives $\partial_r \theta_n(r,0)$ and $\partial_\lambda \theta_n(r,0)$. When the connection rule between two users is deterministic as in \cite{le2021continuum} (i.e. $u \sim u'$ iff LOS and $\| u - u' \| \leq r$), computations give $\partial_r \theta_n(r,0) = 0$ while $\partial_\lambda \theta_n(r,0) > 0$ so that it is impossible to compensate a small increase of parameter $\lambda$ by a small increase of parameter $r$.

The fact that $\partial_r \theta_n(r,0) = 0$ is due to the rigid character of the deterministic connection rule of \cite{le2021continuum}. This is the reason why we replace this deterministic rule with a random one involving a distribution with an unbounded support: see Section \ref{sect:OurModel} for details. Thanks to this new connection rule, we are able in Section \ref{sect:Item4} to compute and compare the partial derivatives $\partial_r \theta_n(r,\lambda)$ and $\partial_\lambda \theta_n(r,\lambda)$ and then to carry out this strategy in order to state that the region \Romanbar{2} cannot be fat.

The choice of an exponential distribution is certainly not the only one that works. This will be discussed in Section \ref{sec:finaldiscussion}. Besides, the fact to use independent exponential r.v.'s along different incident edges at a given vertex $x \in \bfX_p$ allows to directly apply Russo's formula and then to give sense to partial derivatives $\partial_r \theta_n$ and $\partial_\lambda \theta_n$. However, with more work, it is could be possible to use the same r.v. along all incident edges ( see the discussion in Section \ref{sec:finaldiscussion}).

Finally we think that this strategy of comparing partial derivatives to establish compensations between small variations of certain parameters of the model should turn out to be promising for many models (with several parameters) of telecommunication networks.\\

The rest of paper is organized as follows. Section \ref{sect:Items1-3} is devoted to the first three items of Theorem \ref{thm:mainResult}. The fourth one which needs to introduce pivotal edges (Definition \ref{dfn:pivotaledge}), is proved in Section \ref{sect:Item4}. Section \ref{sect:uniqueness} states the uniqueness of the unbounded cluster when it exits, i.e. the proof of Theorem \ref{thm:uniqueness}. The paper ends with a discussion on possible improvements.

\section{Proof of Theorem \ref{thm:mainResult}, Items $1$-$3$}
\label{sect:Items1-3}

\subsection{Stabilization}
\label{sect:Stabilization}

Recall that $\bfT$ is the Delaunay triangulation generated by the PPP $\bfX$. A key point about Delaunay triangulation is the following: any triangle $T$ of $\bfT$ is not sensitive to process resampling $\bfX$ outside its circumscribed circle $C$. One sometimes says that the circle $C$ (or its associated disk) \textit{stabilizes} the triangle $T$. This basic remark drives the next definition.

For any bounded Borel set $A$, we define the \textit{stabilization radius} of $A$, denoted by $R(A)$, as the infimum $\rho > 0$ such that $A \oplus B(0,\rho)$ contains all the circumscribed circles to the triangles of $\bfT$ overlapping $A$. Hence, the Delaunay triangulation restricted to $A$ and its one-dimensional Hausdorff random measure, i.e. $\Lambda_A(\cdot) := \Lambda(A \cap \cdot)$, only depend on the PPP $\bfX$ restricted to $A \oplus B(0,\rho)$ whenever $R(A) < \rho$. For more details about stabilization, the reader may refer to \cite{hirsch2019continuum}, Definition 2.3 and Example 3.1 (applied to a Poisson Voronoi tesselation).

In the current work, we only use both elementary properties of stabilization radii stated in Lemmas \ref{lem:StabExpo} and \ref{lem:StabIndependence} below.

\begin{lem}
\label{lem:StabExpo}
There exist constants $C,c > 0$ such that, for any $n$,
\begin{equation}
\label{StabExpo}
\mathbb{P} ( R(B_{n}) > n ) \leq C e^{-c n} ~.
\end{equation}
\end{lem}

\begin{proof}
The event $R(B_{n}) > n$ means that there exists a (random) disk with radius $n$ overlapping the set $B_n$ and empty of Poisson points. This forces at least one the $Q_i$'s to be empty of Poisson points where $Q_1,\ldots,Q_{8^2}$ are congruent squares with size $n/4$ covering $B_{2n}$ (i.e. each $Q_i$ is a translated of $B_{n/4}$). Henceforth,
\[
\mathbb{P} ( R(B_{n}) > n ) \leq \sum_{i=1}^{8^2} \mathbb{P} ( \bfX \cap Q_i = \emptyset ) \leq 8^2 e^{-(n/4)^2}
\]
from which (\ref{StabExpo}) follows.
\end{proof}

This next result is an independence property satisfied by the connectivity graph $\bdcalG_{p,\lambda,r}$ on subsets far enough from each other and provided their stabilization radii are well controlled.

\begin{lem}
\label{lem:StabIndependence}
Let $m>0$ be an integer and $\rho>0$ be a real number. Let us consider some Borel sets $A_{1},\ldots,A_{m}$ such that $d(A_i , A_j) \geq 2\rho$, for $1 \leq i\not= j\leq m$, and any family of positive, bounded r.v.'s $\chi_1,\ldots,\chi_m$ where each $\chi_i$ is measurable w.r.t. $\bdcalG_{p,\lambda,r} \cap A_{i}$, i.e. w.r.t. $(\Lambda_{A_i}, \bfX_p \cap A_i, \bfY \cap A_i)$. Then, the following equality holds:
\[
\bbE \Big[ \prod_{i=1}^{m} \chi_{i} \mathds{1}_{R(A_{i}) < \rho} \Big] = \prod_{i=1}^{m} \bbE \Big[ \chi_{i} \mathds{1}_{R(A_{i}) < \rho} \Big] ~.
\]
\end{lem}

\begin{proof}
First, we use that conditionally to $\bfX$ the processes $\bfX_p$ and $\bfY$ restricted to disjoint sets-- the $A_i$'s --provides independent r.v.'s:
\[
\bbE \Big[ \prod_{i=1}^{m} \chi_{i} \, \big| \, \bfX \Big] = \prod_{i=1}^{m} \bbE \big[ \chi_{i} \, \big| \, \bfX \big] ~.
\]
By construction of $R(A_{i})$, the r.v. $\mathds{1}_{R(A_{i}) < \rho} \, \bbE \big[ \chi_{i} \, \big| \, \bfX \big]$ only depends on the PPP $\bfX$ through the set $A_i \oplus B(0,\rho)$. Since the $A_i$'s are at distance $2\rho$ (at least) from each other, the independence property of the PPP $\bfX$ allows us to write:
\begin{eqnarray*}
\bbE \Big[ \prod_{i=1}^{m} \chi_{i} \mathds{1}_{R(A_{i}) < \rho} \Big] & = & \bbE \Big[ \mathds{1}_{R(A_{i}) < \rho} \, \bbE \Big[ \prod_{i=1}^{m} \chi_{i} \, \big| \, \bfX \Big] \Big] \\
& = & \bbE \Big[ \prod_{i=1}^{m} \mathds{1}_{R(A_{i}) < \rho} \, \bbE \big[ \chi_{i} \, \big| \, \bfX \big] \Big] \\
& = & \prod_{i=1}^{m} \bbE \Big[ \mathds{1}_{R(A_{i}) < \rho} \, \bbE \big[ \chi_{i} \, \big| \, \bfX \big] \Big] \\
& = & \prod_{i=1}^{m} \bbE \Big[ \chi_{i} \mathds{1}_{R(A_{i}) < \rho} \Big] ~.
\end{eqnarray*}
\end{proof}

\subsection{Proof of Item 1}
\label{ssec:highintensity}

In the previous section, we have explained that the connectivity graph $\bdcalG_{p,0,\infty}$ coincides with the (independent) site percolation process on the Delaunay triangulation $\bfT$, also known in the literature as the Voronoi percolation model (in which any Voronoi cell is colored black when its center is open). Since \cite{bollobas2006critical,duminil2019exponential}, it is known that this model does not percolate with probability $1$ for any $p \in [0,1/2]$. Hence, this prevents $\bdcalG_{p,\lambda,r}$ to percolate whatever the values of $\lambda,r$ (while $p \in [0,1/2]$), i.e. $\lambda_{c}(p,r) = \infty$.\\

In the following, our goal is to prove that $\lambda_{c}(p,r) < \infty$ for any given $p > 1/2$ and $r \in (0,\infty]$. Hence, let us pick such parameters $p > 1/2$ and $r \in (0,\infty]$: it suffices to prove that percolation occurs for $\bdcalG_{p,\lambda,r}$ when $\lambda$ is large enough. Our strategy consists in discretizing our model and comparing it to a supercritical site percolation process. Such a renormalization argument is classic in Percolation theory and used in particular in \cite{le2020crowd}.

To do it, let us consider for any $z \in \mathbb{Z}^{2}$ the events 
\begin{itemize}
\item $\calD_{n}(z) := \{ R(B_{3n}(nz)) < n \}$ where the stabilization radii are defined in Section \ref{sect:Stabilization};
\item $\calC_{n,\lambda,r}(z)$ interpreted as '$\bdcalG_{p,\lambda,r}$ contains a cycle included in $B_{3n}(nz)$ and surrounding $B_{n}(nz)$'. See Fig. \ref{fig:eventngood}.
\end{itemize}
By planarity of the PDT $\bfT$, $\calC_{n,\lambda,r}(z)$ is also characterized by its complementary event, i.e. there is a continuous path joining $S_{n}(nz)$ to $S_{3n}(nz)$ in $\bfT\!\setminus\!\bdcalG_{p,\lambda,r}$.

\begin{figure}[h!]
\centering 
\includegraphics[scale=0.4]{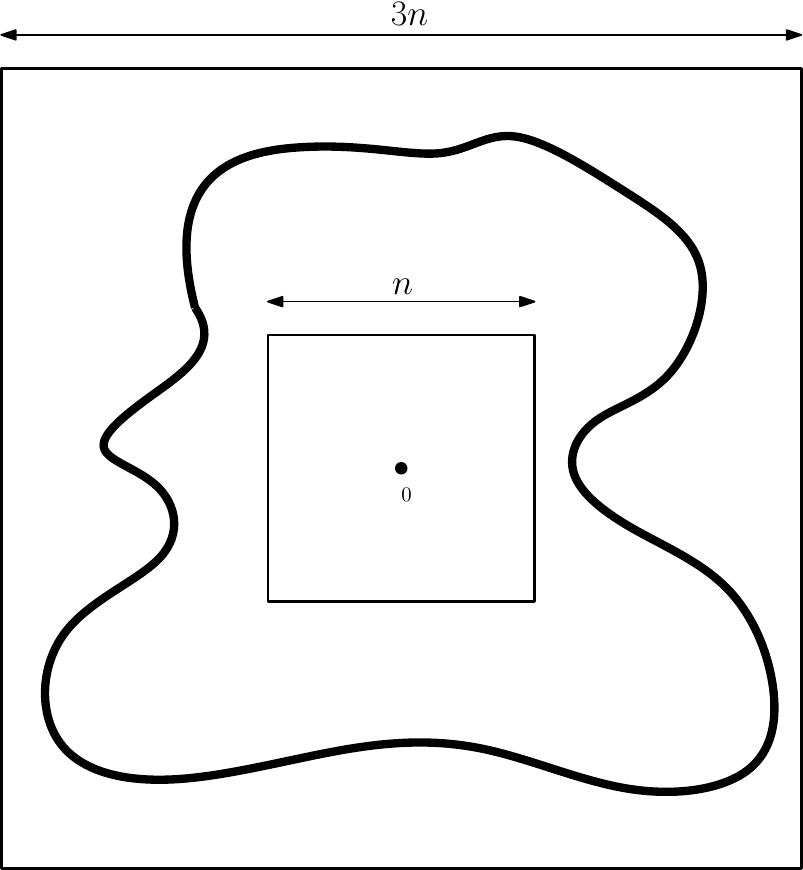}
\caption{The event $\calC_{n,\lambda,r}(0)$, merely denoted by $\calC_{n,\lambda,r}$, relative to $B_{3n}$ is depicted.}
\label{fig:eventngood}
\end{figure}

A site $z \in \bbZ^2$ is said \emph{$n$-good} if the event $\calD_{n}(z) \cap \calC_{n,\lambda,r}(z)$ occurs. Thus, we consider the site percolation process $\{\zeta_{n,z} : z \in \bbZ^2\}$ defined by
\[
\zeta_{n,z} := \mathds{1}_{\mbox{$z$ is $n$-good}} ~.
\]
The discrete random field $\{\zeta_{n,z} : z \in \bbZ^2\}$ is stationary w.r.t.\, translations of $\bbZ^2$ and its percolation (on $\bbZ^2$, for the supremum metric $\|\cdot\|_\infty$) clearly implies that of the connectivity graph $\bdcalG_{p,\lambda,r}$. Indeed, given two $n$-good sites $z,z' \in \bbZ^2$ with $\| z - z' \|_\infty = 1$, their circles in $\bdcalG_{p,\lambda,r}$ whose existence is ensured by $\calC_{n,\lambda,r}(z)$ and $\calC_{n,\lambda,r}(z')$, necessarily overlap. In other words, an unbounded connected component of $n$-good sites provides an unbounded connected component in $\bdcalG_{p,\lambda,r}$.

Henceforth, it is enough to prove that $\{\zeta_{n,z} : z \in \bbZ^2\}$ percolates for $p > 1/2$, $r \in (0,\infty]$ and $\lambda$ is large enough. In this objective, the two following ingredients, namely Lemmas \ref{lem:dependentsites} and \ref{lem:probangoodhighintensity}, will allow us to compare $\{\zeta_{n,z} : z \in \bbZ^2\}$ to a supercritical site percolation process.

\begin{lem}
\label{lem:dependentsites}
W.r.t. the supremum metric $\|\cdot\|_\infty$, the discrete random field $\{\zeta_{n,z} : z \in \bbZ^2\}$ is $4$-dependent.
\end{lem}

The role of the event $\calD_{n}(z)$ is to make local the notion of $n$-goodness.
	
\begin{proof}
Let us prove that for any finite collection of sites $z_1,\dots,z_m \in \bbZ^2$ with $\| z_i - z_j \|_\infty > 4$, $\forall i \neq j$, the r.v.'s $\zeta_{n,z_1},\ldots,\zeta_{n,z_m}$ are mutually independent. Set $A_i := B_{3n}(nz_i)$. The inequality $\| z_i - z_j \|_\infty \geq 5$ means that $d(A_i , A_j) \geq 2n$. We can then apply Lemma \ref{lem:StabIndependence} with the r.v.
\[
\chi_i := \mathds{1}_{\calD_{n}(z_i) \cap \, \calC_{n,\lambda,r}(z_i)} ~,
\]
measurable w.r.t. $\bdcalG_{p,\lambda,r} \cap A_i$, to conclude.
\end{proof}

\begin{lem}
\label{lem:probangoodhighintensity}
The following limit holds:
\[
\lim_{n \to \infty} \lim_{\lambda \to \infty} \bbP (0 \text{\emph{ is $n$-good}} ) = 1 ~.
\]
\end{lem}

Before proving Lemma \ref{lem:probangoodhighintensity}, we first conclude the proof of Theorem \ref{thm:mainResult}, Item $1$. The stochastic domination result of Liggett et al \cite[Theorem 0.0]{liggett1997domination} asserts that the $4$-dependent random field $\{\zeta_{n,z} : z \in \bbZ^2\}$ is stochastically dominated from below by an independent Bernoulli site percolation process $\Theta_{p_0}$, for any $p_0 \in (0,1)$, provided $n,\lambda$ are large enough (depending on $p_0$). If the parameter $p_0 < 1$ is chosen (large enough) so that $\Theta_{p_0}$ a.s. percolates then the same holds for $\{\zeta_{n,z} : z \in \bbZ^2\}$ by stochastic domination. This finally implies the percolation of the connectivity graph $\bdcalG_{p,\lambda,r}$.

\begin{proof}{(of Lemma \ref{lem:probangoodhighintensity})} Recall that $p > 1/2$ and $r \in (0,\infty]$.  Since $\mathbb{P}(R(B_{n}) < n) \to 1$ as $n \to \infty$ by \eqref{StabExpo} then the same holds for $\mathbb{P}(\calD_{n}) = \mathbb{P}(R(B_{3n}) < n)$. In addition, according to \cite[Section 8]{bollobas2006critical}, $\mathbb{P}(\calC_{n,0,\infty})$ also tends to $1$ as $n \to \infty$-- it is crucially used here that $p > 1/2$. We then get:
\begin{equation}
\label{LimDnCn}
\lim_{n \to \infty} \mathbb{P} ( \calD_{n} \cap \calC_{n,0,\infty} ) = 1 ~.
\end{equation}

Let us now introduce the event $\calU_{n,\lambda}$ according to which any edge $\{x,y\}$ of the PDT $\bfT$ with $x,y \in \bfX_p \cap B_{3n}$ is covered by users:
\[
\calU_{n,\lambda} := \big\{ \mbox{for any edge $\{x,y\}$ of $\bfT$ with $x,y \in \bfX_p \cap B_{3n}$ then $[x,y] \subset \bdcalG_{p,\lambda,r}$} \big\} ~.
\]
Lemma \ref{lem:CoxHoleControl} below says that, for any given $n$, the conditional probability $\mathbb{P} (\calU_{n,\lambda} | \bfX , \bfX_p )$ a.s. converges to $1$ as $\lambda \to \infty$. Thus the Lebesgue theorem gives
\[
\mathbb{P} (\calU_{n,\lambda}) = \mathbb{E} \big[ \mathbb{P} (\calU_{n,\lambda} | \bfX , \bfX_p ) \big] \, \to \, 1 \; \mbox{ as } \; \lambda \to \infty ~.
\]
Combining with \eqref{LimDnCn}, we then obtain that
\[
\lim_{n \to \infty} \lim_{\lambda \to \infty} \bbP ( \calD_{n} \cap \calC_{n,0,\infty} \cap \calU_{n,\lambda} ) = 1 ~.
\]
Under the event $\calU_{n,\lambda}$, any path of $\bdcalG_{p,0,\infty}$ included in $B_{3n}$ is still a path of $\bdcalG_{p,\lambda,r}$. This means that $\calC_{n,0,\infty} \cap \calU_{n,\lambda} \subset \calC_{n,\lambda,r}$ and
\[
\bbP ( \calD_{n} \cap \calC_{n,0,\infty} \cap \calU_{n,\lambda} ) \leq \bbP ( \calD_{n} \cap \calC_{n,\lambda,r} ) = \bbP ( 0 \text{\emph{ is $n$-good}} )
\]
from which the searched result follows.
\end{proof}

\begin{lem}
\label{lem:CoxHoleControl}
For any $p,r,n > 0$,
\[
\mbox{a.s.} \; \lim_{\lambda \to \infty} \mathbb{P} (\calU_{n,\lambda} \,|\, \bfX , \bfX_p ) = 1 ~.
\]
\end{lem}

\begin{proof}
We work conditionally to the urban media and the process of users at crossroads, i.e. $\bfX$ and $\bfX_p$. Let us first write
\[
\mathbb{P} \big( \calU_{n,\lambda}^\complement \,|\, \bfX , \bfX_p \big) \leq \sum_{\{x,y\} \in E} \mathbb{P} ( [x,y] \not\subset \bdcalG_{p,\lambda,r} \,|\, \bfX , \bfX_p )
\]
where $E$ denotes the set of edges $\{x,y\}$ of the Delaunay triangulation $\bfT$ such that $x,y \in \bfX_p \cap B_{3n}$. Since $E$ is a.s. finite it suffices to show that each term of the above sum a.s. tends to $0$ as $\lambda\to\infty$. Let $\{x,y\}$ be one of them: only its length, say $\ell$, really matters (where $\ell \leq 3 \sqrt{2} n$ since $x,y \in B_{3n}$).

Let us consider a sequence $(U_k)_{k\geq 1}$ of i.i.d. r.v.'s uniformly distributed on $[0,\ell]$ and a second sequence $(\mathcal{E}_k)_{k\geq 1}$ of i.i.d. r.v.'s with law the exponential distribution with rate $1$, also independent with the $(U_k)$. With probability $1$, at least one of the $(\mathcal{E}_k)$, say $\mathcal{E}_{k_0}$, satisfies $\frac{r}{2} \mathcal{E}_{k_0} \geq \ell$. This implies that a.s.
\[
[0,\ell] \subset \Big[ U_{k_0} \pm \frac{r}{2} \mathcal{E}_{k_0} \Big] \subset \bigcup_{k\geq 1} \Big[ U_{k} \pm \frac{r}{2} \mathcal{E}_{k} \Big] ~.
\]
We can then conclude that, a.s. on $\bfX, \bfX_p$,
\begin{eqnarray*}
1 & = & \lim_{K\to\infty} \mathbb{P} \Big( [0,\ell] \subset \bigcup_{k \leq K} \Big[ U_{k} \pm \frac{r}{2} \mathcal{E}_{k} \Big] \Big) \\
& = & \lim_{\lambda\to\infty} \mathbb{P} ( [x,y] \subset \bdcalG_{p,\lambda,r} \,|\, \bfX , \bfX_p )
\end{eqnarray*}
from which the searched result follows.
\end{proof}

\subsection{Proof of Item 3}
\label{sect:ProofItem3}

In this section, we prove that $(p_{c}(r))_{r\geq 0}$ tends to $1/2$ as $r \to \infty$. By definition of the critical threshold $p_c(r)$, it is sufficient to state that for any $p>\frac{1}{2}$, but thought as close to $1/2$, there exists $r$ large enough such that the connectivity graph $\bdcalG_{p,0,r}$ percolates (i.e. without the help of users on streets), implying that $\lambda_c(p,r) = 0$.

The proof mainly follows the same lines as the one of Item $1$ in Section \ref{ssec:highintensity} with the only difference that Lemma \ref{lem:probangoodhighintensity} has to be replaced with
\begin{equation}
\label{NewLem2}
\lim_{n \to \infty} \lim_{r \to \infty} \bbP (0 \mbox{ is $n$-good}) = 1
\end{equation}
where the definition of $n$-goodness is almost the same; $z \in \mathbb{Z}^2$ is $n$-good iff the event $\calD_{n}(z) \cap \calC_{n,0,r}(z)$ holds. Combining (\ref{NewLem2}) with Lemma \ref{lem:dependentsites}, we can once again apply the stochastic domination result of Liggett et al \cite[Theorem 0.0]{liggett1997domination} stating that the $4$-dependent random field $\{\zeta_{n,z} : z \in \bbZ^2\}$, with $\zeta_{n,z} = \mathds{1}_{z \mbox{ is $n$-good}}$, percolates a.s. which implies that $\bdcalG_{p,0,r}$ percolates too.

In order to prove (\ref{NewLem2}), we introduce the event
\[
\mathcal{V}_{n,r} := \big\{ \mbox{for any edge $\{x,y\}$ of $\bfT$ with $x,y \in \bfX_p \cap B_{3n}$ then $[x,y] \subset \bdcalG_{p,0,r}$} \big\} ~.
\]
On $\mathcal{V}_{n,r}$, for any edge $\{x,y\}$ of $\bfT$ included in $B_{3n}$, both users $x,y \in \bfX_p$ at the extremities are connected without the help of any other user. Thus we will prove a similar result to Lemma \ref{lem:CoxHoleControl}: for any $n$, a.s.
\begin{equation}
\label{Limit-Vnr}
\lim_{r \to \infty} \mathbb{P} (\mathcal{V}_{n,r} \,|\, \bfX , \bfX_p ) = 1 ~.
\end{equation}
On the one hand, (\ref{Limit-Vnr}) implies that $\mathbb{P}(\mathcal{V}_{n,r}) \to 1$ as $r \to \infty$ by the Lebesgue theorem and, on the other hand, we still have that $\mathbb{P}(\calD_{n} \cap \calC_{n,0,\infty}) \to 1$ as $n \to \infty$ (see Section \ref{ssec:highintensity}). We can then conclude:
\[
\bbP ( 0 \text{ is $n$-good} ) = \bbP ( \calD_{n} \cap \calC_{n,0,r} ) \geq \bbP ( \calD_{n} \cap \calC_{n,0,\infty} \cap \mathcal{V}_{n,r} )
\]
which tends to $1$ as $n,r \to \infty$.

It then remains to prove (\ref{Limit-Vnr}). We proceed as in the proof of Lemma \ref{lem:CoxHoleControl}:
\[
\mathbb{P} \big( \mathcal{V}_{n,r}^\complement \,|\, \bfX , \bfX_p \big) \leq \sum_{\{x,y\} \in E} \mathbb{P} ( [x,y] \not\subset \bdcalG_{p,0,r} \,|\, \bfX , \bfX_p )
\]
where $E$ still denotes the set of edges $\{x,y\}$ of the Delaunay triangulation $\bfT$ such that $x,y \in \bfX_p \cap B_{3n}$. Since $E$ is a.s. finite it suffices to show that each conditional probability $\mathbb{P}([x,y] \not\subset \bdcalG_{p,0,r} | \bfX,\bfX_p)$ tends to $0$ as $r\to\infty$. Let $\{x,y\}$ be an edge of $E$: only its length $\ell$ (smaller than $3 \sqrt{2} n$) really matters. Now, using an exponential r.v. $\mathcal{E}$ with rate $1$,
\[
\mathbb{P} ( [x,y] \not\subset \bdcalG_{p,0,r} \,|\, \bfX , \bfX_p ) \leq \mathbb{P} \big( \frac{r}{2} \mathcal{E} \leq 3 \sqrt{2} n \big) \, \to \, 0
\]
as $r \to \infty$ (for a fixed $n$). This proves (\ref{Limit-Vnr}).

\subsection{Proof of Item 2}

Recall that $(p_{c}(r))_{r \geq 0}$ is a non-increasing sequence and $r^\ast$ is defined as the value at which it forks from the horizontal line $p = 1$:
\[
r^{\ast} = \sup \{ r \geq 0 : p_{c}(r) = 1\} \in [0 , \infty] ~.
\]
The finiteness of $r^{\ast}$ is given by Theorem \ref{thm:mainResult}, Item $3$ since $(p_{c}(r))_{r\geq 0}$ tends to $1/2$ as $r \to \infty$.\\

It then remains to prove that $r^{\ast} > 0$. To do it, let us prove the existence of a small $r > 0$ such that the connectivity graph $\bdcalG_{1,0,r}$-- in which there is no users on streets but all crossroads are open --does not percolate with probability $1$, meaning that $p_c(r) = 1$ and then $r^\ast \geq r$.\\

\noindent
\textbf{Comparison to the random set of grains $\Sigma_{1/n}$.} Given $x \in \bfX$, we define the grain $\Star(x,r)$ as follows. Let $y_1,\ldots,y_{\deg(x)}$ be the neighbors of $x$ in the Delaunay triangulation, numbered in the counterclockwise sense and from the semi-line $[x,x+(1,0))$. Recall that $\frac{r}{2}\mathcal{E}_{x,k}$ is the range of connection from $x$ along the segment $[x,y_k]$. Hence, we set
\[
y_k(x) := \left\lbrace
\begin{array}{ll}
x + \frac{r}{2}\mathcal{E}_{x,k} \frac{(y_k - x)}{|y_k - x|} & \mbox{if } \, \frac{r}{2}\mathcal{E}_{x,k} < |y_k - x| \\
y_k  & \mbox{otherwise.}
\end{array}
\right.
\]
Thus the linear piecewise closed curve joining the extremities $y_1(x),y_2(x),\ldots,y_{\deg(x)}(x)$ and at last $y_1(x)$ delimits a compact set, denoted by $\Star(x,r)$. Remark that the grains $\Star(\cdot,r)$ are dependent from each other (through the Delaunay triangulation) and are decreasing with $r$ (in the sense of inclusion). Besides, the connectivity graph $\bdcalG_{1,0,r}$ is included in the random set
\[
\Sigma_r := \bigcup_{x \in \bfX} \Star(x,r)
\]
so that it is sufficient to prove that $\Sigma_{1/n}$ does not percolate for $n$ large enough.

In this goal, we are going to apply the same strategy as in Sections \ref{ssec:highintensity} and \ref{sect:ProofItem3}. For any $z \in \mathbb{Z}^2$, let us define the event $\mathcal{C}_n(z)$ as follows: there exists a continuous path in the plane included in $B_{3n}(nz)$ but surrounding $B_{n}(nz)$, and avoiding the set $\Sigma_{1/n}$. As before, we say that $z$ is \textit{$n$-good} if $\mathcal{D}_n(z) \cap \mathcal{C}_n(z)$ occurs where $\calD_{n}(z)$ still denotes the event $\{ R(B_{3n}(nz)) < n \}$ and we set $\zeta_{n,z} = \mathds{1}_{\mbox{$z$ is $n$-good}}$. Because the percolation of the stationary, discrete random field $\{\zeta_{n,z} : z \in \bbZ^2\}$ (w.r.t. the supremum metric $\|\cdot\|_\infty$) implies the absence of percolation for $\Sigma_{1/n}$, we are now studying the process $\{\zeta_{n,z} : z \in \bbZ^2\}$.\\

\noindent
\textbf{Percolation of the process $\{\zeta_{n,z} : z \in \bbZ^2\}$.} Let  $z \in \mathbb{Z}^2$. Assume that the grain $\Star(x,1/n)$, with $x \in \bfX$, overlaps $B_{3n}(nz)$ and then is possibly relevant for the occurrence of the event $\mathcal{C}_n(z)$. Precisely we assume that the subset $S$ of $\Star(x,1/n)$, delimited by $x, y_i(x), y_{i+1}(x)$, with $i < \deg(x)$, overlaps $B_{3n}(nz)$. By construction, $S$ is included in the Delaunay triangle made up with Poisson points $x, y_i, y_{i+1}$, say $T$. The triangle $T$ overlaps $B_{3n}(nz)$ and, under the event $\calD_{n}(z) = \{ R(B_{3n}(nz)) < n \}$, it only depends on the PPP $\bfX$ inside $B_{3n}(nz) \oplus B(0,n)$. So, using the same proof as that of Lemma \ref{lem:dependentsites} (itself using Lemma \ref{lem:StabIndependence}), we get that the random field $\{\zeta_{n,z} : z \in \bbZ^2\}$ is $4$-dependent w.r.t. the supremum metric $\|\cdot\|_\infty$.

It is worth pointing out here that, under $\calD_{n}(z)$, the subset $S$ of $\Star(x,1/n)$ is stabilized by the PPP inside $B_{3n}(nz) \oplus B(0,n)$, but not the whole grain $\Star(x,1/n)$. Indeed, it is possible that $\Star(x,1/n)$ contains a very large subset, say delimited by $x, y_j(x), y_{j+1}(x)$, avoiding $B_{3n}(nz)$ and exceeding from $B_{3n}(nz) \oplus B(0,n)$. This justifies the use of grains $\Star(x,1/n)$ instead of (larger) balls $B(x, \frac{1}{2n} L_x)$ where $L_x := \max_{1\leq k\leq \deg(x)} \mathcal{E}_{x,k}$.

Henceforth, $\{\zeta_{n,z} : z \in \bbZ^2\}$ being $4$-dependent, it is sufficient to prove that
\begin{equation}
\label{Item2Good}
\lim_{n \to \infty} \bbP (0 \mbox{ is $n$-good}) = 1
\end{equation}
and to apply once again the stochastic domination result of Liggett et al \cite[Theorem 0.0]{liggett1997domination} to establish that the site percolation process $\{\zeta_{n,z} : z \in \bbZ^2\}$ percolates, which concludes the proof of Theorem \ref{thm:mainResult}, Item $2$.

Since $\bbP(\calD_{n})$ tends to $1$ by Lemma \ref{lem:StabExpo}, we have to show that the probability of $\mathcal{C}_{n} := \mathcal{C}_{n}(0)$ also tends to $1$ in order to get (\ref{Item2Good}):

\begin{lem}
\label{lem:Cn}
The probability of the event $\mathcal{C}_{n}$ tends to $1$ as $n$ tends to infinity.
\end{lem}

\begin{proof}
The event $\mathcal{C}_{n}^\complement$ means the existence of a path in $\Sigma_{1/n}$ joining $S_n$ to $S_{3n}$. By translation invariance, this gives
\[
\bbP \big( \mathcal{C}_{n}^\complement \big) \leq 4n \times \bbP \big( S_1 \leftrightarrow S_{n} \; \mbox{in $\Sigma_{1/n}$} \big)
\]
so that it suffices to prove that $\bbP(S_1 \leftrightarrow S_{n} \; \mbox{in $\Sigma_{1/n}$})$ is a $o(1/n)$. Let us restrict the random set $\Sigma_{1/n}$ to grains centered at Poisson points in $\Lambda_n := B_{n} \oplus B(0,n)$:
\[
\Sigma_{1/n}[\Lambda_n] := \bigcup_{x \in \bfX \cap \Lambda_n} \Star(x,1/n)
\]
(in which grains $\Star(x,1/n)$ are still constructed from the whole Delaunay triangulation $\bfT$). Then, by Lemma \ref{lem:StabExpo},
\begin{eqnarray*}
\bbP \big( S_1 \leftrightarrow S_{n} \; \mbox{in $\Sigma_{1/n}$} \big) & \leq & \bbP ( R(B_{n}) > n ) + \bbP \big( S_1 \leftrightarrow S_{n} \; \mbox{in $\Sigma_{1/n}[\Lambda_n]$} \big) \\
& = & \bbP \big( S_1 \leftrightarrow S_{n} \; \mbox{in $\Sigma_{1/n}[\Lambda_n]$} \big) + o(1/n) ~.
\end{eqnarray*}
It is well known that the maximal degree in the Delaunay triangulation $\bfT$ among vertices inside $\Lambda_n$ is smaller than $C \log(n)$ with probability tending to $1$ exponentially fast. See for instance Bonnet and Chenavier \cite{BC}. We then get that the probability $\bbP(S_1 \leftrightarrow S_{n} \; \mbox{in $\Sigma_{1/n}$})$ is smaller than
\begin{equation}
\label{DegLogn}
\bbP \Big( \{ S_1 \leftrightarrow S_{n} \; \mbox{in $\Sigma_{1/n}[\Lambda_n]$} \} \, \cap  \bigcap_{x \in \bfX \cap \Lambda_n} \{\deg(x) \leq C \log(n)\} \Big)
\end{equation}
up to a term $o(1/n)$. When $\deg(x) \leq C \log(n)$, the grain $\Star(x,1/n)$ is included in the ball $B(x,L_n)$ where
\[
L_n := \frac{1}{2n} \max_{1\leq k\leq C\log(n)} \mathcal{E}_{x,k} ~.
\]
and whose distribution satisfies Lemma \ref{lem:Cn} stated below. Henceforth the probability (\ref{DegLogn}) is smaller than $\bbP( S_1 \leftrightarrow S_{n} \; \mbox{in $\textrm{Bool}(n)$})$ where $\textrm{Bool}(n)$ denotes the Poisson Boolean model defined by
\[
\textrm{Bool}(n) := \bigcup_{x \in \bfX} B(x , L_{n,x})
\]
in which the r.v.'s $(L_{n,x})_{x \in \bfX}$ are i.i.d. copies of $L_n$.

To sum up,
\[
\bbP \big( S_1 \leftrightarrow S_{n} \; \mbox{in $\Sigma_{1/n}$} \big) = \bbP \big( S_1 \leftrightarrow S_{n} \; \mbox{in $\textrm{Bool}(n)$} \big) + o(1/n) ~.
\]

The expected volume of each ball in $\textrm{Bool}(n)$ is $\pi \mathbb{E}[L_n^2]$. It is well known (since Hall \cite{Hall}) that the cluster of $B_1$ in $\textrm{Bool}(n)$ (and the number of balls that this cluster contains) is stochastically dominated by a Galton-Watson tree whose the mean number of children is of order $\mathbb{E}[L_n^2]$. Since this expectation can be made as small as we want as $n \to \infty$ (Lemma \ref{lem:Ln}), the dominating Galton-Watson tree will be subcritical for $n$ large enough. In this case, its total progeny (i.e. its total number of elements) admits an exponential tail decay (see the end of the first chapter in \cite{AN}). So the same holds for the number of elements belonging to the cluster of $S_1$ in $\textrm{Bool}(n)$. Combining with the fact that the r.v. $L_n$ admits also an exponential tail decay (Lemma \ref{lem:Ln}), we conclude that $\bbP(S_1 \leftrightarrow S_{n} \; \mbox{in $\textrm{Bool}(n)$})$ converges to $0$ exponentially fast. This achieves the proof of Lemma \ref{lem:Cn}.
\end{proof}

It then remains to show:

\begin{lem}
\label{lem:Ln}
For any integer $n>0$ and any real number $t>0$,
\[
\bbP( L_n > t) \leq C \log(n) e^{-2 n t} \; \mbox{ and } \; \lim_{n\to\infty} \mathbb{E}[L_n^2] = 0 ~.
\]
\end{lem}

\begin{proof}
By definition of the r.v. $L_n$, we can write
\[
\bbP( L_n > t) = 1 - \bbP \Big( \max_{k \leq C\log(n)} \mathcal{E}_{k} \leq 2nt \Big) = 1 - \big( 1 - e^{-2nt} \big)^{\lfloor C\log(n) \rfloor}
\]
where $\lfloor \cdot \rfloor$ denotes the integer part and $(\mathcal{E}_{k})_{k\geq 1}$ is a sequence of i.i.d. exponential r.v.'s with rate $1$. Thus, using the inequality $\log(1-u) \leq -u$ valid for any $u\in[0,1)$, we obtain the searched inequality for $\bbP( L_n > t)$. Moreover,
\[
\mathbb{E}[L_n^2] = 2 \int_0^\infty t \, \bbP( L_n > t) \, dt \leq 2 C \log(n) \int_0^\infty t e^{-2nt} \, dt \leq \frac{C \log(n)}{2 n^2}
\]
which tends to $0$ as $n\to\infty$.
\end{proof}

Let us notice that the proof of Theorem \ref{thm:mainResult}, Item $2$ should be significantly simpler if two given users $x,y \in \bfX_1 \cup \bfY$ (recall that $p=1$ here) on the same edge were connected iff $|x - y| \leq r$ as in \cite{le2020crowd,le2021continuum}. Indeed, in this case, the connectivity graph $\bdcalG_{1,0,r}$ could be immediately compared to a Poisson Boolean model with deterministic radius $r$ that it suffices to choose small enough to conclude.

\section{Proof of Theorem \ref{thm:mainResult}, Item $4$}
\label{sect:Item4}

In this section, we prove Item~$4$ claiming that the hypothetical region \Romanbar{2} does not exist outside the curve~$r \mapsto p_c(r)$, what is equivalent to say that its interior is empty. As sketched in the introduction, such result is obtained by comparing in some sense the partial derivatives w.r.t. parameters~$r$ and~$\lambda$ of the probability~$\bbP(S_1 \leftrightarrow \infty \text{ in $\bdcalG_{p,\lambda,r}$})$ that an infinite component in~$\bdcalG_{p,\lambda,r}$ exists and intersects the ball~$S_1$. In fact, for the sake of simplicity, we deal in this section with a slightly modified model, namely the graph only made of the streets of the PDT~$\bfT$ which are \textit{entirely} contained in~$\bdcalG_{p,\lambda,r}$. We denote it by~$\tilde{\bdcalG}_{p,\lambda,r}$. The operation performed to get~$\tilde{\bdcalG}_{p,\lambda,r}$ from~$\bdcalG_{p,\lambda,r}$ is a pruning, unable to break---if it exists---an infinite component. So the exact values of~$\bbP(S_1 \leftrightarrow \infty \text{ in $\tilde{\bdcalG}_{p,\lambda,r}$})$ may differ from~$\bbP(S_1 \leftrightarrow \infty \text{ in $\bdcalG_{p,\lambda,r}$})$, but both are either null or strictly positive at once. Otherwise said:
\[ \bbP(S_1 \leftrightarrow \infty \text{ in $\tilde{\bdcalG}_{p,\lambda,r}$})=0  \text{ if and only if } \bbP(S_1 \leftrightarrow \infty \text{ in $\bdcalG_{p,\lambda,r}$})=0. \]   
Our strategy now essentially rests on the following lemma, stating the heralded comparison, for finite approximations of the event~$\{S_1 \leftrightarrow \infty\}$:

\begin{prop} 
\label{prop:comparisonpartialderivatives}
Let~$(p,\lambda,r) \in (1/2,1]\times [0,+\infty) \times (0,+\infty)$ and~$n \geq 1$ an integer. We set~$\Theta_n (\lambda,r)=\bbP(S_1 \leftrightarrow S_{n} \; \text{ in $\tilde{\bdcalG}_{p,\lambda,r}$})$. The partial derivatives~$\partial_\lambda \Theta_n$ and~$\partial_r \Theta_n$ exists and are positive. For some continuous map~$C : (0,+\infty) \mapsto (0,+\infty)$ independent of~$n$, as well as of~$p$ and~$\lambda$, it holds that: 
\begin{align}
\label{eq:omparisonpartialderivatives}
    \partial_\lambda \Theta_n \leq C(r) e^{\frac{\lambda r}{2}} \partial_r \Theta_n.
\end{align}
\end{prop} 

Directional differentiability of~$\Theta_n$ and the positivity assertion will be the subject of a specific proposition below. The inequality~\eqref{eq:omparisonpartialderivatives} is enough to get Item 4. 
\begin{proof}[Proof of Item 4]
   A nonempty interior means that for some~$r_1 <r_2$ and $p < p_c(r)$, we have~$\bbP(S_1 \leftrightarrow \infty \text{ in $\tilde{\bdcalG}_{p,\epsilon,r_1}$})>0$ for any~$\epsilon>0$, while~$\bbP(S_1 \leftrightarrow \infty \text{ in $\tilde{\bdcalG}_{p,0,r_2}$})=0$. The finite-increments formula however ensures that
\[\Theta_n (0,r_2)-\Theta_n(\epsilon,r_1) = (r_2-r_1) \cdot \partial_r \Theta_n (\lambda^*,r^*) - \epsilon \cdot \partial_\lambda \Theta_n(\lambda^*,r^*), \]
for some couple~$(\lambda^*,r^*) \in [0,\epsilon] \times [r_1,r_2] $. Proposition~\ref{prop:comparisonpartialderivatives} then implies that:
\[\Theta_n (0,r_2)-\Theta_n(\epsilon, r_1) \geq  \left(r_2-r_1 - \epsilon e^{\epsilon r_2} \sup_{r \in [r_1,r_2]} C(r)\right)  \cdot \partial_r \Theta_n (\lambda^*,r^*) . \]
Since~$\partial_r \Theta_n \geq 0$, we obtain~$\Theta_n (0,r_2)-\Theta_n(\epsilon, r_1)\geq 0$ as~$\epsilon$ is small enough. By letting~$n \to +\infty$, it leads to a contradiction:
\[0=\bbP(S_1 \leftrightarrow \infty \text{ in $\tilde{\bdcalG}_{p,0,r_2}$})=\lim_{n \to +\infty}  \Theta_n (0,r_2) \geq \lim_{n \to +\infty}  \Theta_n (\epsilon, r_1) = \bbP(S_1 \leftrightarrow \infty \text{ in $\tilde{\bdcalG}_{p,\epsilon,r_1}$})>0,\] 
which completes the proof.
\end{proof}

\subsection{Proof of Proposition~\ref{prop:comparisonpartialderivatives}}
\label{subsec:proofcomparisonpartialderivatives}

Proving Proposition~\ref{prop:comparisonpartialderivatives} requires two ingredients. First, a thorough study of a finite one-dimensional Boolean model, in which we show that the partial derivatives of the probability to fully cover the segment satisfy an inequality akin to~\eqref{eq:omparisonpartialderivatives}. The second piece is made of Russo's type formulas which relate~$\partial_\lambda \Theta_n$ and~$\partial_r \Theta_n$ to the set of edges in the PDT, being \textit{pivotal} for the occurrence of the event~$\{S_1 \leftrightarrow S_{n}\}$. We detail the arguments in the two paragraphs below and explain how the whole yields Proposition~\ref{prop:comparisonpartialderivatives}. \\

\noindent
\textbf{A finite one-dimensional Boolean model.} Let~$\ell >0$ and~$N$ a natural number. We set~$(\bfU_k)_{0 \leq k \leq N+1}$ a finite sequence of points such that:
\begin{itemize}
    \item the first one $\bfU_0$ is the left extremity of the segment~$[0,\ell]$, that is~0;
    \item the~$N$ next ones~$\bfU_1, \dotsc, \bfU_N$ are uniformly and independently drawn on the same segment;
    \item finally, the last one $\bfU_{N+1}$ is the right extremity, that is~$\ell$. 
\end{itemize}
Attach to them~$N+2$ i.i.d. random variables~$\bdcalE_0, \dotsc, \bdcalE_{N+1}$ with common distribution~$\text{Exp}(1)$. Our Boolean model is then defined as follows. 
\begin{itemize}
    \item Every \textit{internal point}~$1 \leq k \leq N$ is assumed to cover the area centered at it and of radius~$\bfR_j^r:=\frac{r}{2} \bdcalE_j$, that is the segment~$[\bfI_j^r,\bfS_j^r]:=[\bfU_j-\bfR_j^r,\bfU_j+\bfR_j^r]$;
    \item the expected range of \textit{boundary points} is doubled, so that we respectively have
    \[[\bfI_0^r,\bfS_0^r]=[-r \bdcalE_0,r \bdcalE_0]=[-\bfR_0^r,\bfR_0^r]\]
    and
    \[[\bfI_{N+1}^r,\bfS_{N+1}^r]=[\ell-r \bdcalE_{N+1},\ell+r \bdcalE_{N+1}]=[\ell-\bfR_{N+1}^r,\ell+\bfR_{N+1}^r].\]
\end{itemize}
We are interested in the probability 
\begin{align}
\label{eq:probafullcoveringsegment}
  \bfp(\ell,\lambda,r):= \bbP\left([0,\ell] \subseteq \cup_{j=0}^{\bfN+1} [\bfI_j^r,\bfS_j^r]\right),   
\end{align}
that the surface covered by points includes the segment~$[0,\ell]$, in the case where~$\bfN$ is a Poisson random variable of parameter~$\lambda \ell$, independently drawn from the locations of points and their range. It is indeed equal, conditionally on the realization of~$\bfX_p$, to the probability that~$\tilde{\bdcalG}_{p,\lambda,r}$ entirely contains a given street of length~$\ell$, a quantity that will appear as critical in the next paragraph. We present now a crucial result on the way taking us to Proposition~\ref{prop:comparisonpartialderivatives}:

\begin{prop}
\label{prop:coveringasegment}
The partial derivative~$\partial_\lambda \bfp$ and~$\partial_r \bfp$ both exist and are positive. Also, for any~$r>0$ and~$\lambda \geq 0$:
\begin{align}
\label{eq:boundednesspartialderivativesonedimodel}
\max\Bigg\{\sup_{\ell>0} \partial_\lambda \bfp (\ell,\lambda, r), \ \sup_{\ell>0} \partial_r \bfp (\ell,\lambda, r) \Bigg\}<+\infty
\end{align}
Finally, for some continuous map~$C: (0,+\infty) \mapsto (0,+\infty)$ independent of~$\ell$ and~$\lambda$, it holds that:
\begin{align}
\label{eq:ineqpartialderivativesonedimodel}
    \partial_\lambda \bfp \leq C(r) e^{\frac{\lambda r}{2}} \partial_r \bfp.
\end{align}
%; \quad (ii) \ \partial_\lambda \bfp \leq \ell \quad \text{and} \quad (iii) \ \partial_r\bfp \leq c \ell^2 \text{(TO CHECK)}.\]
\end{prop}

The proof merely consists of an in-depth but elementary analysis of the function~$\bfp$, that we postpone to Section~\ref{subsec:analysisprobacoversegment}. Developments outlined there are not helpful to derive Proposition~\ref{prop:comparisonpartialderivatives}, statements~\eqref{eq:boundednesspartialderivativesonedimodel} and~\eqref{eq:ineqpartialderivativesonedimodel} being all we need at this stage for such purpose. They are rather aimed at readers eager to fully grasp the underlying reasons of some hypothesis made in our model---random and exponentially-distributed ranges, with a higher mean for users at crossroads---and how the latter could be relaxed. Recall for instance, as said in the introduction, that all our results remain true by setting the connection radius of relays at more than twice that of users on streets; in this section, by more than doubling the mean area that cover boundary points, compared to the internal ones. See Section~\ref{sec:finaldiscussion} about the other conjectures. \\

\noindent
\textbf{Russo's formulas.} In this paragraph, we work with a more convenient representation of the graph~$\tilde{\bdcalG}_{p,\lambda,r}$. Let~$(\bdcalU_{y,y'})_{y,y' \in \bbR}$ be a sequence of i.i.d. random variables, uniformly distributed on~$[0,1]$. Label as \textit{open} the edges~$\{x,x'\}$ of the PDT~$\bfT$ which obey two conditions:
\begin{enumerate}
    \item the crossroads~$x,x'$ flanking it both belong to~$\bfX_p$; 
    \item the companion random variable~$\bdcalU_{x,x'}$ satisfies the inequality~$\bdcalU_{x,x'}\leq \bfp(\| x-x' \|,r,\lambda)$.
\end{enumerate} 
We claim that the graph built from the open edges of~$\bfT$, denoted by~$\overline{\bdcalG}_{p,\lambda,r}$, is distributed as~$\tilde{\bdcalG}_{p,\lambda,r}$:
\begin{align}
\label{eq:distributionequalityprunedgraph}
    \overline{\bdcalG}_{p,\lambda,r}\overset{(d)}{=}\tilde{\bdcalG}_{p,\lambda,r}.
\end{align}
Such alternative representation allows us to make rigorous the key notion of \textit{pivotal} edge:

\begin{dfn}
\label{dfn:pivotaledge}
    An edge~$\{x,x'\}$ of~$\bfT$ is said to be \emph{pivotal} for the event~$\{S_1 \leftrightarrow S_{n}\}$ if the latter occurs in~$\overline{\bdcalG}_{p,\lambda,r} \cup (x,x')$, but does not in~$\overline{\bdcalG}_{p,\lambda,r}\!\setminus\!(x,x')$. 
\end{dfn}

Note that a pivotal edge necessarily intersects the box~$B_n$. Their total number is hence almost surely finite, and even integrable given the exponential decay of the stabilization radius~$R(B_n)$ in~PDT. See Lemma~\ref{lem:StabExpo}. 

A Russo-type formula affirms that the local growth rate of an event's probability is all the more greater as the number of edges being pivotal for it is high. This is exactly what we observe for~$\Theta_n$:

\begin{prop}
    \label{prop:russoformulas}
    For any triplet~$(p,\lambda,r)$, we have:
    \[\partial_\lambda \Theta_n = \bbE\Bigg[\sum_{\{x,x'\} \text{ pivotal}} (\partial_\lambda \bfp) (\| x-x' \|,\lambda,r)\Bigg] \text{ and } \partial_r \Theta_n = \bbE\Bigg[\sum_{\{x,x'\} \text{ pivotal}} (\partial_r \bfp) (\| x-x' \|,\lambda,r)\Bigg].\]
\end{prop}

\begin{proof}[Proof of Proposition~\ref{prop:russoformulas}]
    The demonstration is quite standard and is the same for both equalities. We focus on the first one. It essentially stems from a \textit{quenched} version of the Russo's formula. For any~$h>0$, set
\[\Delta_h \Theta_n^{|\bfX,\bfX_p}(\lambda,r):=\bbP(S_1 \leftrightarrow S_{n} \; \text{ in $\overline{\bdcalG}_{p,\lambda+h,r}$} | \bfX, \bfX_p)-\bbP(S_1 \leftrightarrow S_{n} \; \text{ in $\overline{\bdcalG}_{p,\lambda,r}$} | \bfX, \bfX_p).\] 
Then, almost surely:
\begin{align}
\label{eq:quenchedrussoformula}
\lim_{h \to 0^+} \frac{\Delta_h \Theta_n^{|\bfX,\bfX_p}(\lambda,r)}{h}=\sum_{\{x,x'\} \text{ pivotal}} (\partial_\lambda \bfp) (\| x-x' \|,\lambda,r).
\end{align}
The difference~$\Delta_h \Theta_n^{|\bfX,\bfX_p}(\lambda,r)$ is indeed equal to the probability that the event~$\{S_1 \leftrightarrow S_{n}\}$ occurs in~$\overline{\bdcalG}_{p,\lambda+h,r}$, but does not in~$\overline{\bdcalG}_{p,\lambda,r}$, conditionally on~$\bfX$ and~$\bfX_p$. This happens only if at least one edge~$\{x,x'\}$ of the PDT has been opened by increasing the parameter~$\lambda$ to~$\lambda+h$, meaning that
\[\bfp(\| x-x' \|, \lambda,r)<\bdcalU_{x,x'} \leq \bfp(\| x-x' \|,\lambda+h,r).\] 
As~$h \to 0^+$, because of the mutual independence between the random variables~$\bdcalU_{y,y'}$, there cannot be more than one such edge. So there is exactly one, which is furthermore pivotal since its opening goes with the occurrence of~$\{S_1 \leftrightarrow S_{n}\}$. Rephrased in mathematical terms:
\begin{align*}
    \Delta_h \Theta_n^{|\bfX,\bfX_p}(\lambda,r)&=\sum_{\{x,x'\} \text{ pivotal}} \bbP\Big(\bfp(\| x-x' \|, \lambda,r)<\bdcalU_{x,x'} \leq \bfp(\| x-x' \|,\lambda+h,r)\Big) \quad + \quad o(h) \\ 
    &=\sum_{\{x,x'\} \text{ pivotal}} \Big[\bfp(\| x-x' \|,\lambda+h,r)-\bfp(\| x-x' \|,\lambda,r)\Big] \quad + \quad o(h).
\end{align*}
Thus the pointwise limit~\eqref{eq:quenchedrussoformula}. Finally, we use Lebesgue's theorem for upgrading the latter to the expected Russo's formula, thanks to the following domination, which is true for some constant~$K=K_{\lambda,r}>0$ provided by~\eqref{eq:boundednesspartialderivativesonedimodel}:
\[ \sum_{\{x,x'\} \text{ pivotal}} (\partial_\lambda \bfp) (\| x-x' \|,\lambda,r) \leq K \times \# \text{ pivotal edges} ,\]
where~$\# \text{ pivotal edges}$ is an integrable random variable as explained earlier.
\end{proof}

\noindent
\textbf{Proof of Proposition~\ref{prop:comparisonpartialderivatives}.} Once we have in hand Propositions~\ref{prop:coveringasegment} and~\ref{prop:russoformulas}, the proof of Proposition~\ref{prop:comparisonpartialderivatives} is immediate. Indeed:
\begin{align*}
    \partial_\lambda \Theta_n &\underset{\text{Prop.~\ref{prop:russoformulas}}}{=} \bbE\Bigg[\sum_{\{x,x'\} \text{ pivotal}} (\partial_\lambda \bfp) (\| x-x' \|,\lambda,r)\Bigg] \nonumber \\
    &\underset{\text{Prop.~\ref{prop:coveringasegment}}}{\leq} \bbE\Bigg[\sum_{\{x,x'\} \text{ pivotal}} C(r) e^{\frac{\lambda r}{2}} (\partial_r \bfp) (\| x-x' \|,\lambda,r)\Bigg],
    \end{align*}
   for some continuous map~$C: (0,+\infty) \mapsto (0,+\infty)$. Hence:
   \[\partial_\lambda \Theta_n \leq C(r) e^{\frac{\lambda r}{2}} \times \bbE\Bigg[\sum_{\{x,x'\} \text{ pivotal}}(\partial_r \bfp) (\| x-x' \|,\lambda,r)\Bigg] \underset{\text{Prop.~\ref{prop:russoformulas}}}{=} C(r) e^{\frac{\lambda r}{2}} \partial_r \Theta_n,\]
   as it has been heralded.

\subsection{Proof of Proposition~\ref{prop:coveringasegment}}
\label{subsec:analysisprobacoversegment}
This section is devoted to prove Proposition~\ref{prop:coveringasegment}, which enumerates several properties of the function~$\bfp$ defined by~\eqref{eq:probafullcoveringsegment}. We focus, more specifically, on its partial derivatives w.r.t. variables~$\lambda$ and~$r$. All the arguments used in the demonstration are nothing but basic analysis. We start by merely establishing the existence of~$\partial_\lambda \bfp$ and~$\partial_r \bfp$. Amenable expressions are derived at once, as well as their positivity.  \\

\noindent
\textbf{Existence, positivity and expression of~$\partial_\lambda \bfp$.} Let~$h>0$. Given the superposition property of the Poisson distribution, the difference~$\bfp(\ell,\lambda+h,r)-\bfp(\ell,\lambda,r)$ can be rewritten as follows:
\[\bfp(\ell,\lambda+h,r)-\bfp(\ell,\lambda,r)=\bbP\left([0,\ell] \not\subseteq \cup_{j=0}^{\bfN+1} [\bfI_j^r,\bfS_j^r]; \ [0,\ell] \subseteq \left( \cup_{j=0}^{\bfN+1} [\bfI_j^r,\bfS_j^r]\right) \cup \left( \cup_{j=1}^{\tilde{\bfN}} [\tilde{\bfI}_j^r,\tilde{\bfS}_j^r] \right) \right),  \]
where~$\tilde{\bfN}$ is a Poisson random variable of parameter~$h \ell$, independent of~$\bfN$, while~$\cup_{j=1}^{\tilde{\bfN}} [\tilde{\bfI}_j^r,\tilde{\bfS}_j^r]$ is the surface spanned by~$\tilde{\bfN}$ additional points, drawn on the segment~$[0,\ell]$ in the same way as the~$\bfN$ former ones, and independently of them. Otherwise said, by increasing the intensity~$\lambda$, we give birth to new points, what provides a chance to get rid of existing blank zones. As~$h \to 0$, the probability to see arising more than two is however of order~$o(h)$, so that:
\begin{align}
&\bfp(\ell,\lambda+h,r)-\bfp(\ell,\lambda,r) \nonumber \\
&= \underset{=\bbP(\tilde{\bfN}=1)}{\underbrace{h \ell e^{- h \ell} }} \times \bbP\left([0,\ell] \not\subseteq \cup_{j=0}^{\bfN+1} [\bfI_j^r,\bfS_j^r]; \ [0,\ell] \subseteq \left( \cup_{j=0}^{\bfN+1} [\bfI_j^r,\bfS_j^r]\right) \cup [\tilde{\bfI}_1^r,\tilde{\bfS}_1^r] \right) + o(h). \nonumber 
\end{align}
The partial derivative~$\partial_\lambda \bfp$ thus exists for any triplet~$(\ell,\lambda,r)$ and:
\begin{align}
\label{eq:partialderivativelambdaexpression}
    \partial_\lambda \bfp (\ell,\lambda,r) = \ell \times \bbP\left([0,\ell] \not\subseteq \cup_{j=0}^{\bfN+1} [\bfI_j^r,\bfS_j^r]; \ [0,\ell] \subseteq \left( \cup_{j=0}^{\bfN+1} [\bfI_j^r,\bfS_j^r]\right) \cup [\bfI^r,\bfS^r] \right) , 
\end{align}
with~$[\bfI^r,\bfS^r]:=[\bfU-\bfR,\bfU+\bfR]$ the interval that a point~$\bfU$, uniformly drawn on~$[0,\ell]$, typically covers as its range~$\bfR:=\frac{r}{2}\bdcalE$ is exponentially distributed with mean~$r/2$. Positivity of~\eqref{eq:partialderivativelambdaexpression} is trivial. \\

\noindent
\textbf{Existence, positivity and expression of~$\partial_r \bfp$.}  By increasing~$r$, we do not act this time on the number~$\bfN$ of points drawn on the segment. We rather enlarge the area that they respectively span. The expression of the increment~$\bfp(\ell,\lambda,r+h)-\bfp(\ell,\lambda,r)$ as the probability of some event is then for any~$h>0$:
\begin{align}
\label{eq:expressionasaprobaincrementwrtr}
    \bfp(\ell,\lambda,r+h)-\bfp(\ell,\lambda,r)=\bbP\left([0,\ell] \not\subseteq \cup_{j=0}^{\bfN+1} [\bfI_j^r,\bfS_j^r]; \ [0,\ell] \subseteq \cup_{j=0}^{\bfN+1} [\bfI_j^{r+h},\bfS_j^{r+h}]\right).
\end{align}
A hole in the coverage is necessarily of the form~$\left(\bfS_{i}^r,\bfI_j^r\right)$, where~$0 \leq i \neq j \leq \bfN+1$ are such that~$\bfS_i^r<\bfI_j^r$ and~$\left(\bfS_{i}^r,\bfI_j^r\right)\cap\left[\bfI_k^r,\bfS_k^r\right]=\emptyset$ for any~$k \notin \{i,j\}$. Filling it by stretching extremities of intervals entails the existence of some other~$0 \leq m \neq q \leq \bfN+1$ satisfying both~$\bfS_m^r \leq \bfS_i^r < \bfI_j^r \leq \bfI_q^r$ and~$\bfI_q^{r+h} \leq \bfS_i^r < \bfI_j^r \leq \bfS_m^{r+h}$. In particular, we have~$\left|\bfI_q^r -\bfS_m^r \right| \leq \frac{h}{2} \left(\bdcalE_m + \bdcalE_q\right)$, or equivalently:
\[\bdcalE_m + \bdcalE_q \in \left[2 \frac{\bfU_q-\bfU_m}{r+h},2 \frac{\bfU_q-\bfU_m}{r-h}\right].\]
The probability of such event is of order~$h$ as~$h \to 0$. That of the whole picture just described is a~$\calO(h^2)$ as~$m \neq i$ or~$q \neq j$, because the random variables~$\bfU_0, \dotsc, \bfU_{\bfN+1}, \bdcalE_0, \dotsc, \bdcalE_{\bfN+1}$ are mutually independent, conditionally on~$\bfN$. This is all the more true as several distinct holes coexist. 

Return now to a framework where the number of points spread over~$[0,\ell]$ is a deterministic integer~$N \geq 0$. We introduce two key functions. For any~$0 \leq i \neq j \leq N+1$ and any~$0 \leq x \leq y \leq \ell$, we set
\begin{align}
\label{eq:defphi}
    \Phi_{i,j}^{N,r}(x,y)= \bbP\left(\forall k \notin \{i,j \}, \ [\bfI_{k}^r,\bfS_{k}^r]\cap [\bfS_i^r,\bfI_j^r] =\emptyset \right) = \prod_{\underset{k \neq i,j}{k=0}}^{N+1}\mathbb{P}\big([\bfI_{k}^r,\bfS_{k}^r]\cap [x,y] =\emptyset \big).
\end{align}
We also define:
\begin{align}
\label{eq:expressionpartialderivativeswrtrdecompos}
   \calR_{i,j}^{N,r}(h) &= \bbP\left(\forall k \notin \{i,j \}, \ [\bfI_{k}^r,\bfS_{k}^r]\cap [\bfS_i^r,\bfI_j^r] =\emptyset; \ \bfS_i^r < \bfI_j^r, \ \bfI_j^{r+h} \leq \bfS_i^{r+h}  \right) \nonumber \\
   &=\bbE\left[ \Phi(\bfS_i^r,\bfI_j^r) \mathds{1}_{\bfS_{i}^r < \bfI_{j}^r } \mathds{1}_{\bfS_{i}^{r+h}\geq \bfI_{j}^{r+h}}\right].  
\end{align}
Therefore, given~\eqref{eq:expressionasaprobaincrementwrtr} and what has been said about holes, we derive that:
\begin{align}
\label{eq:incrementwrtapprox}
    \bfp(\ell,\lambda,r+h)-\bfp(\ell,\lambda,r) = \bbE\left[\sum_{0 \leq i \neq j \leq \bfN +1 } \calR_{i,j}^{\bfN,r} (h) + o(h) \right].
\end{align}
The existence of~$\partial_r \bfp$ then results from successive applications of Lebesgue's theorems. 
Fix~$1 \leq i \neq j \leq N$, two internal points. Let~$\bfs_i^r$ and~$\bfi_j^r$ be the joint density of~$(\bfS_i^r, \bfR_i^r)$ and~$(\bfI_j^r,\bfR_j^r)$ respectively. It holds that:
    \begin{align}
        \calR_{i,j}^{N,r}(h)  &= \int_{[0,\ell]^4} \Phi(u,x) \bfs_i^r (u,v) \bfi_j^r(x,y) \mathds{1}_{u\leq x } \mathds{1}_{u+\frac{\ee}{r} v \geq x - \frac{\ee}{r} y} \mathds{1}_{v+y \leq \ell}  \ \mathrm{d}u \ \mathrm{d}v \ \mathrm{d}x \ \mathrm{d}y \nonumber \\
        &= \int_{[0,\ell]^3} \mathrm{d}u \ \mathrm{d}v \ \mathrm{d}y \ \bfs_i^r (u,v) \mathds{1}_{v+y \leq \ell} \left(\int_{u}^{u+\frac{\ee}{r} (v+y)} \Phi(u,x)  \bfi_j^r(x,y) \ \mathrm{d}x\right).  \nonumber
    \end{align}
    Since~$\Phi_{i,j}^{N,r}$ is bounded, the Lebesgue's dominated convergence theorem, coupled with the Lebesgue differentiation theorem, implies that~$\calR_{i,j}^{N,r}(h)$ has a finite limit as~$h \to 0$:
    \begin{align}
    \label{eq:integralrepresentationR}
        \frR_{i,j}^{N,r} := \lim_{h \to 0}  h^{-1} \calR_{i,j}^{N,r}(h) = \frac{1}{r} \int_{[0,\ell]^3} (v+y) \Phi(u,u) \bfs_i^r(u,v) \bfi_j^r(u,y) \mathds{1}_{v+y \leq \ell}  \ \mathrm{d}u \ \mathrm{d}v \ \mathrm{d}y.
    \end{align}
    Note that the value of the above quantity does not vary with the pair~$i \neq j$, so is constantly equal to~$\frR_{1,2}^{N,r}$. In cases where either~$i$ or~$j$ is a boundary point, the associated joint density degenerates, because~$\bfS_0^r=\bfR_0^r$ and~$\bfI_{N+1}^r=\ell-\bfR_{N+1}^r$. This a however a false problem. Based on similar arguments, the limit~\eqref{eq:integralrepresentationR} still exists. The integral representation must be modified, though. For~$i=0$ and~$1 \leq j \leq N$---that is a mixed situation with one boundary point and one internal point, we get in the same vein that:
    \begin{align}    \label{eq:integralrepresentationbisR}
        \frR_{0,j}^{N,r} = \frac{1}{r} \int_{[0,\ell]^2} \Phi(u,u) (u+y) \bfs^r(u) \bfi_j^r(u,y)  \mathds{1}_{u+y \leq \ell} \ \mathrm{d}u \ \mathrm{d}y,
    \end{align}
    with~$\bfs^r$ the density of~$\bfS_0^r$. Like~\eqref{eq:integralrepresentationR}, the latter formula does not depend on~$j$, meaning that~$\frR_{0,j}^{N,r}=\frR_{0,1}^{N,r}$ for every~$j$. Furthermore, given the model is invariant in distribution after reflecting points over the vertical line~$x=1/2$, it does not change anything by taking instead~$i=N+1$ and~$1 \leq j \leq N$.   
    Finally, when~$i=0$ and~$j=N+1$---both are an extremity of the segment~$[0,\ell]$, we obtain that 
\begin{align}
\label{eq:integralrepresentationterR}
        \frR_{0,N+1}^{N,r}  = \frac{\ell}{r} \int_{[0,\ell]} \Phi(u,u) \bfs^r(u) \bfi^r(u)  \ \mathrm{d}u,
    \end{align}
    with~$\bfi^r$ the density of~$\bfI_{N+1}^r$. From now on, considering the observations just made, we simply write~$\frR_{\frb}^{N,r}$ for~$\frb \in \left\{0,1,2\right\}$, rather than~$\frR_{i,j}^{N,r}$ as~$\frb$ boundary points are among~$i$ and~$j$. 
    From all the foregoing, we deduce the almost sure convergence:
    \begin{align}
    \label{eq:aslimitR}
      \lim_{h \to 0} \frac{1}{h} \sum_{0 \leq i \neq j \leq \bfN +1 } \calR_{i,j}^{\bfN,r} =  \sum_{0 \leq i \neq j \leq \bfN +1 } \frR_{i,j}^{\bfN,r}=2  \frR_{2}^{\bfN,r} + 4 \bfN \frR_{1}^{\bfN,r} + \bfN (\bfN-1) \frR_{0}^{\bfN,r}.
    \end{align}
    To complete the proof, we need a suitable domination of the limit, by some integrable random variable. For this purpose, we remark that the function~$\Phi$ is the only term involved in the integrals~\eqref{eq:integralrepresentationR}, \eqref{eq:integralrepresentationbisR} or~\eqref{eq:integralrepresentationterR}, which does depend on~$N$. It can trivially be upper bounded by one, as the probability of an event---see~\eqref{eq:defphi}. Hence, for any~$N \geq 0$, we have~$\max\left\{\frR_{0}^{N,r}, \ \frR_{1}^{N,r}, \ \frR_{2}^{N,r}\right\} \leq K$,
    where~$K=K_{\ell,r}$ is the largest element among:
    \begin{align}
    \nonumber
        \bigg\{\frac{\ell}{r} \int_0^\ell \bfs^r(u) \bfi^r(u)   \mathrm{d}u, \ \frac{1}{r}\int (u+y) \bfs^r(u) \bfi_1^r(u,y)  \mathds{1}_{u+y \leq \ell} \ \mathrm{d}u \ \mathrm{d}y, \nonumber \\
        \frac{1}{r}\int_{[0,\ell]^3} (v+y)\bfs_1^r(u,v) \bfi_2^r(u,y) \mathds{1}_{v+y \leq \ell}  \ \mathrm{d}u \ \mathrm{d}v \ \mathrm{d}y \bigg\}. \nonumber 
    \end{align}
    It follows that almost surely:
\begin{align}
    \nonumber
    2  \frR_{2}^{\bfN,r} + 4 \bfN \frR_{1}^{\bfN,r} + \bfN (\bfN-1) \frR_{0}^{\bfN,r} \leq K \times \left( \bfN(\bfN-1) + 4 \bfN +2 \right),
    \end{align}
    namely the kind of domination that we were hoping for. By invoking the Lebesgue's dominated convergence theorem, we conclude from~\eqref{eq:incrementwrtapprox} and~\eqref{eq:aslimitR} that
    \begin{align}
    \label{eq:partialderivativewrtrdetailedexpression}
    \partial_{r} \bfp (\ell,\lambda, r)=\bbE\left[2  \frR_{2}^{\bfN,r} + 4 \bfN \frR_{1}^{\bfN,r} + \bfN (\bfN-1) \frR_{0}^{\bfN,r}\right],
    \end{align}
    which puts an end to the demonstration of the existence and positivity of~$\partial_r \bfp$.
    \\
    
\noindent
\textbf{Boundedness and comparison of the partial derivatives.} We eventually deal with the proof of~\eqref{eq:boundednesspartialderivativesonedimodel} and~\eqref{eq:ineqpartialderivativesonedimodel}. The formula~\eqref{eq:partialderivativelambdaexpression} relates~$\partial_\lambda \bfp$ to the probability of observing blank zones, vanishing after they are absorbed by the range of a newborn point. We aim to get a more tractable expression that mimicks~\eqref{eq:partialderivativewrtrdetailedexpression}---here it will be an upper bound actually---by using again our analysis of discontinuities in the coverage. As argued in the first lines of the previous paragraph, we know indeed that a hole is an interval of the form~$\left(\bfS_{i}^r,\bfI_j^r\right)$ with~$0 \leq i \neq j \leq \bfN+1$, such that~$\bfS_i^r<\bfI_j^r$ and~$\left(\bfS_{i}^r,\bfI_j^r\right)\cap\left[\bfI_k^r,\bfS_k^r\right]=\emptyset
$ for every~$k \notin \{i,j\}$. Then, set for any~$0 \leq x \leq y \leq \ell$:
\begin{align}
\label{eq:defpsi}
\Psi(x,y) = \bbP\left(\left[x,y\right] \subseteq \left[\bfI^r,\bfS^r\right] \right).
\end{align}
Recall that~$[\bfI^r,\bfS^r]$ is the interval that covers a point which is uniformly drawn and whose the range is exponentially distributed with mean~$r/2$. The function~$\Psi$ measures the chance that it fully contains some given area.
Let now~$N \geq 0$ be the deterministic number of points spread on the segment~$\left[0,\ell\right]$. On the model of~\eqref{eq:expressionpartialderivativeswrtrdecompos}, we define for any~$0 \leq i \neq j \leq N+1$:
\begin{align}
\label{eq:expressionpartialderivativeswrtlambdadecompos}
\frL_{i,j}^{N,r} &= \bbP\left(\forall k \notin \{i,j \}, \ [\bfI_{k}^r,\bfS_{k}^r]\cap [\bfS_i^r,\bfI_j^r] =\emptyset; \ \bfS_i^r < \bfI_j^r; \ \left( \bfS_i^r, \bfI_j^r\right) \subseteq [\bfI^r,\bfS^r]  \right) \nonumber \\
&=\bbE\left[ \Phi\left(\bfS_i^r,\bfI_j^r\right) \Psi\left(\bfS_i^r,\bfI_j^r \right) \mathds{1}_{\bfS_{i}^r < \bfI_{j}^r }\right],
\end{align}
namely the probability that the hole~$\left(\bfS_{i}^r,\bfI_j^r\right)$ is filled after a new point has arisen. 
As it has already been remarked earlier for the~$\frR_{i,j}^{N,r}$---see~\eqref{eq:integralrepresentationR}, the latter only depends in fact on how many boundary points we count, say~$\frb \in \left\{0,1,2\right\}$, among~$i$ and~$j$. So we likewise choose to abbreviate~$\frL_{i,j}^{N,r}$ as~$\frL_{\frb}^{N,r}$ in such case. Since there exists at least one hole in a sparse coverage, an union bound argument applied to~\eqref{eq:partialderivativelambdaexpression} implies that:
\begin{align}
\label{eq:partialderivativewrtrdetailedinequality}
\partial_\lambda \bfp (\ell,\lambda,r) \leq \ell \times \bbE\left[\sum_{0 \leq i \neq j \leq \bfN +1} \frL_{i,j}^{\bfN,r} \right] = \ell \times \bbE\left[2  \frL_{2}^{\bfN,r} + 4 \bfN \frL_{1}^{\bfN,r} + \bfN (\bfN-1) \frL_{0}^{\bfN,r} \right],
\end{align} 
which is the upper bound that we were looking for. Statements~\eqref{eq:boundednesspartialderivativesonedimodel} and~\eqref{eq:ineqpartialderivativesonedimodel} then result from a comparison between~$\frL_{\frb}^{N,r}$ and~$\frR_{\frb}^{N,r}$ that we detail in the following proposition:

\begin{prop}
    \label{prop:comparisonLambdaR}
    Set~$W_{\ell, r}=1-\frac{r}{2\ell}\left(1-e^{-\frac{2\ell}{r}}\right)$. Let~$\frb \in \left\{0,1,2\right\}$. There exists three functions~$\left(G_\frb^r,H_\frb^r,K_\frb^r\right)$, all positive, continuous and bounded on~$(0,+\infty)$, such that:
    \begin{align}
    \label{eq:comparisonLambdaR}
        \ell \cdot \frL_{\frb}^{N,r} \leq W_{\ell, r}^{N-2+\frb} G_\frb^r(\ell) \leq  W_{\ell, 2r}^{N-2+\frb} H_\frb^r(\ell) \leq \frR_{\frb}^{N,r} \leq K_\frb^r(\ell),
    \end{align}
    for any~$N \geq 0$. Furthermore, the first two satisfy:
    \begin{align}
    \label{eq:inequalityfunctionGH}
        G_\frb^r(\ell) \leq C_\frb(r) H_\frb^r(\ell),
    \end{align}
    for some continuous map~$C_\frb: (0,+\infty) \mapsto (0,+\infty)$ independent of~$\ell$.
\end{prop}

The proof of Proposition~\ref{prop:comparisonLambdaR} is postponed to the next paragraph. We  first show how to complete that of Proposition~\ref{prop:coveringasegment}, and start with~\eqref{eq:ineqpartialderivativesonedimodel}. The leftmost inequality in~\eqref{eq:comparisonLambdaR}, coupled with~\eqref{eq:partialderivativewrtrdetailedinequality}, ensures that:
\begin{align}
\label{eq:finalupperboundpartialderivativeswrtl}
    \partial_\lambda \bfp (\ell,\lambda,r) &\leq 2 G_2^r (\ell) \bbE\left[W_{\ell, r}^\bfN \right] + 4 G_1^r(\ell) \bbE\left[\bfN W_{\ell, r}^{\bfN-1} \right] + G_0^r(\ell) \bbE\left[\bfN\left(\bfN-1\right)\bfN W_{\ell, r}^{\bfN-2} \right] \nonumber \\
    &= 2 G_2^r(\ell)e^{\lambda \ell \left(W_{\ell, r}-1 \right)} + 4 G_1^r(\ell) \lambda \ell e^{\lambda \ell \left(W_{\ell, r}-1 \right)} + G_0^r(\ell) \lambda^2 \ell^2 e^{\lambda \ell \left(W_{\ell, r}-1 \right)}.
\end{align}
We use in the last line that the generating function of a Poisson distribution of parameter~$\lambda \ell$ is~$t \mapsto e^{\lambda \ell(t-1)}$. In exactly the same way, we derive this time from~\eqref{eq:partialderivativewrtrdetailedexpression} and from the second rightmost inequality that:
\begin{align}
\label{eq:finalinequalitiespartialderivativeswrtr}
   2 H_2^r(\ell)e^{\lambda \ell \left(W_{\ell, 2r}-1 \right)} + 4 H_1^r(\ell) \lambda \ell e^{\lambda \ell \left(W_{\ell, 2r}-1 \right)} + H_0^r(\ell) \lambda^2 \ell^2 e^{\lambda \ell \left(W_{\ell, 2r}-1 \right)}  \leq \partial_r \bfp (\ell,\lambda,r).
\end{align}
It directly follows from~\eqref{eq:finalupperboundpartialderivativeswrtl}, \eqref{eq:finalinequalitiespartialderivativeswrtr} and \eqref{eq:inequalityfunctionGH} that:
\begin{align}
    e^{\lambda \ell \left(W_{\ell, 2r}-W_{\ell, r} \right)} \times \partial_\lambda \bfp (\ell,\lambda,r) \leq C(r) \partial_r \bfp (\ell,\lambda,r), \nonumber
\end{align}
for~$C:=\max\left\{C_0,C_1,C_2\right\}$. The map~$C$ is continuous given that the~$C_\frb$ all are.
This is enough to get~\eqref{eq:ineqpartialderivativesonedimodel} since~$W_{\ell, 2r}-W_{\ell, r} \geq -r / 2 \ell$. Finally, we deduce the boundedness assertion~\eqref{eq:boundednesspartialderivativesonedimodel} from~\eqref{eq:partialderivativewrtrdetailedexpression} and the rightmost inequality in~\eqref{eq:comparisonLambdaR}. Indeed, we have:
\begin{align}
\nonumber
    \partial_r \bfp (\ell,\lambda,r) \leq \left( \max_\frb \sup_\ell K_\frb^r(\ell) \right) \times \left(\bbE\left[2 + 4 \bfN + \bfN \left(\bfN-1\right)  \right] \right).
\end{align}
The right hand side above is finite because the functions~$K_\frb^r$ are said to be bounded. So is then~$\ell \mapsto \partial_r \bfp (\ell,\lambda,r)$. The comparison~\eqref{eq:ineqpartialderivativesonedimodel} allows us to extend the result to~$\ell \mapsto \partial_\lambda \bfp (\ell,\lambda,r)$ at once, which yields the expected conclusion.  

\subsection{Proof of Proposition~\ref{prop:comparisonLambdaR}: cases investigation.}
\label{subsec:casesinvestigationproofcomparison}
Three situations arise in Proposition~\ref{prop:comparisonLambdaR}, depending on the number of boundary points which are involved in the definition of the quantities~$\frL_{i,j}^{N,r}$ and~$\frR_{i,j}^{N,r}$, those that we aim to compare. We start the demonstration with some preliminary observations. They will constitute a common guideline for the analysis of each case. 

First, we provide an exact formula for the probability that the surface covered by a point uniformly drawn on~$\left[0,\ell\right]$ does not interset some given interval:

\begin{lem}[Blank zone]
\label{lem:emptyoverlap}
Set for any~$0 \leq x \leq y \leq \ell$:
\[\varphi(x,y)=\mathbb{P}\left(\left[\bfI^r, \bfS^r\right]\cap \left[x,y\right] =\emptyset\right).\]
Then:
\begin{align}
\label{eq:blankzones}
    \varphi(x,y)=\frac{1}{\ell}\left(\ell-(y-x)-\frac{r}{2}\left(2-e^{-2 \frac{\ell-y}{r}}-e^{- \frac{2x}{r}}\right)\right).
\end{align}
\end{lem}
When the point is rather stuck at an extremity of the segment, endowed with a doubled expected range, the above formula is simpler:
\begin{align}
\label{eq:blankzonesBP}
\mathbb{P}\left(\left[-\bfR,\bfR\right]\cap \left[x,y\right] =\emptyset\right)=1-e^{-\frac{x}{r}} \quad \text{and} \quad \mathbb{P}\left(\left[\ell-\bfR,\ell+\bfR\right]\cap \left[x,y\right] =\emptyset\right)=1-e^{-\frac{\ell-y}{r}}.
\end{align}

Distributions of~$\bfI^r$ and~$\bfS^r$ being explicit, the proof of Lemma~\ref{lem:emptyoverlap} consists of the easy computation of an integral, whose details are left to the reader. %The right-hand side of~\eqref{eq:blankzones} is of course continuous w.r.t.~$x$ and~$y$. 
Since not overlapping an interval is harder than a single point, it is always true that~$\varphi(x,y) \leq \varphi(x,x)$. Furthermore, on the main diagonal, the function~$\varphi$ admits both lower and upper bounds:
\begin{cor}
   \label{cor:globalboundsemptypointoverlap}
    Let~$W_{\ell,r}$ be defined as in Proposition~\eqref{prop:comparisonLambdaR}. For any~$0 \leq x \leq \ell$, we have~$W_{\ell,2r} \leq \varphi(x,x) \leq W_{\ell,r}$. 
    \end{cor}
    Here also we do not write the details of the elementary demonstration. Note that~$W_{\ell,r}$ is even a global upper bound for~$\varphi$, given the remark made just above. As~$\ell \to 0$, it turns out that:
    \begin{align}
    \label{eq:asymptoticW}
        W_{\ell,r} \sim \ell/r.
    \end{align} 
    It implies that~$\varphi$ is of linear order w.r.t.~$\ell$ in the small length regime. Corollary~\ref{cor:globalboundsemptypointoverlap} will be useful to control the function~$\Phi$, which appears in the definitions of~$\frL_{\frb}^{N,r}$ and~$\frR_{\frb}^{N,r}$. There is indeed an obvious connection with~$\varphi$. See~\eqref{eq:defphi}. For any~$N \geq 0$ and~$0 \leq i \neq j \leq N+1$:
    \begin{align}
    \label{eq:relationPhivarphi}
      \Phi_{i,j}^{N,r}(x,y)= \left(1-e^{-\frac{x}{r}} \right)^{\mathds{1}_{0 \notin \left\{i,j \right\}}} \times \left(1-e^{-\frac{\ell-y}{r}} \right)^{\mathds{1}_{N+1 \notin \left\{i,j \right\}}} \times \varphi(x,y)^{N-2+\frb},
    \end{align}
    where~$\frb$ is as usual the number of boundary points among~$i$ and~$j$. The first two terms in the above product come from~\eqref{eq:blankzonesBP}. They correspond to the (potential) contribution of boundary points.
    
We continue with a second lemma. This time is computed the probability to fully cover some fixed interval:
\begin{lem}[Filling a hole]
    \label{lem:fullcoverage}
    Recall the definition~\eqref{eq:defpsi} of~$\Psi$. For any~$0 \leq x \leq y \leq \ell$:
    \begin{align}
        \label{eq:fullcoverage}
        \Psi(x,y)=\frac{r}{2 \ell}\left(2e^{-\frac{y-x}{r}}-e^{-2 \frac{\ell-x}{r}}-e^{-\frac{2y}{r}}\right).
    \end{align}
\end{lem}
The proof is once again skipped as it rests on basic analysis. We are now in good position to define the maps~$\left(G_\frb^r,H_\frb^r,K_\frb^r\right)$ of Proposition~\ref{prop:comparisonLambdaR}.
From~\eqref{eq:expressionpartialderivativeswrtlambdadecompos} and~\eqref{eq:fullcoverage}, we deduce an upper bound on~$\ell \cdot \frL_{i,j}^{N,r}$, valid for any~$0 \leq i \neq j \leq N+1$: 
\begin{align}
%\label{eq:lambdamoreexplicit}
    \ell \cdot \frL_{i,j}^{N,r}&=\frac{r}{2} \cdot \bbE\left[\Phi(\bfS_i^r,\bfI_j^r) \left(2e^{- \frac{\bfI_j^r-\bfS_i^r}{r}}-e^{-2\frac{\ell-\bfS_i^r}{r}}-e^{-\frac{2 \bfI_j^r}{r}}\right) \mathds{1}_{\bfS_{i}^r < \bfI_{j}^r}\right] \nonumber \\
    &\leq r \cdot \bbE\left[\Phi(\bfS_i^r,\bfI_j^r) e^{- \frac{\bfI_j^r-\bfS_i^r}{r}} \mathds{1}_{\bfS_{i}^r < \bfI_{j}^r}\right]. \nonumber
\end{align}
 When~$1 \leq j \leq N$ (and~$0 \leq \frb \leq 1$), it holds that~$\bfI_j^r$ is distributed as~$\ell - \bfS_j^r$, due to the reflection symmetry over the vertical line~$x=1/2$. The last inequality can then be rephrased in the more convenient way:
\begin{align}
%\label{eq:lambdamoreexplicit}
    \ell \cdot \frL_{i,j}^{N,r}\leq r e^{-\frac{\ell}{r}} \bbE\left[\Phi(\bfS_i^r,\bfI_j^r) e^{\frac{\bfS_i^r+\bfS_j^r}{r}} \mathds{1}_{\bfS_{i}^r + \bfS_j^r < \ell}\right]. \nonumber
\end{align}
We get from~\eqref{eq:relationPhivarphi} that:
\begin{align}
\label{eq:lambdamoreexplicit}
    \ell \cdot \frL_{i,j}^{N,r}\leq W_{\ell,r}^{N-2+\frb} \left(1-e^{-\frac{\ell}{r}} \right)^{2-\frb} r e^{-\frac{\ell}{r}}  \bbE\left[e^{\frac{\bfS_i^r+\bfS_j^r}{r}} \mathds{1}_{\bfS_{i}^r + \bfS_j^r < \ell}\right],
\end{align}
by using that~$1-e^{-\frac{x}{r}} \leq 1-e^{-\frac{\ell}{r}}$ for~$x \in \left[0,\ell\right]$ and also that~$W_{\ell,r}$ globally dominates~$\varphi$ according to Corollary~\ref{cor:globalboundsemptypointoverlap}. It invites us to set:
\begin{align}
\label{eq:defGonetwo}
   G_\frb^r (\ell) := r e^{-\frac{\ell}{r}}  \left(1-e^{-\frac{\ell}{r}} \right)^{2-\frb} \bbE\left[e^{\frac{\bfS_i^r+\bfS_j^r}{r}} \mathds{1}_{\bfS_{i}^r + \bfS_j^r < \ell}\right],
\end{align}
 for~$0 \leq \frb \leq 1$. When~$j=N+1$, given~$\bfI_{N+1}^r=\ell-\bfR_{N+1}^r$, the same exact reasoning leads to:
 \begin{align}
\label{eq:defGzero}
    G_2^r(\ell) := r e^{-\frac{\ell}{r}}   \bbE\left[e^{\frac{\bfR_0^r+\bfR_{N+1}^r}{r}} \mathds{1}_{\bfR_{0}^r + \bfR_{N+1}^r < \ell}\right].
\end{align}

Such work can equally be done with~$H_\frb^r$ and~$K_\frb^r$, which respectively minorizes and majorizes~$\frR_{\frb}^r$. On a case-by-case basis because the integral representation of the latter depends on the value of~$\frb$. See~\eqref{eq:integralrepresentationR},~\eqref{eq:integralrepresentationbisR} and~\eqref{eq:integralrepresentationterR}. We rely on the following inequalities:
\begin{align}
    %\label{eq:controlPhiforR}
    \left(1-e^{-\frac{x}{r}} \right)^{\mathds{1}_{0 \notin \left\{i,j \right\}}} \times \left(1-e^{-\frac{\ell-x}{r}} \right)^{\mathds{1}_{N+1 \notin \left\{i,j \right\}}} \times W_{\ell,2r}^{N-2+\frb} \leq \Phi_{i,j}^{N,r}(x,y) \leq 1, \nonumber
\end{align}
the first being a consequence of Corollary~\ref{cor:globalboundsemptypointoverlap}. Hence, we define for~$\frb=2$:
\begin{align}
\label{eq:definitionHKzero}
    H_2^r (\ell)=K_2^r (\ell)=\frac{\ell}{r} \int_{[0,\ell]} \bfs^r(u) \bfi^r(u)  \ \mathrm{d}u.
\end{align}
It must be reminded that~$\bfs^r$ and~$\bfi^r$ denote the densities of~$\bfS_0^r$ and~$\bfI_{N+1}^r$ respectively.
For~$\frb=1$:
\begin{align}
\label{eq:definitionHtwo}
    H_1^r (\ell)= \frac{1}{r} \int_{[0,\ell]^2} \left(1-e^{-\frac{\ell-u}{r}} \right) (u+y) \bfs^r(u) \bfi_1^r(u,y)  \mathds{1}_{u+y \leq \ell} \ \mathrm{d}u \ \mathrm{d}y
\end{align}
and
\begin{align}
\label{eq:definitionKtwo}
    K_1^r (\ell)= \frac{1}{r} \int_{[0,\ell]^2} (u+y) \bfs^r(u) \bfi_1^r(u,y)  \mathds{1}_{u+y \leq \ell} \ \mathrm{d}u \ \mathrm{d}y,
\end{align}
with~$\bfi_1^r$ the joint density of~$\left(\bfI_1^r,\bfR_1^r\right)$.
Finally, for~$\frb=0$:
\begin{align}
\label{eq:definitionHthree}
    H_0^r (\ell)= \frac{1}{r} \int_{[0,\ell]^3} (v+y) \left(1-e^{-\frac{\ell-u}{r}} \right) \left(1-e^{-\frac{u}{r}} \right) \bfs_1^r(u,v) \bfi_2^r(u,y) \mathds{1}_{v+y \leq \ell}  \ \mathrm{d}u \ \mathrm{d}v \ \mathrm{d}y
\end{align}
and
\begin{align}
\label{eq:definitionKthree}
    K_0^r (\ell)= \frac{1}{r} \int_{[0,\ell]^3} (v+y)\bfs_1^r(u,v) \bfi_2^r(u,y) \mathds{1}_{v+y \leq \ell}  \ \mathrm{d}u \ \mathrm{d}v \ \mathrm{d}y,
\end{align}
where~$\bfs_1^r$ designates the joint density of~$\left(\bfS_1^r,\bfR_1^r\right)$ and~$\bfi_2^r$ that of~$\left(\bfI_2^r,\bfR_2^r\right)$.
The maps involved in Proposition~\ref{prop:comparisonLambdaR} are now explicit. We have to establish their continuity, their boundedness and the inequality~\eqref{eq:inequalityfunctionGH} in every case. A common strategy is employed. Continuity either results from a direct computation of the integrals, or from a standard application of the Lebesgue's dominated convergence theorem. A separate asymptotic analysis shows afterwards that all the functions vanish as~$\ell \to 0$ and as~$\ell \to +\infty$, so they are bounded. We also prove at once that~$G_\frb^r/H_\frb^r$ converges in both regimes, locally uniformly in~$r$. It implies that~$\sup_{\ell >0} \frac{G_\frb^r(\ell)}{H_\frb^r(\ell)}$ exists and is continuous w.r.t.~$r$. It can therefore play the role of the map~$C_\frb$ in~\eqref{eq:inequalityfunctionGH}. All of this is then enough to conclude.    \\

\textit{The case~$\frb=2$.}
Here the involved points~$i$ and~$j$ are the extremities of the segment~$[0,\ell]$. Recall that~$\bfS_0^r=\bfR_0^r$ and~$\bfI_{N+1}^r=\ell-\bfR_{N+1}^r$, where~$\bfR_0^r$ and~$\bfR_{N+1}^r$ are exponentially-distributed random variables with mean~$r$. Their density is~$x \ \mapsto \frac{1}{r} e^{-\frac{x}{r}}$. We can compute~\eqref{eq:defGzero}:
\begin{align}
    G_2^r(\ell) = r e^{-\frac{\ell}{r}} \int_{[0,\ell]^2} e^{\frac{x+y}{r}}  \mathds{1}_{x+y < \ell} \ \frac{1}{r} e^{-\frac{x}{r}} \frac{1}{r} e^{-\frac{y}{r}} \ \mathrm{d}x \ \mathrm{d}y = \frac{\ell^2}{2r} e^{-\frac{\ell}{r}}, \nonumber
\end{align}
and~\eqref{eq:definitionHKzero}:
\begin{align}
    H_2^r (\ell)=K_2^r(\ell) = \frac{\ell}{r} \int_0^\ell \frac{1}{r} e^{-\frac{u}{r}} \frac{1}{r} e^{-\frac{\ell-u}{r}} \ \mathrm{d}u = \frac{\ell^2}{r^3} e^{-\frac{\ell}{r}}, \nonumber
\end{align}
by using that~$\bfi^r (\cdot)=\bfs^r(\ell-\cdot)$. The three maps are thus bounded and continuous, as is~$\sup_{\ell >0} \frac{G_2^r(\ell)}{H_2^r(\ell)}$, equal to~$\frac{r^2}{2}$, which then suits to be~$C_2$. We got everything that we expected. \\

\textit{The case~$\frb=0$.}
 In this paragraph, both points~$i$ and~$j$ are internal. Without loss of generality, we assume that~$i=1$ and~$j=2$. Since~$\bfS_1^r=\bfU_1 + \frac{r}{2} \bdcalE_1$ with~$\bfU_1 \sim \text{Unif}[0,\ell]$ and~$\bdcalE_1 \sim \text{Exp(1)}$ independent of each other, the associated density, restricted to the segment~$[0,\ell]$, is~$x \mapsto \frac{1}{\ell}(1-e^{-\frac{2x}{r}})$. So is the density of~$\bfS_2^r$. It allows us to develop~\eqref{eq:defGonetwo}:
 \begin{align}
 \label{eq:defGbissimplified}
     &G_0^r (\ell) \nonumber \\ & = \frac{r}{\ell^2} e^{-\frac{\ell}{r}}  \left(1-e^{-\frac{\ell}{r}} \right)^{2}\int_{0 \leq x+y < \ell} e^{\frac{x+y}{r}} (1-e^{-
     \frac{2x}{r}}) (1-e^{-\frac{2y}{r}}) \mathrm{d}x \ \mathrm{d}y \nonumber \\
     &= \frac{r}{\ell^2} e^{-\frac{\ell}{r}}  \left(1-e^{-\frac{\ell}{r}} \right)^{2} \left( \int_{0 \leq x +y \leq \ell} e^{\frac{x+y}{r}} \mathrm{d}x\mathrm{d}y - 2 \int_{0 \leq x +y \leq \ell} e^{\frac{y-x}{r}} \mathrm{d}x \mathrm{d}y + \int_{0 \leq x +y \leq \ell} e^{-\frac{x+y}{r}} \mathrm{d}x  \mathrm{d}y \right) \nonumber \\
    &= \frac{r}{\ell^2} e^{-\frac{\ell}{r}}  \left(1-e^{-\frac{\ell}{r}} \right)^{2} \left( r \ell e^{\frac{\ell}{r}} - r^2 e^{\frac{\ell}{r}} + r^2 -r^2 e^{\frac{\ell}{r}} +2 r^2 -r^2 e^{\frac{-\ell}{r}} +r^2 -r \ell e^{-\frac{\ell}{r}} - r^2 e^{-\frac{\ell}{r}} \right) \nonumber \\
    &= \frac{r}{\ell^2} e^{-\frac{\ell}{r}}  \left(1-e^{-\frac{\ell}{r}} \right)^{2}  \left( 4 r^2 - 4 r^2 \cosh{\left(\frac{\ell}{r}\right)} + 2 r \ell \sinh{\left(\frac{\ell}{r}\right)}  \right).
 \end{align}
The latter expression is clearly continuous w.r.t~$\ell$. In the vicinity of~$0$, we check that 
$$4 r^2 - 4 r^2 \cosh{\left(\frac{\ell}{r}\right)} + 2 r \ell \sinh{\left(\frac{\ell}{r}\right)} \sim \frac{\ell^4}{6 r^2},$$ 
so that:
\begin{align}
\label{eq:asymptoticGtwovicinityzero}
    G_0^r (\ell) \underset{\ell \to 0}{\sim} \frac{\ell^4}{6 r^3}.
\end{align}
At the same time, as~$\ell \to +\infty$:
\begin{align}
\label{eq:asymptoticGtwovicinityinfinity}
    G_0^r (\ell) \sim \frac{r}{\ell^2} e^{-\frac{\ell}{r}} \times 2r \ell \times \sinh{\left(\frac{\ell}{r}\right)} \sim 2 r^2 \ell^{-1}. 
\end{align}
We deal now with~\eqref{eq:definitionHthree}. The joint density~$\bfs_1^r$ of~$\left(\bfS_1^r,\bfR_1^r\right)$ is
$$(u,v) \mapsto \frac{2}{r \ell }e^{-\frac{2v}{r}}\mathds{1}_{u \in [v,v+\ell]} \mathds{1}_{v\geq 0},$$ 
while that of~$\left(\bfI_2^r,\bfR_2^r\right)$ is
$$\bfi_2^r : \ (u,v) \mapsto \frac{2}{r \ell }e^{-\frac{2y}{r}}\mathds{1}_{u \in [-y,-y+\ell]} \mathds{1}_{y\geq 0}.$$
We inject into~\eqref{eq:definitionHthree}:
\begin{align}
    H_0^r (\ell) = \frac{4}{r^3 \ell^2} \int_{\bbR_+^3} (v+y) e^{-\frac{2\left(v+y\right)}{r}} \left(1-e^{-\frac{u}{r}}\right) \left(1-e^{-\frac{\ell-u}{r}}\right) \mathds{1}_{u \in [v,\ell-y]} \ \mathds{1}_{v+y \leq \ell} \ \mathrm{d}u \ \mathrm{d}v \ \mathrm{d}y. \nonumber
\end{align}
The integrand above is a piecewise continuous function on~$\bbR_+^4$. For any fixed~$\ell > 0$, its restriction to~$\left\{\ell\right\}\times\bbR_+^3$ has a compact support. The continuity of~$H_0^r$ is then a straightforward consequence of the Lebesgue's dominated convergence theorem. We clarify the expression of~$H_0^r$ thanks to the substitution~$z=v+y$, after first integrating w.r.t~$u$:
\begin{align}
\label{eq:defHbissimplified}
H_0^r(\ell) = \frac{4 \left(1+e^{-\frac{\ell}{r}}\right)}{r^3 \ell^2} \int_{0}^\ell z^2 (\ell-z) e^{-\frac{2z}{r}} \ \mathrm{d}z - \frac{8}{r \ell^2} \int_{0}^\ell z e^{-\frac{2z}{r}} (1-e^{-\frac{\ell-z}{r}}) (1-e^{-\frac{z}{r}})\ \mathrm{d}z.
\end{align}
In the small length regime~$\ell \to 0$:
\begin{align}
     \int_{0}^\ell z^2 (\ell-z) e^{-\frac{2z}{r}} \ \mathrm{d}z &= \int_0^\ell z^2 \left(1-\frac{2z}{r} + \frac{2z^2}{r^2}\right) \left(\ell-z\right) \ \mathrm{d}z + o\left(\ell^6 \right) \nonumber \\
     &=\ell \int_{0}^\ell \left[z^2-\frac{2 z^3}{r} + \frac{2z^4}{r^2} + o\left(z^4\right)\right]  \mathrm{d}z  - \int_{0}^\ell \left[z^3 -\frac{2 z^4}{r}+\frac{2 z^5}{r^2}+o\left(z^5\right)\right]   \mathrm{d}z \nonumber \\ 
     &= \frac{\ell^4}{3} - \frac{\ell^5}{2r}+ \frac{2\ell^6}{5 r^2} -\frac{\ell^4}{4} + \frac{2 \ell^5}{5r} - \frac{\ell^6}{3r^2}+o\left(\ell^6\right) \nonumber \\ 
     &= \frac{\ell^4}{12} - \frac{\ell^5}{10 r}  + \frac{\ell^6}{15 r^2} + o\left(\ell^6\right). \nonumber
\end{align}
and:
\begin{align}
    &\int_{0}^\ell z e^{-\frac{2z}{r}} (1-e^{-\frac{\ell-z}{r}}) (1-e^{-\frac{z}{r}})\ \mathrm{d}z  \nonumber \\
    &= \left(1+e^{-\frac{\ell}{r}} \right) \int_{0}^\ell z e^{-\frac{2z}{r}} \ \mathrm{d}z  - \int_{0}^\ell z e^{-\frac{3z}{r}} \ \mathrm{d}z - e^{-\frac{\ell}{r}} \int_{0}^\ell z e^{-\frac{z}{r}} \ \mathrm{d}z  \nonumber \\
    &= \left(1+e^{-\frac{\ell}{r}} \right) \left(\int_{0}^\ell z \left(e^{-\frac{2z}{r}}-e^{-\frac{z}{r}}\right) \ \mathrm{d}z \right) + \int_{0}^\ell z \left( e^{-\frac{z}{r}}-e^{-\frac{3z}{r}}\right) \ \mathrm{d}z  \nonumber \\
    &= \left(1+e^{-\frac{\ell}{r}} \right) \left(-\frac{\ell^3}{3r}+\frac{3\ell^4}{8 r^2} -\frac{7 \ell^5}{30 r^3} + \frac{5\ell^6}{48 r^4} + o\left(\ell^6 \right) \right)+\left(\frac{2\ell^3}{3r}-\frac{\ell^4}{r^2} +\frac{13 \ell^5}{15 r^3} - \frac{5\ell^6}{9 r^4} + o\left(\ell^6 \right)\right). \nonumber
\end{align}
Consequently:
\begin{align}
\label{eq:asymptoticHtwoinvicinityzero}
   &H_0^r (\ell) \nonumber \\
    &= \frac{1+e^{-\frac{\ell}{r}}}{r^3 \ell^2}\left(\frac{8r}{3}\ell^3 -\frac{8}{3} \ell^4 + \frac{22}{15 r} \ell^5 - \frac{17}{30 r^2} \ell^6\right) - \frac{8}{r \ell^2} \left(\frac{2\ell^3}{3r}-\frac{\ell^4}{r^2} +\frac{13 \ell^5}{15 r^3} - \frac{5\ell^6}{9 r^4} \right) + o\left(\ell^6\right) \nonumber \\ 
    &= \frac{1}{r^3 \ell^2} \left( \frac{16r}{3}\ell^3 - 8 \ell^4 + \frac{104}{15r} \ell^5 - \frac{197}{45 r^2} \ell^6 \right) - \left( \frac{16}{3 r^2} \ell - \frac{8}{r^3} \ell^2 + \frac{104}{15 r^4}\ell^3 - \frac{40}{9 r^5} \ell^4 \right) + o\left(\ell^6\right) \nonumber \\
    & \underset{\ell \to 0}{\sim} \frac{\ell^4}{15 r^5}.
\end{align}
In the large length regime~$\ell \to +\infty$, we have
$$\int_{0}^\ell z^2 (\ell-z) e^{-\frac{2z}{r}} \ \mathrm{d}z \sim \ell \int_0^{+\infty} z^2 e^{-\frac{2z}{r}} \ \mathrm{d}z,$$
and
$$\lim_{\ell \to +\infty} \int_{0}^\ell z e^{-\frac{2z}{r}} (1-e^{-\frac{\ell-z}{r}}) (1-e^{-\frac{z}{r}})\ \mathrm{d}z = \int_{0}^{+\infty} z e^{-\frac{2z}{r}} (1-e^{-\frac{z}{r}})\ \mathrm{d}z <+\infty, $$
thanks to the Lebesgue's dominated convergence theorem. We deduce from~\eqref{eq:defHbissimplified} that:
\begin{align}
    \label{eq:asymptoticHtwovicinityinfinity}
    H_0^r (\ell) \underset{\ell \to +\infty}{\sim} \left( \frac{4}{r^3} \int_0^{+\infty} z^2 e^{-\frac{2z}{r}} \ \mathrm{d}z \right) \times \ell^{-1} = \ell^{-1}.
\end{align}
The asymptotics~\eqref{eq:asymptoticGtwovicinityzero}, \eqref{eq:asymptoticGtwovicinityinfinity}, \eqref{eq:asymptoticHtwoinvicinityzero} and~\eqref{eq:asymptoticHtwovicinityinfinity} show that~$G_0^r$ and~$H_0^r$ vanish and are of exact same order as~$\ell \to 0$, as well as~$\ell \to +\infty$, so that~$G_0^r/H_0^r$ converges to some limit in both regimes. The maps being continuous, they are bounded and~$\sup_{\ell >0} \frac{G_0^r(\ell)}{H_0^r(\ell)} <+\infty$. Set~$C_0(r):=\sup_{\ell >0} \frac{G_0^r(\ell)}{H_0^r(\ell)}$. Continuity w.r.t.~$r$ of~$C_0$ results from several facts put together. The first is that~$G_0/H_0$ is continuous on~$(0,+\infty)^2$ as a function of~$r$ and~$\ell$, and continuously extendable at~$(0,+\infty)\times[0,+\infty)$. Indeed, given~\eqref{eq:asymptoticGtwovicinityzero} and~\eqref{eq:asymptoticHtwoinvicinityzero}, the limit at~$0$ of~$G_0^r(\ell)/H_0^r(\ell)$ is~$5r^2/2$ for any~$r>0$, that is a continuous function w.r.t~$r$ on~$(0,+\infty)$. Besides, by investigating even more thoroughly the behaviour of~$G_0^r$ and~$H_0^r$ in the vicinity of zero, we could demonstrate that the latter convergence is locally uniform in~$r$, meaning that on any segment~$I \subset (0,+\infty)$:
\begin{align}
\nonumber
    \sup_{r \in I} \left| \frac{G_0^r(\ell)}{H_0^r(\ell)}-\frac{5r^2}{2}\right| < c \ell 
\end{align}
for some constant~$c=c_I>0$ and~$\ell$ small enough. This ensures that the map is continuously extendable as claimed above. We derive from it that~$\frac{G_0^r(\ell)}{H_0^r(\ell)}$ is locally uniformly continuous on~$(0,+\infty)\times[0,+\infty)$. Therefore, for any~$r>0$ and any~$L>0$:
\begin{align}
\label{eq:uniformcontinuity1}
    \sup_{0 \leq \ell \leq L} \ \sup_{r' \in (r-\delta,r+\delta)} \left|\frac{G_0^{r'}(\ell)}{H_0^{r'}(\ell)}-\frac{G_0^{r}(\ell)}{H_0^{r}(\ell)}\right| \xrightarrow[\delta \to 0]{} 0.
\end{align}
The second fact to mention is an equivalent of the former at infinity. Given~\eqref{eq:asymptoticGtwovicinityinfinity} and~\eqref{eq:asymptoticHtwovicinityinfinity}, the limit of~$G_0^r(\ell)/H_0^r(\ell)$ as~$\ell \to +\infty$ is~$2 r^2$ for any~$r>0$, that is here again a continuous function w.r.t~$r$ on~$(0,+\infty)$. Some basic analysis manipulations on the expressions~\eqref{eq:defGbissimplified} and~\eqref{eq:defHbissimplified} show that the rate of convergence is locally uniform in~$r$, in the sense that on any segment~$I \subset (0,+\infty)$:
\begin{align}
\nonumber
    \sup_{r \in I} \left| \frac{G_0^r(\ell)}{H_0^r(\ell)}-2r^2\right| < c \ell^{-1}. 
\end{align}
for some constant~$c=c_I>0$ and~$\ell$ large enough. Thus:
\begin{align}
   \limsup_{\delta \to 0} \ \sup_{\ell > L} \ \sup_{r' \in (r-\delta,r+\delta)} \left|\frac{G_0^{r'}(\ell)}{H_0^{r'}(\ell)}-\frac{G_0^{r}(\ell)}{H_0^{r}(\ell)}\right| \leq c L^{-1}, \nonumber
\end{align}
for some~$c>0$ and any~$L>0$ large enough. Combined with~\eqref{eq:uniformcontinuity1}, it entails that:
\begin{align}
\sup_{\ell>0} \sup_{r' \in (r-\delta,r+\delta)} \left|\frac{G_0^{r'}(\ell)}{H_0^{r'}(\ell)}-\frac{G_0^{r}(\ell)}{H_0^{r}(\ell)}\right| \xrightarrow[\delta \to 0]{} 0. \nonumber
\end{align}
The continuity of~$C_0$ directly follows.
The final word in this paragraph is about the map~$K_0^r$. We remark that its definition~\eqref{eq:definitionKthree} is very close to that of~$H_0^r$~\eqref{eq:definitionHthree}. On the model of what we did for it, we prove continuity of~$K_0^r$ from the Lebesgue's dominated convergence theorem. We compute afterwards that~$K_0^r (\ell) = \frac{4}{r^3 \ell^2} \int_0^\ell z^2 \left( \ell-z\right) e^{-\frac{2z}{r}} \ \mathrm{d}z$, which implies that~$K_0^r(\ell) =\calO\left(\ell^2\right)$ as~$\ell \to 0$ and~$K_0^r(\ell) \sim \ell^{-1}$ as~$\ell \to +\infty$. The function is hence bounded.  \\

\textit{The case~$\frb=1$.} We finally deal with the mixed situation where~$i$ is a boundary point and~$j$ an internal one, say for instance~$i=0$ and~$j=1$. Recall that the density~$\bfs^r$ of~$\bfS_0^r$ is~$x \mapsto \frac{1}{r} e^{-\frac{x}{r}}$. That of~$\bfS_1^r$ is~$x \mapsto \frac{1}{\ell}\left(1-e^{-\frac{2x}{r}}\right)$. We derive an explicit expression of~$G_1^r$ from~\eqref{eq:defGonetwo}:
\begin{align}
    G_1^r(\ell) &= \ell^{-1} e^{-\frac{\ell}{r}}  \left(1-e^{-\frac{\ell}{r}} \right) \int_{0 \leq x+y < \ell} e^{\frac{x+y}{r}} e^{-\frac{x}{r}} (1-e^{-
     \frac{2y}{r}}) \ \mathrm{d}x \ \mathrm{d}y \nonumber \\
     &= 2 \ell^{-1} e^{-\frac{\ell}{r}}  \left(1-e^{-\frac{\ell}{r}} \right) \left(\int_{0}^\ell \left(\ell-y \right)\sinh{\left(\frac{y}{r} \right)} \ \mathrm{d}x \ \mathrm{d}y \right) \nonumber \\
     &= 2 \ell^{-1} e^{-\frac{\ell}{r}}  \left(1-e^{-\frac{\ell}{r}} \right) \left( r^2 \sinh{\left(\frac{\ell}{r}\right)}-r \ell  \right). \nonumber
\end{align}
The map is unequivocally continuous. Since~$r^2 \sinh{\left(\frac{\ell}{r}\right)}-r \ell \sim \frac{\ell^3}{6r}$ in the vicinity of~$0$, we get 
\begin{align}
\label{eq:asymptoticGonevicinityzero}
    G_1^r (\ell) \underset{\ell \to 0}{\sim} \frac{\ell^3}{3 r^2}. 
\end{align}
As~$\ell \to +\infty$, it holds this time that:
\begin{align}
\label{eq:asymptoticGonevicinityinfinity}
    G_1^r (\ell) \sim r^2 \ell^{-1}.
\end{align}
We turn now our attention to~$H_1^r$, defined by~\eqref{eq:definitionHtwo}. The joint density~$\bfi_1^r$ of~$\left(\bfI_1^r,\bfR_1^r\right)$ is 
$$(u,y) \mapsto \frac{2}{r \ell }e^{-\frac{2y}{r}}\mathds{1}_{u \in [-y,-y+\ell]} \mathds{1}_{y\geq 0},$$
so that:
\begin{align}
    H_1^r(\ell) = \frac{2}{r^3 \ell} \int_{\bbR_+^2} \left(1-e^{-\frac{\ell-u}{r}} \right) (u+y) e^{-\frac{u}{r}} e^{-\frac{2y}{r}} \mathds{1}_{u+y \leq \ell} \ \mathrm{d}u \ \mathrm{d}y. \nonumber 
\end{align}
Continuity of~$H_1^r$ again results from a simple application of the Lebesgue's dominated convergence theorem, whose details are left to the reader. The substitution~$z=u+y$ allows us to rewrite the above integral:
\begin{align}
\label{eq:defHbisonesimplified}
    H_1^r(\ell) = \frac{2}{r^3 \ell} \int_{0}^\ell z e^{-\frac{2z}{r}} \left( \int_0^z e^{\frac{u}{r}} \left(1-e^{-\frac{\ell-u}{r}}\right) \ \mathrm{d}u \right) \ \mathrm{d}z .
\end{align}
Given~$e^{\frac{u}{r}} \left(1-e^{-\frac{\ell-u}{r}} \right) = 1-e^{-\frac{\ell}{r}} + o\left(1 \right)$ around~$0$, we deduce that
\begin{align}
\label{eq:asymptoticHoneinvicinityzero}
    H_1^r (\ell) 
    &= \frac{2}{r^3 \ell} \int_{0}^\ell z e^{-\frac{2z}{r}} \left( \int_0^z \left[ 1-e^{-\frac{\ell}{r}} + o\left(1 \right)\right] \ \mathrm{d}u \right) \ \mathrm{d}z \nonumber \\
    &= \frac{2 \left( 1-e^{-\frac{\ell}{r}} \right)}{r^3 \ell} \int_{0}^\ell \left[ z^2 + o\left( z^2 \right)\right] \ \mathrm{d}z \underset{\ell \to 0}{\sim} \frac{2 \ell^3}{3 r^4},
\end{align}
In the large length regime, according to the Lebesgue's dominated convergence theorem:
\begin{align}
    &\lim_{\ell \to +\infty}  \int_{0}^\ell z e^{-\frac{2z}{r}} \left( \int_0^z e^{\frac{u}{r}} \left(1-e^{-\frac{\ell-u}{r}}\right) \ \mathrm{d}u \right) \ \mathrm{d}z \nonumber \\
    &= \int_{0}^{+\infty} z e^{-\frac{2z}{r}} \left( \int_0^z e^{\frac{u}{r}} \ \mathrm{d}u \right) \ \mathrm{d}z = r \int_{0}^{+\infty} z \left( e^{-\frac{z}{r}}-e^{-\frac{2z}{r}} \right)  \ \mathrm{d}z = \frac{3 r^3}{4}. \nonumber
\end{align}
Thus:
\begin{align}
\label{eq:asymptoticHonevicinityinfinity}
    H_1^r(\ell) \underset{\ell \to +\infty}{\sim} \frac{3}{2 \ell}.
\end{align}
We conclude exactly like in the previous case.
Given~\eqref{eq:asymptoticGonevicinityzero}, \eqref{eq:asymptoticGonevicinityinfinity}, \eqref{eq:asymptoticHoneinvicinityzero} and~\eqref{eq:asymptoticHonevicinityinfinity}, the maps~$G_1^r$ and~$H_1^r$ are asymptotically equivalent in the vicinity of zero and at infinity, so that~$G_1^r/H_1^r$ converges in both regimes. Since the maps are continuous, it implies that~$\sup_{\ell >0} \frac{G_1^r(\ell)}{H_1^r(\ell)} <+\infty$. We then set~$C_1:=\sup_{\ell >0} \frac{G_1^r(\ell)}{H_1^r(\ell)}$ and prove that~$C_1$ is continuous w.r.t.~$r$ on~$(0,+\infty)$ in the same way we did for~$C_0$. Lastly, we compute that $$K_1^r(\ell)=\frac{2}{r^2 \ell}  \int_{0}^{\ell} z \left( e^{-\frac{z}{r}}-e^{-\frac{2z}{r}} \right) \mathrm{d}z.$$ The expression is continuous w.r.t~$\ell$. It tends to zero as~$\ell \to 0$, as well as~$\ell \to +\infty$. Hence, the function is bounded. The proof of Proposition~\ref{prop:comparisonLambdaR} is now complete.

\section{Uniqueness of the infinite cluster}
\label{sect:uniqueness}

The purpose of this section is to show the uniqueness of the infinite cluster as stated in Theorem \ref{thm:uniqueness}. In this whole section, we assume that $p>0$ and $r<\infty$. If $p=0$, there is no percolation anyway, and for $r=\infty$ the models boils down to the  well-known Poisson-Vorono\"i percolation model.

We start with a geometric lemma on the trace of the Delaunay triangulation in any Euclidean ball. For a locally finite set of points $\gamma$ of $\mathbb{R}^{2}$ with associated Delaunay triangulation $\bfT_{\gamma}$ and a convex subset $\mathcal{C}$ of $\mathbb{R}^{2}$, we define the trace $\text{Tr}_{\mathcal{C}}(\mathbf{T}_\gamma)$ of $\bfT_{\gamma}$ in $\mathcal{C}$ as the subset of $\bfT$ made of the segments $[x,y]$ of $\mathbf{T}_\gamma$ such that $x\in \gamma\cap \mathcal{C}$ or $y\in \gamma\cap \mathcal{C}$, so
$$
\text{Tr}_{\mathcal{C}}(\mathbf{T}_\gamma)=\bigcup_{\underset{\scriptstyle y\in \gamma}{x\in \gamma\cap\mathcal{C}}}[x,y]\cap \mathbf{T}_\gamma.
$$
In other words, this is the subgraph made of the point in $\mathcal{C}$ in addition to the outer edges (see Figure \ref{fig:trace}).

We can now state our main tool, whose proof is postponed to the end of the section.
\begin{lem}
    \label{prop:trace}
	Let $\gamma$ be a locally finite subset of $\mathbb{R}^{2}$.
	Then, the trace $\text{Tr}_{\mathcal{C}}(\mathbf{T}_\gamma)$ of $\bfT_{\gamma}$  on any Euclidean ball $\mathcal{C}$ is connected.
\end{lem}
Note that, as shown in Figure \ref{fig:trace}, this does not hold for any convex set.

\begin{figure}
\captionsetup[subfigure]{labelformat=empty}
	\centering
	\subfloat[\centering A Delaunay triangulation]{\includegraphics[scale=0.45,trim=120 20 140 10,clip]{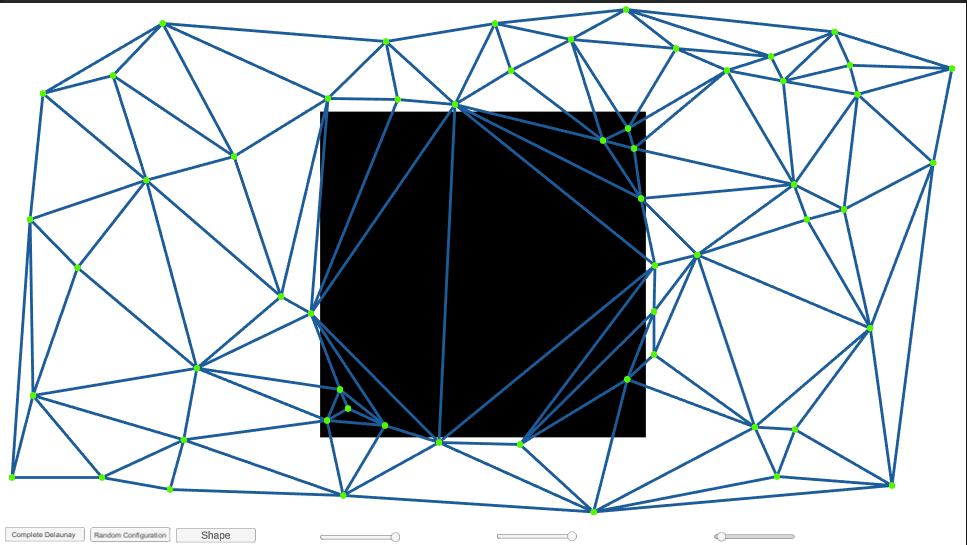}}
	\quad
	\subfloat[\centering The subgraph of the points in the square]{\includegraphics[scale=0.45,trim=150 30 150 0,clip]{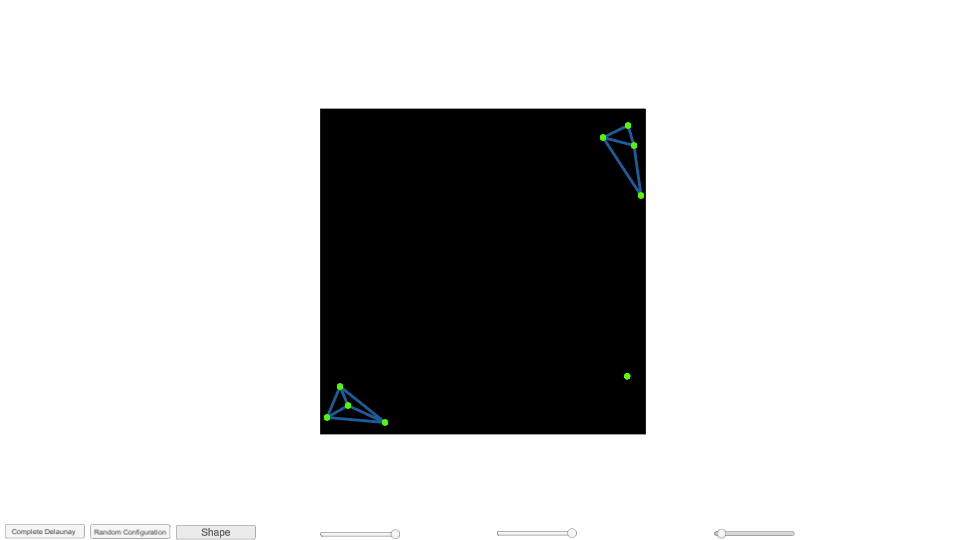}}
	\quad
	\subfloat[\centering The Trace in the square]{\includegraphics[scale=0.45,trim=120 20 120 0,clip]{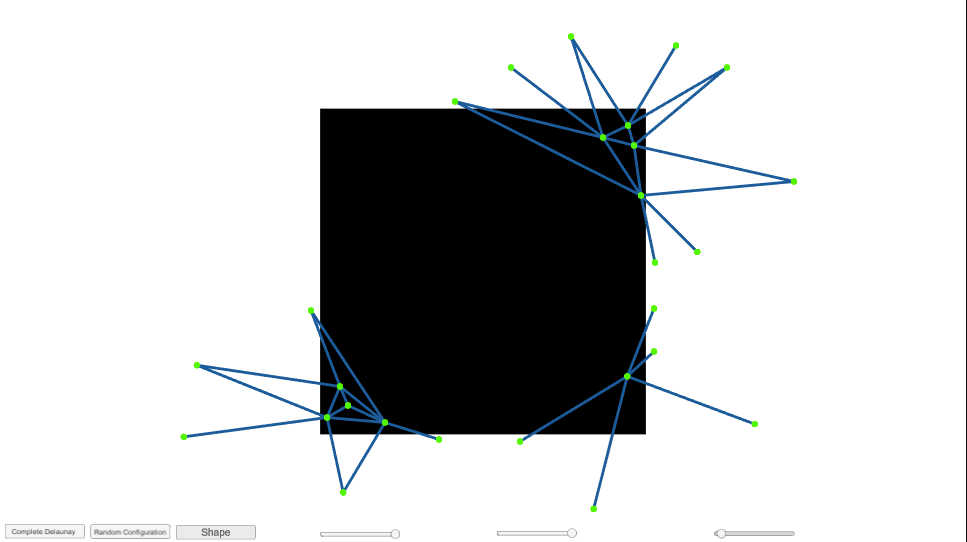}}
	\caption{Trace graph can be not connected on square }
	\label{fig:trace}
\end{figure}

\newcommand{\B}{{\mathfrak{B}_{\alpha}}}
\newcommand{\Bd}{\mathfrak{B}_{2\alpha}}

\newcommand{\trace}{\text{Tr}_{\B}(\mathbf{T})}
We can now start the proof of the uniqueness of the infinite connected component. 
Let $\ncc$ denotes the number of infinite connected components in $\bdcalG_{p,\lambda,r}$. Let us also denotes by $\B$ the Euclidean ball centered at $0$ with radius $\alpha$. Now, $\ncc_{\B}$ denotes the number of infinite connected components intersecting $\text{Tr}_{\B}(\mathbf{T})$ which are disjoint in the closure of $\bdcalG_{p,\lambda,r}\setminus \trace$. This in particular entails that we ask these components to penetrate $\trace$ though different vertices.  Note that infinite connected components intersecting $\B$ but without vertices in $\trace$ are not taken into account here.

For a ball $\B$ we also denote $\textrm{Full}(\B)$ the event that the trace of the Delaunay triangulation $\bfT$ in $\B$ is fully connected (all vertices are open, and all edges are completely covered).  Because edges can be covered by the exponential random variables associated to their endpoints, we have

\begin{lem}
\label{lem:fulltrace}
Almost surely
$$
    \mathbb{P}(\textrm{Full}(\B)\vert\ \bfX,\ \bfT)>0.$$
\end{lem}

\begin{proof}

We work by contradiction by assuming that $\mathbb{P}(\textrm{Full}(\B)\vert\ \bfX,\ \bfT)$ has a positive probability of being equal to $0$. In any case, we have
$$
\mathbb{P}(\textrm{Full}(\B)\vert\ \bfX,\ \bfT)\geq p^{\left|\bfX\cap \trace \right|}\exp\left(-2r\Lambda(\bfT \cap \trace)\right).
$$
In particular, this implies that 
$$
\mathbb{P}\Bigg(\left|\bfX\cap \trace \right|=\infty \text{ or } \Lambda(\bfT \cap \trace)=\infty\Bigg)>0.
$$
This in turns, because $\bfX$ and $\bfT$ are a.s.\ locally finite, implies that
$$
\mathbb{P}\left(\trace \text{ is unbounded}\right)>0.
$$
Fix $\beta>\alpha$. Let us denote by $B_{x,y}^+$ the ``upper'' half-ball with diameter $[x,y]$ defined by $B^{+}_{x,y}=\{z\in B((x+y)/2,\|x-y\|) \vert \det(x-y,x-z)\geq 0 \}$. $B^{-}_{x,y}$ stands for the ``lower'' half-ball, defined accordingly.

Because the event $\{\trace \nsubseteq \mathfrak{B}_{\beta}\}$ implies the existence of an edge of length at least $\beta-\alpha$. If such an edge has endpoints, say, $x$ and $y$, it is known that this implies that $B^{-}_{x,y}\cap \bfX=\emptyset$ or $B^{+}_{x,y}\cap \bfX=\emptyset$. Hence,
\begin{align*}
\mathbb{P}\left(\trace \nsubseteq \mathfrak{B}_{\beta}\right)&\leq \mathbb{P}\left(\exists x\in \bfX\cap \mathfrak{B}_{\alpha},\ \exists y\in \bfX\cap \mathfrak{B}_{\beta}^{c}  \text{ s.t.\ }\ B^{+}_{x,y}\cap \bfX =\emptyset \text{ or }B^{-}_{x,y}\cap \bfX =\emptyset\right)\\
&\leq \mathbb{E}\left[\sum_{x\in \bfX\cap\mathfrak{B}_{\alpha},\ y\in \bfX\cap\mathfrak{B}_{\beta}^c}\mathds{1}_{B^{+}_{x,y}\cap \bfX =\emptyset \text{ or }B^{-}_{x,y}\cap \bfX =\emptyset}\right].
\end{align*}
Applying two times Mecke's formula gives
\begin{align*}
 \mathbb{E}&\left[\sum_{x\in \bfX\cap\mathfrak{B}_{\alpha},\ y\in \bfX\cap\mathfrak{B}_{\beta}^c}\mathds{1}_{B^{+}_{x,y}\cap \bfX =\emptyset \text{ or }B^{-}_{x,y}\cap \bfX =\emptyset}\right]\\&=\int_{\mathfrak{B}_{\alpha}\times \mathfrak{B}_{\beta}^c}\mathbb{P}\left(B_+(x,y)\cap \bfX =\emptyset \text{ or }B^{-}_{x,y}\cap \bfX =\emptyset\right)\ dx\ dy\\
&\leq \int_{\mathfrak{B}_{\alpha}\times \mathfrak{B}_{\beta}^c}2e^{-\pi \|x-y\|^2/4}\ dx\ dy\leq \frac{\alpha^{2}}{4 } e^{-\pi (\alpha-\beta)^2/8}, 
\end{align*}
where the last inequality was obtained using that $\|x-y\|>(\alpha-\beta)$.
%\[
%\int_{\mathfrak{B}_{\alpha}\times \mathfrak{B}_{\beta}^c}2e^{-\pi \|x-y\|^2/4}\ dx\ %dy\leq2 \pi e^{-\pi (\alpha-\beta)^2/8}\alpha^{2} \int_{\mathbb{R}^2}2e^{-\pi \|x-
%y\|^2/8}\ dx=2 \pi e^{-\pi (\alpha-\beta)^2/8}\alpha^{2} \frac{1}{8\pi } 
%\]

\smallskip

 This gives a contradiction and that 
$\mathbb{P}(\textrm{Full}(\B)\vert\ \bfX,\ \bfT)>0$ almost surely.
\end{proof}

It is classical in such situation, because of ergodicity of the model, that $\ncc$ must be trivial as events of the form $\{\ncc=K\}$ are translation invariant. Hence, there exists $K\in \mathbb{N}\cup\{\infty\}$ such that $\mathbb{P}(\ncc=K)=1$.
\begin{figure}
	
\end{figure}
\bigskip

We begin by disqualifying all finite values greater than $1$.
\begin{lem}
	\label{lem:exclusion1}
	If, for some integer $K>1$, $\mathbb{P}(\ncc = K)>0$ then ${\mathbb{P}(\ncc=1)>0}$. In particular, this implies that $\mathbb{P}(\ncc=K)=0$ for all $K\in\mathbb{N}\setminus\{0,1\}$.
\end{lem}
\newcommand{\Fu}{\textrm{Full}}
\begin{proof}
	Let assume that, for some $K>1$, $\mathbb{P}(\ncc = K)>0$. In particular, we also must have $\mathbb{P}(\ncc=K \vert\ \bfX,\ \bfT)>0$ on an event $A$ (measurable with respect to $\bfX$) of positive probability. Let $N_\alpha$ be the number of infinite connected components strictly included in $\mathbf{T}\setminus\trace$. There exists $\alpha>0$ such that $\mathbb{P}(N_\alpha = 0,\ \ncc=K\vert\ \bfX,\ \bfT)>0$ on $A$ (which means that for $\alpha$ large enough, all connected component intersect $\trace$). Moreover, we have that $\mathbb{P}(N_\alpha=0,\ \ncc_{\B}\geq K\vert\ \bfX,\ \bfT  )\geq \mathbb{P}(N_\alpha = 0,\ \ncc=K\vert\ \bfX,\ \bfT)$.  \\
	
	We have that (conditionally on $\bfX$) $\ncc_{\B}$ and $N_\alpha$ only depends on $\bfT\setminus\trace$ and $\Fu(\B)$ on $\trace$, but these are independent conditionally on $\bfX$ (because all edges are independent conditionally on $\bfX$). Hence,
	\begin{align*}
		&\mathbb{P}(\ncc=1\vert\ \bfX,\ \bfT)\geq\mathbb{P}(N_\alpha=0,\ \ncc_{\B}\geq K,\ \Fu(\B)\vert\ \bfX,\ \bfT  )\\&\quad=\mathbb{P}(N_\alpha=0,\ \ncc_{\B}\geq K\vert\ \bfX,\ \bfT  )\mathbb{P}( \textrm{Full}(\B)\vert\ \bfX,\ \bfT  )>0\text{ a.s. on }A.
	\end{align*} 
	The first inequality is a consequence of Lemma \ref{prop:trace} while the last inequality is a consequence of Lemma \ref{lem:fulltrace}. Integrating with respect to $\bfX$ gives the result.
\end{proof}
So, it remains to remove the case $\ncc=\infty$. We use an adaptation of Burton-Keane method. To do so, we say that a point $x\in\bfX$ is a \textit{trifurcation} if
\begin{itemize}
    \item $x$ belongs to an infinite connected component $\mathcal{C}$ in $\bdcalG_{p,\lambda,r}$
    \item $\mathcal{C}\setminus\{x\}$ consists of three disjoint infinite connected components.
\end{itemize}

Conditionally on $\bfX$ and $\mathcal{G}_{p,\lambda,r}\setminus \trace$, on the event $\{\ncc_{\B}\geq 3\}$, there exists (at least) three distinct points $x,y,z\in \trace\cap \bfX$ such that three infinite connected components penetrate $\trace$ through these points. We denote by $\mathcal{E}$ the event that
\begin{itemize}
    \item all vertices in $\trace \setminus \B $ except $x,y$ and $z$ are closed,
    \item for any points among $x,y$ or $z$, there exists one and only one open edge going from the considered point into $\B$ which is open,
    \item there exists a unique open path in $\trace$ connecting $x,y$ and $z$,
    \item all other edges in $\trace$ are closed.
\end{itemize}
It is important at this point to remark that the existence of such connecting paths is guaranteed by Lemma \ref{prop:trace}. Moreover, the paths from $x$ to $y$ and from $x$ to    $z$ pass through $\B$ and these paths bifurcate at some point $t$ of $\bfX\cap \B$. This point $t$ must then be a trifurcation.

In particular, on the event $\{\ncc_{\B}\geq 3\}$, we have the rough bound
\begin{equation}
\label{eq:newControl}
\mathbb{P}(\mathcal{E}\mid \bfX,\ \bfT,\ \mathcal{G}_{p,\lambda,r}\setminus \trace)\geq \mathbb{P}(\textrm{Full}(\B)\vert\ \bfX,\ \bfT) \mathcal{H}_{\alpha}>0 \text{ a.s.,}
\end{equation}
with
$$
\mathcal{H}_{\alpha}=(1-p)^{\left|\bfX\cap \trace \right|}e^{-\lambda \Lambda(\bfT \cap \trace) }(1-e^{-r\inf_{\{x,y\}\subset \mathbf{X}\cap\B}\|x-y\|}).
$$
The fact that $\mathcal{H}_{\alpha}>0$ a.s.\ is obtained similarly as Lemma \ref{lem:fulltrace}.
 It is important to note that the choice of $x,y,z$ depends on $\mathcal{G}_{p,\lambda,r}\setminus \trace$ but that the lower bound in the above inequality does not.

\medskip

In the following, we also need the Palm distribution of, $\bfX$ denoted $\mathbb{P}^x$. Under $\mathbb{P}^{x}$, $\bfX$ has the law of a Poisson process with an additional deterministic point at $x$.

We then have the lemma.
\begin{lem}
	\label{lem:trifurcation}
	If $\mathbb{P}(\ncc=\infty)=1$, then
	\[
	\mathbb{P}^{0}(0\text{ is a trifurcation})>0.
	\]
	In addition, if $T_{\alpha}$ denotes the number of trifurcation in the ball $\B$ (for $\alpha>0$), we have
	\begin{equation}
		\label{eq:trifurcationControl}
		\mathbb{E}\left[T_{\alpha}\right]=|\B|\mathbb{P}^{0}(0\text{ is a trifurcation})=\pi\alpha^{2}\mathbb{P}^{0}\left(0\text{ is a trifucation}\right).
	\end{equation}
\end{lem}
\begin{proof}
	As in the proof of Lemma \ref{lem:exclusion1}, under the hypothesis that ${\mathbb{P}(\ncc=\infty)=1}$, there exists $\alpha>0$ such that $\mathbb{P}(\ncc_{\B}\geq 3\vert\ \bfX,\ \bfT)>0$ on an event $A$ of positive probability (measurable with respect to $\bfX$). 
 Now,
	\begin{align*}
		\mathbb{P}\left(T_{\alpha}>0\vert\ \bfX,\ \bfT\right)&\geq \mathbb{E}\left[\mathds{1}_{\ncc_{\B}\geq 3}\ \mathbb{P}\left( \mathcal{E}\vert\ \bfX,\ \bfT, \mathcal{G}_{p,\lambda,r}\setminus \trace\right)\vert\  \bfX,\ \bfT\right]\\&\geq \mathbb{P}\left(\ncc_{\B}\geq3\vert\ \bfX,\ \bfT\right)\mathcal{H}_{\alpha}>0\text{ a.s. on }A, 
	\end{align*}
 where the last inequality follows from \eqref{eq:newControl}.
	Integrating with respect to $\bfX$ then gives $\mathbb{P}(T_{\alpha}>0)>0$, and $\mathbb{E}[T_{\alpha}]>0$. But according to Mecke formula
	\[
	\mathbb{E}\left[T_{\alpha}\right]=\mathbb{E}\left[\sum_{x\in \bfX\cap \B}\mathds{1}_{x\text{ is a trifurcation}}\right]=\int_{\B} dx\ \mathbb{P}^{x}\left(x\text{ is a trifucation}\right).
	\]
	But, stationarity gives
	\[
	\int_{\B} dx\ \mathbb{P}^{x}\left(x\text{ is a trifucation}\right)=|\B|\mathbb{P}^{0}\left(0\text{ is a trifucation}\right),
	\]
	hence the Lemma.
\end{proof}
We can now conclude.
\setcounter{thm}{1}

\begin{thm}
	There exists a unique infinite connected component in the supercritical phase, that is
	\[
	\mathbb{P}(\# C \leq 1)=1.
	\]
\end{thm}
\begin{proof}
	We work by contradiction by assuming that $\mathbb{P}(\# C=\infty)>0$. To simplify calculation, we work on the square box $B_\alpha$ for $\alpha>0$.
   Let $T_{\alpha}$ be the number of trifurcation in $B_\alpha$.
	According to the peeling strategy of Burton-Keane argument (which can be applied here for its nature is purely geometric. See \cite{BurtonKeane} or \cite[Section 8.2]{grimmett}.), there exists a subgraph $\mathbb{T}$ of $\mathcal{G}_{p,\lambda,r}\cap B_\alpha$ which is a forest whose branching points are trifurcations in $B_\alpha$. In particular, $\# \mathbb{T}\cap \partial B_\alpha\geq T_{\alpha}$ where $\# \mathbb{T}\cap \partial B_\alpha$ corresponds to the number of edges of the forest intersecting the boundary of the ball. But,
	\begin{equation*}
		\mathbb{E}\left[\# \mathbb{T}\cap \partial B_\alpha\right]\leq\mathbb{E}\left[\sum_{x\in \bfX\cap B_\alpha,\ y\in \bfX\cap B_\alpha^{c}}\mathds{1}_{x\sim y}\right]\leq2\mathbb{E}\left[\sum_{x\in \bfX\cap B_\alpha,\ y\in \bfX\cap B_\alpha^{c}}\mathds{1}_{\bfX\cap B^{+}_{x,y}=0}\right],
	\end{equation*}
	where $B^{+}_{x,y}$ is defined in the proof of Lemma \ref{lem:fullcoverage}. The last inequality above follows, once again, from the fact that one of the two half-ball of diameter $[x,y]$ being empty of point of $\bfX$ is a necessary condition for $x\sim y$.
	Now, using two times Mecke formula entails
	\begin{align*}
		\mathbb{E}\left[\sum_{x\in \bfX\cap B_\alpha,\ y\in \bfX\cap B_\alpha^{c}}\mathds{1}_{\bfX\cap B^{+}_{x,y}=0}\right]&=\int_{B_\alpha}dx\ \mathbb{E}^{x}\left[\sum_{ y\in \bfX\cap B_\alpha^{c}}\mathds{1}_{\bfX\cap B^{+}_{x,y}=0}\right]\\
		&=\int_{B_\alpha}dx \int_{B_\alpha^{c}}dy\ \mathbb{P}\left(\bfX\cap B^{+}_{x,y}=0\right)\\
		&=\int_{B_\alpha\times B_\alpha^{c}}e^{-\pi\|x-y\|^{2}/2}\ dx dy.
	\end{align*}
	Some computations now gives
	\begin{multline*}
		\int_{B_\alpha\times B_\alpha^{c}}e^{-\pi\|x-y\|^{2}/2}\ dx dy\leq 4\alpha \int_{[-\alpha/2,\alpha/2]\times [\alpha/2,\infty)}\exp(-\pi(x-y)^{2}/2)\  dxdy\\=4\alpha \int_{0}^{\alpha}\int_{x}^{\infty}e^{-\pi y^{2}/2}\ dy dx\leq \frac{8 \alpha}{\pi}.
	\end{multline*}
	Hence, we finally obtain in conjunction with Lemma \ref{lem:trifurcation} that,
	\[
	\alpha^{2}\mathbb{P}\left(0\text{ is a trifucation in }\mathcal{G}^{0}_{p,\lambda,r}\right)=\mathbb{E}\left[T_{\alpha}\right]\leq \mathbb{E}\left[\# \mathbb{T}\cap \partial B_\alpha\right]\leq \frac{16 \alpha}{\pi},
	\]
	which gives a contradiction.
	
\end{proof}

We end the section with the proof of Lemma \ref{prop:trace}. The proof start with a geometric result.
\begin{lem}
\label{lem:geom}
Let $\mathcal{C}$ be an open Euclidean ball. Let $x,y\in\mathcal{C}$, and $a,b\in \mathbb{R}\setminus\mathcal{C}$ such that $[A,B]\cap \mathcal{C}\neq \emptyset$. Assume that, $[A,B]$  separates $\mathcal{C}$ into two open subsets $\mathcal{C}_1$ and $\mathcal{C}_2$ such that $x\in\mathcal{C}_1$ and $y\in\mathcal{C}_2$. Then,  for any closed ball $\mathcal{B}$ such that $\{A,B\}\subset \partial \mathcal{B}$,  $x\in \mathcal{B}$ or $y\in \mathcal{B}$.
\end{lem}
\begin{proof}[Proof of Lemma \ref{lem:geom}]
Let $\mathcal{B}$ be any closed ball such that $\{A,B\}\subset \partial \mathcal{B}$. Let also be $C$ and $D$ such that $\{C,D\}=[A,B]\cap \partial\mathcal{C}$. Let $E$ be the center of $\mathcal{B}$. The perpendicular bisector of the segment $[C,D]$ separates the plane in two half-planes. We can assume without loss of generality that $E$ is in the same half-plane as $D$, so that the segment $[E,C]$ intersects the perpendicular bisector of $[C,D]$ at some point $H$ (see Figure \ref{fig:geom}).

\medskip 

Let some point $z$ such that $d(z,H)\leq d(H,C)$. Then, $d(z,E)\leq d(z,H)+d(H,E)\leq d(H,C)+d(H,E)$, but, because $H\in [E,C]$, $d(H,C)+d(H,E)=d(E,C)$. Hence, since $C\in \mathcal{B}$, we have $d(z,E)\leq d(E,C)\leq d(E,B)$. This says that the ball $\widetilde{\mathcal{B}}$ of radius $d(H,C)$ with center $H$ is included in $\mathcal{B}$.

\medskip

Now, since $\{A,B\}\subset \partial C\cap \partial \widetilde{\mathcal{B}}$, it follows that the center $O$ of $\mathcal{C}$ and $H$ both belongs to the perpendicular bisector of the segment $[C,D]$. Let $U$ and $V$ be the intersection points of the perpendicular bisector of $[C,D]$ with $\partial \mathcal{C}$ (see Figure \ref{fig:geom2}). If $H\notin [U,V]$ then either $V\in\widetilde{\mathcal{B}}$ or $U\in\widetilde{\mathcal{B}}$ because $(C+D)/2\in\widetilde{\mathcal{B}}\cap [U,V]$. If $H\in [U,V]$, then either $d(H,U)=d(O,U)-d(H,O)$ or $d(H,V)=d(O,V)-d(H,O)$. Assume the latter, then, since $d(O,C)\leq d(H,C)+d(H,O)$, we have
$d(H,V)\leq d(O,V)+d(H,C)-d(O,C)=d(H,C)$. Thus, $V\in \widetilde{\mathcal{B}}$. 

\medskip

It follows that either $V$ or $U$ belongs to $ \widetilde{\mathcal{B}}$. Because, two distinct circles have at most two intersection points, it follows that the corresponding circular arc ($(CVD)$ or $(CUD)$) is a subset of $\widetilde{\mathcal{B}}$. From this, it follows that either $x$ or $y$ lies in $\widetilde{\mathcal{B}}\subset B$. This ends the proof.
\end{proof}

\begin{figure}
\includegraphics[scale = 0.4, trim= 100 100 0 60, clip]{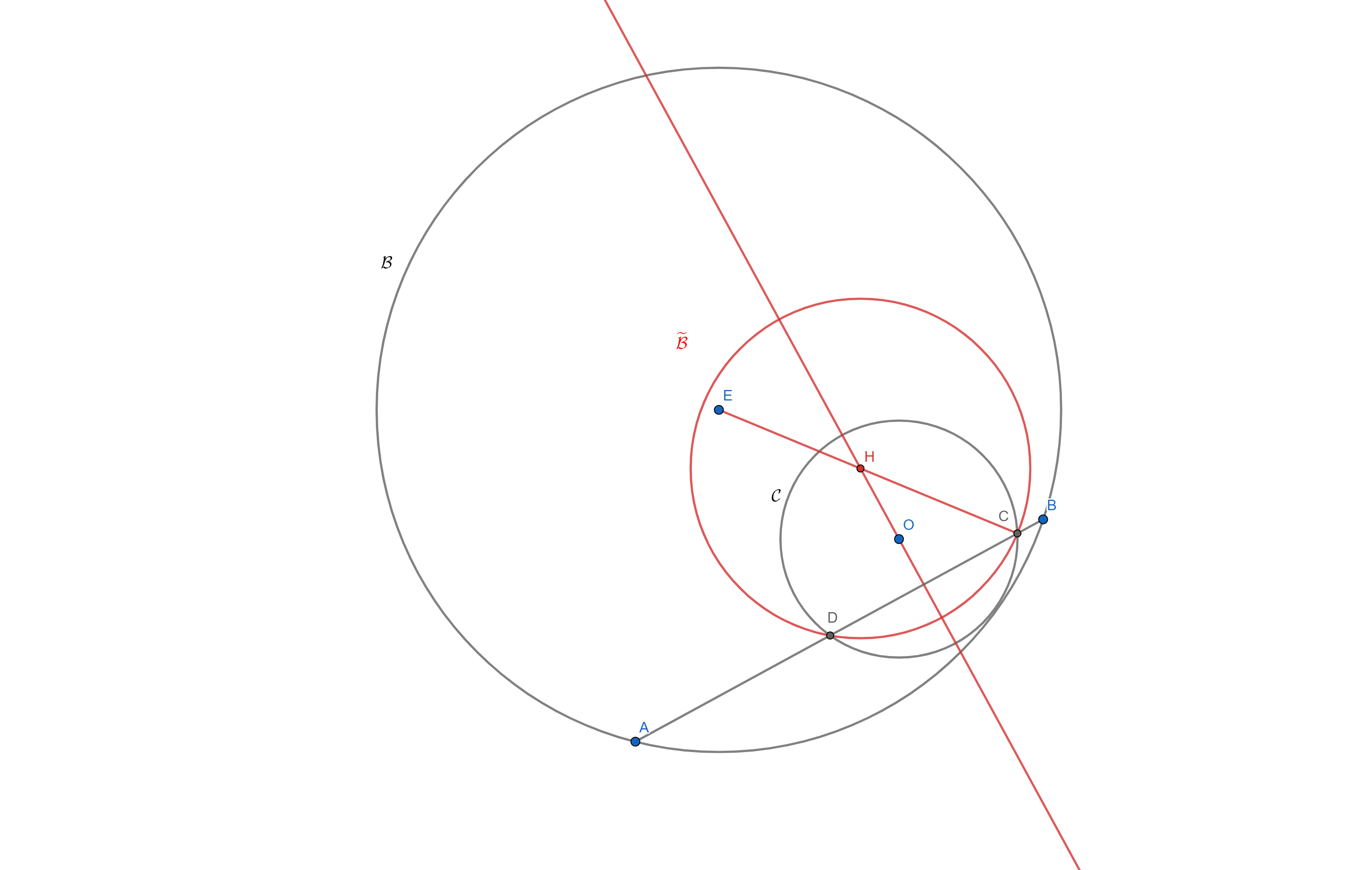}
\caption{Contruction of $\widetilde{\mathcal{B}}$}
	\label{fig:geom}
\end{figure}
\begin{figure}
	\includegraphics[scale = 0.4, trim= 100 300 400 60, clip]{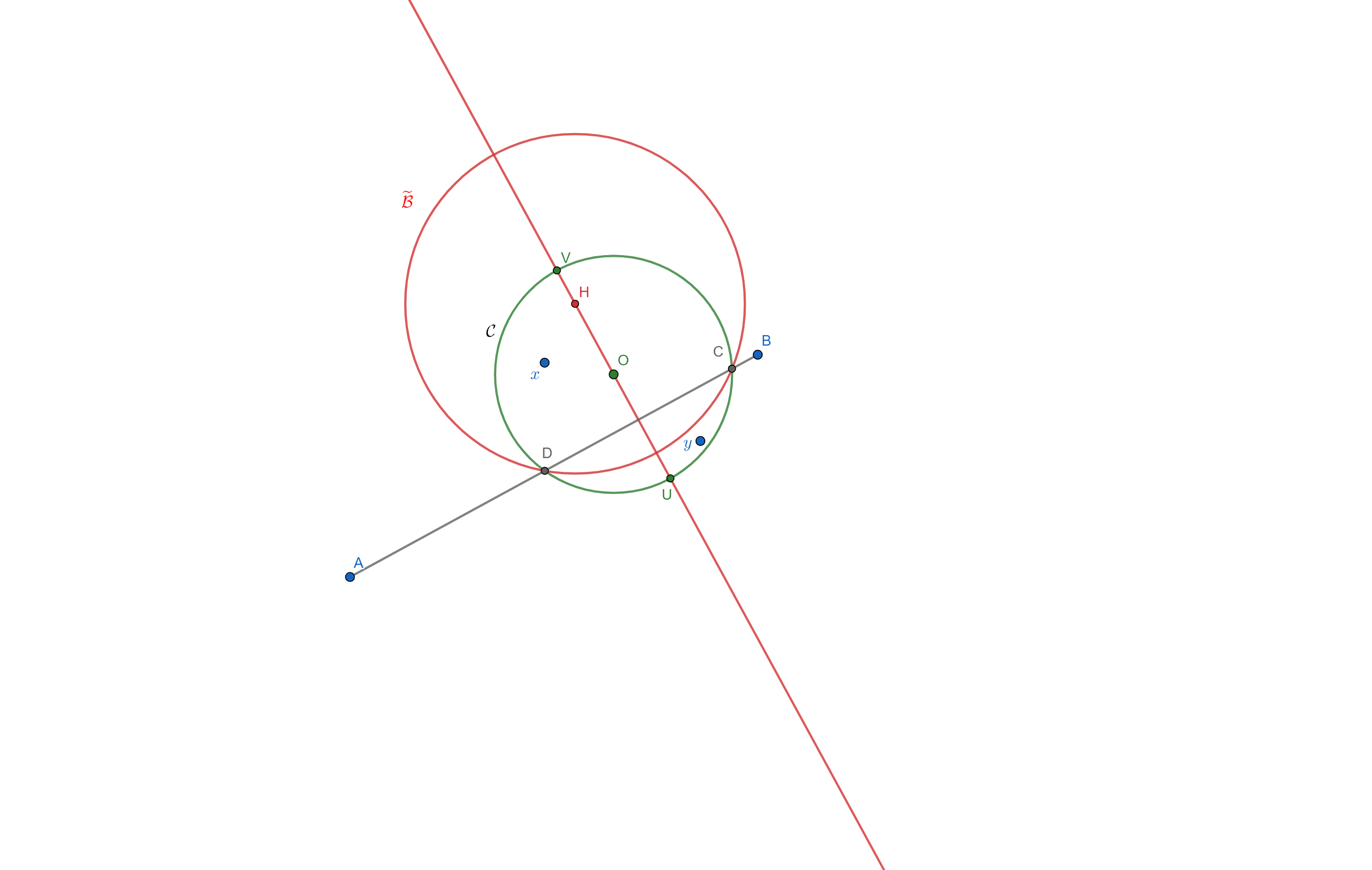}
    \caption{Proof that $\widetilde{\mathcal{B}}$ contains $x$ or $y$.}
	\label{fig:geom2}
\end{figure}
With this in hands, we can now prove Lemma \ref{prop:trace}.
\begin{proof}[Proof of Proposition \ref{prop:trace}]
We show that the subgraph induced by the points of $\gamma$ in $\mathcal{C}$ is connected. By construction, this would directly imply that the trace is connected.
	The proof is based on Bowyer–Watson algorithm \cite{bowyer,watson}, which gives an iterative construction of the Delaunay triangulation. Thus, we need to recall some facts on this algorithm. It is based on the sequential construction of the triangulation. Let $\nu$ a finite subset of $\gamma$ and $x\in \nu$. Assume that the triangulation is already constructed for $\gamma\setminus\nu$. Then, adding point $x$ to $\bfT_{\gamma\setminus \nu}$ proceeds as follows (Note that the triangulation $\bfT_{\gamma\setminus \nu}$ already covers $\mathbb{R}^{2}$ with triangles).\\
	
	\bigskip
	
	{\bf Add point algorithm}
	\begin{enumerate}
		\item For each triangle $T$ in $\bfT_{\gamma\setminus \nu}$ check if the circumcircle of $T$ contains $x$. If it does:
		\begin{itemize}
			\item Then, for each edge $e$ of $T$, check if the neighboring triangle of $T$, having also edge $e$, has $x$ in its circumcircle.
			\item If so, remove $e$ from the triangulation.
		\end{itemize} 
		\item Step 1. creates, in the triangulation, a polygonal convex hole containing $x$.
		\item For each vertex $y$ of the polygonal boundary of the convex hole, add the edge $[y,x]$ to the triangulation. 
	\end{enumerate}
	
	With this tool in hand, we can now prove the lemma. The proof works by induction on the cardinal of $\gamma\cap\mathcal{C}$ which is finite by the locally finite hypothesis.
	
	\medskip
	
	\noindent
	{\bf Case $\#\gamma\cap \mathcal{C}=2$:}\\
	Assume that $\gamma\cap \mathcal{C}=\{x,y\}$.
	To prove the lemma in the case $\#\gamma\cap \mathcal{C}=2$, we only need to construct a circle passing through $x$ and $y$ containing no point of $\gamma$.
	 We build a circle passing through $x$ and $y$ and included in $\mathcal{C}$ (so because $\#\gamma\cap \mathcal{C}=2$ there won't be any other point of $\gamma$ inside the circle, and $x$ and $y$ would be connected in the triangulation). To do so, take the perpendicular bisector of $[x,y]$: this separates $\mathbb{R}^{2}$ into two half-planes containing either $x$ or $y$. Assume that the center $\Omega$ of $\mathcal{C}$ is in the half-plane of $y$. Then, the segment $[\Omega,x]$ intersects the bisector of $[x,y]$ at some point (say $F$). Then, it can be shown that the circle of center $F$ and passing through $x$ and $y$ is included in $\mathcal{C}$. Hence, there exists an edge between $x$ and $y$ in $\bfT_{\gamma}$, so that the trace is connected (see Figure \ref{fig:inductionpart1}).
	\begin{figure}
 \includegraphics[scale = 0.4, trim= 100 300 400 60, clip]{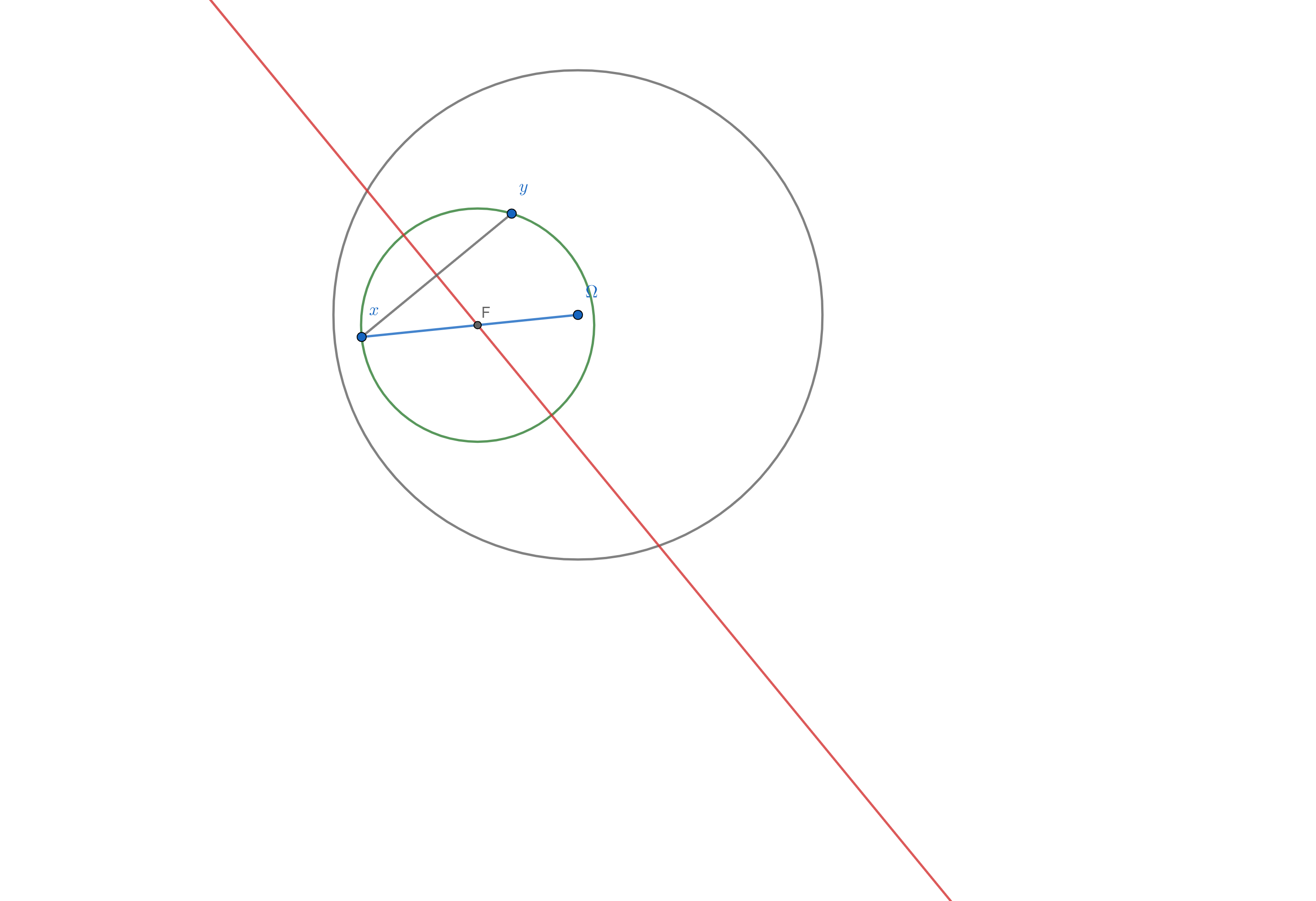}
 \caption{Construction of an inner circle}
 \label{fig:inductionpart1}
\end{figure}
	\bigskip
	
	\noindent
	{\bf Induction: }\\
	Assume that the result is true for any $\gamma$ such that $\#\gamma\cap \mathcal{C}=n$ (for $n\geq 2$). Now take $x\in \mathcal{C}\setminus \gamma$ and let us show that the trace Delaunay triangulation of $\gamma\cup\{x\}$ is connected.  
	By contradiction, we make the hypothesis that the trace of $\bfT_{\gamma\cup\{x\}}$ is not connected.
  \medskip

First, note that, according to Lemma \ref{lem:geom}, the polygonal convex hole (created by the algorithm) surrounding $x$ cannot contain edges which separate the ball $\mathcal{C}$ into two subsets containing points of $\gamma\cap\mathcal{C}$. Hence, the boundary of the hole contains at least one point, say $u$, of $\gamma\cap \mathcal{C}$. Thus, $x$ must be connected to $u$ at least. But, $u$ was connected to all other vertices in the trace of $\bfT_{\gamma}$ but not in the trace of $\bfT_{\gamma\cup \{x\}}$. Hence, there exists $a$ and $b$ in $\gamma$ such that $[a,b]\subset \bfT_{\gamma}$ but $a$ and $b$ are not connected after the addition of $x$.

So, all this implies that the edge $[a,b]$ has been removed by the {\bf ``add point algorithm''}, and so lied inside the convex hole created by the algorithm. This implies that $a$ and $b$ are both vertices of the polygon surrounding $x$. But if so, by construction, $x$ is connected to $a$ and $b$, so $a$ and $b$ must be connected. This gives a contradiction.\\ 
	By induction, the lemma follows.
% \begin{figure}
%	\centering
%	\scalebox{1.2}{
%		\hspace{-2cm}
%\begin{tikzpicture}
%	\node[anchor=south west,inner sep=0] at (0,0) {\includegraphics[width=\textwidth]{mainImage}};
%\end{tikzpicture}}
%\caption{Settings of the induction in the proof of Lemma \ref{lem:trace}}
%\label{fig:geometry}
%\end{figure}
 \end{proof}

\section{Discussion}
\label{sec:finaldiscussion}

In this final section, we discuss several hypothesis made to build our D2D model and evaluate the impact of relaxing them on the scale of results stated in Theorem~\ref{thm:mainResult}. \\

\noindent
\textbf{The distribution in the random connection rule.} A crucial difference between our model and that of~\cite{le2021continuum} is the choice of a \textit{random} connection rule, instead of a deterministic one. As explained in the introduction---see Section~\ref{sect:Modifs}---it was necessary to implement the strategy, outlined in Section~\ref{sect:Item4}, that we have adopted to exclude the possibility of a fat region~\Romanbar{2}. We have defined the range of a user as~$\frac{r'}{2} \bdcalE$, where~$r'$ is the mean diameter of the area that it covers. We have set~$r'=r$ for users on streets and~$r'=2r$ for the relays at crossroads, while the normalized range~$\bdcalE$ is an exponentially distributed random variable of mean one. It seems quite natural to consider replacing the latter by some other positive and unit mean random variable. We do not see any problem raised in such new setting, except, significantly, for Item~$4$. As the attentive reader may have noticed, doubling the connection radius of relays was indeed not innocuous. Recall that a key step in our reasoning consists of comparing the probabilistic cost to fill a hole in the sparse coverage of a street, either with the help of an additional user uniformly drawn on the latter, or by slightly increasing the ranges of those already present. The inequality~\eqref{eq:inequalityfunctionGH} formalizes it. It turns out that in the case~$\frb=2$, when the hole is delimited by the extremities of the intervals covered by the two relays standing around the street, the inequality collapses as their connection radius~$r'$ is strictly less than~$2r$. We observe a similar phenomenon in the case where normalized ranges all have a half-normal density~$\frac{2}{\pi r'} e^{-\frac{x^2}{\pi r'^2}} \mathds{1}_{x \geq 0}$: the domination~\eqref{eq:inequalityfunctionGH} only holds if the expected range of relays exceeds enough that of users on streets. On the contrary, as the density of~$\bdcalE$ is heavy-tailed---is for instance~$3 \left(1+x \right)^{-3} \mathds{1}_{x \geq 0}$---the inequality becomes systematically false, no matter how much larger is the connection radius of relays. Here, in contrast to the previous cases, we remark that by acting on the value of~$r'$, we do not change the heaviness of the range's tail distribution. It leads to identify the latter as a more critical feature to make the statement true than a gap between connection radii. More precisely, we can show that our proof of Item 4 still works as some asymptotic holds, involving the range's tail distribution of both kind of users. Let~$g$ be the density of the normalized range of relays and~$\tilde{g}:=g*g$ the convolution with itself. Let also~$\calH$ be some positive differentiable function, vanishing at infinity, satisfying~$\calH'=-H\left(\frac{\cdot}{2}\right)'$, where~$H$ is the tail distribution of the normalized range of users on streets. We claim that the inequality~\eqref{eq:inequalityfunctionGH} remains valid in the case~$\frb=2$ as
$$\calH * \tilde{g} \left(u\right) = \underset{u \to +\infty}{\calO}\left(u \times \tilde{g} \left(u\right) \right).$$
In particular, in the heavy-tailed situation described above, it implies that the range density associated to users on streets has to decrease quicker than~$x^{-4}$ for preserving the statement. The reader may now wonder if extra conditions to obey go with the two other cases~$\frb=0$ and~$\frb=1$. We do not believe so. Indeed, as the asymptotics~\eqref{eq:asymptoticGtwovicinityinfinity}, \eqref{eq:asymptoticHtwovicinityinfinity}, \eqref{eq:asymptoticGonevicinityinfinity} and~\eqref{eq:asymptoticHonevicinityinfinity} suggest, the hole to fill is typically far narrower, limited by the uniform drawing of the user on street that covers the area immediately bordering it. The role of ranges (and their tail distribution) then gets less decisive.       
\\

%as the densities~$f$ and~$g$ of ranges, respectively of relays and of users on streets, satisfy~$g=o\left(f \right)$ at infinity.   \\
%Indeed, for streets of length~$\ell >> 1$, the hole is typically large, of order~$\ell$. Since a user on the street covers an area whose width is exponentially distributed with mean~$r$, the cost to absorb the hole is for it roughly~$e^{-\frac{\ell}{r}}$. Rather doing it by just increasing the range of users at crossroads implies that the latter is in fact very small, not large. Or also that the sum of the ranges is about~$\ell$. The cost (density) of such configuration is of order~$e^{-\frac{2 \ell}{r'}}$. It is thus overwhelmed by the former as~$\ell \to +\infty$, what contradicts~\eqref{eq:inequalityfunctionGH}.   

\noindent
\textbf{The directional independence.} 
In our model, we associate to any relay at a crossroad as many independent (random) ranges as there are streets converging to it. Introduce more dependencies, or even a unique range in every direction like~\cite{le2021continuum}, can be viewed as more realistic. Again, it would not pose any difficulty to generalize the Items 1 to 3 of Theorem~\ref{thm:mainResult}. The independence assumption is mostly used to ease the proof of Item 4. More specifically to get the Russo's formulas---see Proposition~\ref{prop:russoformulas}, which usually hold under the condition of state independence of the pivotal elements, here the edges of the mosaic. Yet, we tend to consider it as nothing more than a technical hurdle to overcome, so that Item 4 should remain true even by relaxing the directional independence hypothesis. 
\\

\noindent
\textbf{The urban media.} As argued in Section~\ref{sect:Modifs}, changing the street system of~\cite{le2021continuum} was utterly intentional. Poisson-Delaunay triangulations benefit indeed from far better percolation properties than Poisson-Voronoi mosaics. Thanks to them, we have cleared the phase diagram characterizing the model of the blurred region~\Romanbar{1} and~\Romanbar{3}. See Figure~\ref{fig:DiagPhase} and~\ref{fig:NotreDiag}. There is then little doubt that nothing in our work could help to make equal progress in the original Poisson-Voronoi set-up. This times however, it seems different for the region~\Romanbar{2} and the Item 4 of Theorem~\ref{thm:mainResult}, which deals with it. In the proof of the latter, detailed in Section~\ref{sect:Item4}, we show that the comparison between the partial derivatives of the probability of a long-range connection event is derived from an analytical study of the model at the smallest possible scale, on one single edge of given length, without regard to the specific geometry of the urban media. We thus strongly believe that the statement could be extended to a large class of mosaics, at least those of finite intensity, that is containing almost surely a finite number of edges in any finite region. 

\bibliographystyle{plain}
\bibliography{biblio}
\end{document}

%% file: DCMCommands.tex
\newcommand{\calC}{\mathcal{C}}
\newcommand{\calD}{\mathcal{D}}

\newcommand{\calH}{\mathcal{H}}

\newcommand{\calO}{\mathcal{O}}

\newcommand{\calR}{\mathcal{R}}

\newcommand{\calU}{\mathcal{U}}

\newcommand{\frb}{\mathfrak{b}}

\newcommand{\frL}{\mathfrak{L}}

\newcommand{\frR}{\mathfrak{R}}

\newcommand{\bbE}{\mathbb{E}}

\newcommand{\bbP}{\mathbb{P}}

\newcommand{\bbR}{\mathbb{R}}

\newcommand{\bbZ}{\mathbb{Z}}

\newcommand{\bfp}{\mathbf{p}}

\newcommand{\bfT}{\mathbf{T}}

\newcommand{\bdcalG}{\boldsymbol{\mathcal{G}}}

\newcommand{\bfX}{\mathbf{X}}
\newcommand{\bfY}{\mathbf{Y}}

\newcommand{\bfN}{\mathbf{N}}

\newcommand{\bdcalE}{\boldsymbol{\mathcal{E}}}
\newcommand{\bdcalU}{\boldsymbol{\mathcal{U}}}
\newcommand{\bfU}{\mathbf{U}}
\newcommand{\bfI}{\mathbf{I}}
\newcommand{\bfS}{\mathbf{S}}
\newcommand{\bfR}{\mathbf{R}}
\newcommand{\bfs}{\mathbf{s}}
\newcommand{\bfi}{\mathbf{i}}